\documentclass[times,final,11pt]{elsarticle}
 
\usepackage{jcomp}

\usepackage{amssymb,amsmath,amsthm,bm}
\usepackage{latexsym}
\usepackage{mathtools}
 
\usepackage{cleveref}
\usepackage{subcaption} 

\usepackage{float}

\definecolor{newcolor}{rgb}{.8,.349,.1}

\makeatletter

\newcommand{\Rmnum}[1]{\uppercase\expandafter{\romannumeral #1}}
\makeatother

\newtheorem{theorem}{Theorem}

\usepackage{booktabs}  
\usepackage{threeparttable} 
\usepackage{multirow}   
\usepackage{epsfig}
\usepackage{epstopdf}
\usepackage{color}
\usepackage{colordvi}
\usepackage{appendix}

\numberwithin{equation}{section}
\numberwithin{figure}{section} 
\numberwithin{table}{section}
\usepackage{tikz}

\usepackage{soul} 
\usepackage{caption}
\usepackage{geometry}   
\usepackage{titlesec}   
\theoremstyle{definition}
\newtheorem{definition}{Definition}
\newtheorem{example}{Example}
\newtheorem{remark}{Remark}

\allowdisplaybreaks

\renewcommand\baselinestretch{1.2}

\begin{document}

\verso{Chuan Fan, Kailiang Wu}

\begin{frontmatter}

\title{{\bf High-Order Oscillation-Eliminating Hermite WENO Method for Hyperbolic Conservation Laws}\tnoteref{tnote1}}  
 
\author[1]{Chuan {Fan}}
\ead{fanc@sustech.edu.cn}
\author[1,2]{Kailiang {Wu}\corref{cor1}}
\cortext[cor1]{Corresponding author.}
\ead{wukl@sustech.edu.cn}

\address[1]{Department of Mathematics and Shenzhen International Center for Mathematics, Southern University of Science and Technology, Shenzhen 518055, China}
\address[2]{Guangdong Provincial Key Laboratory of Computational Science and Material Design, Shenzhen 518055, China} 

\begin{abstract}

This paper proposes high-order accurate, oscillation-eliminating, Hermite weighted essentially non-oscillatory (OE-HWENO) finite volume schemes for hyperbolic conservation laws,  
motivated by the oscillation-eliminating (OE) discontinuous Galerkin schemes recently proposed in [M.~Peng, Z.~Sun, and K.~Wu, {\em Math. Comp.}, 2024, {\tt doi.org/10.1090/mcom/3998}].  
 The OE-HWENO schemes incorporate an OE procedure after each Runge--Kutta stage, by dampening the first-order moments of the HWENO solution to suppress spurious oscillations without any problem-dependent parameter. The OE procedure acts as a moment filter and is derived from the solution operator of a novel damping equation, which is exactly solved without any discretization. Thanks to this distinctive feature, the OE-HWENO method remains stable with a normal CFL number, even for strong shocks resulting in highly stiff damping terms. To ensure the essentially non-oscillatory property of the OE-HWENO method across problems with varying scales and wave speeds, we design a scale-invariant and evolution-invariant damping equation and propose a generic dimensionless transformation for HWENO reconstruction. The OE-HWENO method offers several notable advantages over existing HWENO methods. First, the OE procedure is highly efficient and straightforward to implement, requiring only simple multiplication of first-order moments by a damping factor. Furthermore, we rigorously prove that the OE procedure maintains the high-order accuracy and local compactness of the original HWENO schemes and demonstrate that it does not compromise the spectral properties via the approximate dispersion relation for smooth solutions. Notably, the proposed OE procedure is non-intrusive, enabling seamless integration as an independent module into existing HWENO codes. Finally, we rigorously analyze the bound-preserving (BP) property of the OE-HWENO method using the optimal cell average decomposition approach [S. Cui, S. Ding, and K. Wu, {\em SIAM J. Numer. Anal.}, 62:775–810, 2024], which relaxes the theoretical BP constraint for time step-size and reduces the number of decomposition points, thereby further enhancing efficiency. Extensive benchmarks validate the accuracy, efficiency, high resolution, and robustness of the OE-HWENO method. 
\vspace{-0.0001mm}
\vspace{2mm}
\end{abstract}

\begin{keyword}	
\vspace{-11mm}
\end{keyword}

\end{frontmatter}


\renewcommand\baselinestretch{1.39}

\section{Introduction}

Hyperbolic conservation laws are a class of fundamental mathematical models used to depict the evolution of conservative variables in physical systems. These laws play a significant role in various science and engineering fields, such as fluid mechanics, gas dynamics, meteorology, oceanography, and hydrology. The equations governing hyperbolic conservation laws can be expressed as follows: 
\begin{equation}\label{sec:HCLs}
	\left\{
	\begin{aligned}
		&u_t+ \nabla \cdot \bm{f}(u)=0,~(\bm{x},t)\in \mathbb{R}^d\times \mathbb{R}^+,
		\\
		&u({\bm{x}},0)=u_0(\bm{x}),~~\;\bm{x} \in \mathbb{R}^d.
	\end{aligned}
	\right.
\end{equation}  
Solutions to these nonlinear hyperbolic conservation laws often lack regularity and may include discontinuities such as shocks, even if the initial conditions are smooth. High-order numerical methods employed to solve these problems frequently generate spurious non-physical oscillations near discontinuities, leading to numerical instability and potentially causing the solutions to blow up. Consequently, designing efficient and essentially non-oscillatory high-order numerical methods for solving hyperbolic conservation laws is critically important. 

Over the last few decades, the increasing demands for solving hyperbolic conservation laws and related equations have spurred vigorous developments and widespread applications of various high-order numerical methods. These methods include, but are not limited to, finite difference (FD) methods \cite{Harten,JS,Pirozzoli,WT}, finite volume (FV) methods \cite{HS,LOC,Shu1998}, and discontinuous Galerkin (DG) finite element methods \cite{CS2,DBTM,LLS2,LLS1,PSW}. Each of these methods offers distinct advantages, effectively addressing challenges posed by phenomena such as shock waves, contact discontinuities, and other intricate waves. Among these, the weighted essentially non-oscillatory (WENO) schemes represent a significant class of high-order numerical methods. These schemes were developed based on the earlier essentially non-oscillatory schemes \cite{Harten}. The third-order FV WENO scheme was first proposed by Liu, Osher, and Chan in 1994 \cite{LOC}. In 1996, Jiang and Shu introduced a general framework for constructing FD WENO schemes of arbitrary order by incorporating smooth indicators and nonlinear weights. Subsequently, in 1998, Shu designed a fifth-order FV WENO scheme \cite{Shu1998}. Since then, WENO schemes have gained popularity and have been further developed and refined in works such as \cite{BGS,BS,CCD,WC2001,ZQ_FDWENO,ZS_MRWENO}. The common attribute of these schemes is their ability to achieve high-order numerical accuracy in regions of smooth solutions while preserving favorable non-oscillatory properties near discontinuities. For more developments of WENO schemes, readers are referred to the recent review \cite{Shu2020} and the references therein.

The Hermite WENO (HWENO) method is an efficient variant of the WENO method based on the Hermite interpolations \cite{QS1}. This approach utilizes two pieces of information at each point (both the point value and its first derivative) or within each cell (the cell average and its first-order moment), rendering the stencil of HWENO schemes more compact than that of the standard WENO schemes. However, the derivatives or first-order moments may be quite large near discontinuities, which can easily produce numerical oscillations and make the HWENO schemes less robust than the standard WENO schemes.  Qiu and Shu \cite{QS1} first proposed one-dimensional (1D) FV HWENO schemes by evolving the equations \eqref{sec:HCLs} and the first-order derivative equations simultaneously and designed the HWENO limiter for DG methods. However, applying this approach directly to two-dimensional (2D) FV HWENO schemes as in \cite{QS2} did not effectively control spurious oscillations near strong shocks. To overcome this, Zhu and Qiu \cite{ZQS2008} developed an alternative 2D HWENO scheme using a set of different stencils to approximate the first-order derivatives. 
Subsequently, various HWENO schemes have been developed by employing different stencils or techniques for the spatial discretization of the governing equations and their first-order moments (or derivatives) \cite{Cap,LSQ1,LQ1,TLQ,ZA,ZCQ}. For example, Zhao, Chen, and Qiu proposed a hybrid HWENO scheme \cite{ZCQ} that uses a technique to modify the first-order moments, akin to an HWENO limiter, to manage their magnitudes and suppress spurious oscillations near discontinuities. 
These moment-limiting or derivative-limiting techniques have enhanced the robustness of HWENO schemes and were widely used in \cite{FZQ,FZXQ,LSQ2,WYE,ZZ,ZQ1,ZZQ}. However, the resulting moment-based HWENO schemes \cite{FZQ,LSQ2,ZCQ,ZQ1} typically achieve a maximum of fifth-order accuracy due to the application of moment-limiting approach. 
Recently, inspired by the oscillation-free DG methods \cite{LLS1}, Zhao and Qiu proposed a sixth-order oscillation-free HWENO scheme \cite{ZQ2} that modifies the first-order moment equations by incorporating damping terms to mitigate spurious oscillations.  
It is worth mentioning that these damping terms in \cite{ZQ2} include empirical parameters that heavily depend on numerical experience in simulations; inappropriate parameter choices can significantly affect the stability and performance of the resulting schemes. Additionally, the damping terms become highly stiff when simulating strong discontinuities and/or large-scale problems, leading to very stringent restrictions on the time step-size. The (modified) exponential Runge--Kutta (RK) time discretization \cite{HS_ERK} is often required to alleviate this issue.

The aim of this paper is to develop new robust, moment-based, high-order HWENO schemes, exemplified by a sixth-order version, for solving hyperbolic conservation laws. Building on the recent advancements in filter-based oscillation-eliminating DG (OEDG) approach \cite{PSW}, we propose novel oscillation-eliminating HWENO (OE-HWENO) method. This method effectively suppresses spurious oscillations across a wide range of scales and wave speeds, without relying on any problem-specific parameters across all cases tested. The OE-HWENO schemes exhibit several distinctive features: they maintain an essentially non-oscillatory behavior in the presence of discontinuities, offer scale invariance to accommodate multi-scale problems, and ensure evolution invariance across varying wave speeds. Additionally, they are rigorously proven to be bound-preserving (BP) through optimal convex decomposition \cite{CDW1,CDW2} under a suitable time step constraint, provided that the HWENO reconstructed values satisfy the desired bound constraints (see Theorem \ref{sec2:thm_BP2d} for details).   
Furthermore, the proposed OE-HWENO schemes retain many advantageous features of traditional HWENO schemes, including compact reconstructed stencils, the flexibility to use arbitrary linear weights, and high resolution for capturing discontinuities. The specific efforts and innovations of this work are detailed as follows:
\begin{itemize}
	\item A prevalent issue in moment-based HWENO schemes is the uncontrolled large variations of the first-order moments near discontinuities, which can easily lead to spurious oscillations and numerical instability. To address this challenge, we introduce an oscillation-eliminating (OE) procedure after each Runge--Kutta (RK) stage to control the magnitudes of the moments. The proposed OE technique acts as a filter on the first-order moments, based on the solution operator of the novel damping equations and theoretically maintaining the original high-order accuracy (see Theorems \ref{thm:accuracy} and \ref{thm:accuracy2d}), significantly differs from the moment-limiting techniques \cite{LSQ2,ZCQ,ZQ1}, which retained at most fifth-order accuracy for originally sixth-order HWENO schemes.

	\item Thanks to the linearity of our damping equations, they are exactly solvable without any discretization (Theorems \ref{thm:OE} and \ref{thm:2dOE}).  Consequently, the implementation of the OE procedure is straightforward and highly efficient, as it involves only the multiplication of first-order moments by a damping factor. This exact solver ensures the stability of OE-HWENO method when coupled with standard explicit RK time discretization with a normal CFL number, even in the presence of highly stiff damping terms associated with strong shocks. Unlike the damping-based oscillation-free HWENO method \cite{ZQ2}, our approach does not require empirical problem-dependent parameters or (modified) exponential time discretization.

	\item 
	To maintain the non-oscillatory behavior of the OE-HWENO schemes for problems spanning various scales and wave speeds, we propose a damping operator devoid of problem-dependent parameters, ensuring the scale and evolution invariance of the damping strength (Theorems \ref{sec2:thm_F} and \ref{thm:EI}). Additionally, we introduce a generic dimensionless transformation to achieve scale-invariant high-order accurate HWENO reconstruction for spatial discretization (Theorem \ref{sec2:thm_si}).

	\item We present a rigorous BP analysis (Theorem \ref{sec2:thm_BP2d}) of the OE-HWENO method, based on the optimal cell average decomposition (OCAD) approach in \cite{CDW1,CDW2}. Compared to the classic decomposition \cite{ZS1, ZS2}, the OCAD requires fewer internal points, and moreover, it allows us to establish the BP property of the OE-HWENO schemes under the mildest theoretical time step constraint. This further enhances the computational efficiency of our method.

	\item We prove that the OE procedure retains the original high-order accuracy of the HWENO schemes in theory (Theorems \ref{thm:accuracy} and \ref{thm:accuracy2d}). Furthermore, using the approximate dispersion relations (ADR) \cite{Pirozzoli}, we analyze the dispersion and dissipation properties of the proposed 1D OE-HWENO method (Section \ref{sec:ADR}), illustrating that the OE procedure does not affect the spectral properties of the original HWENO schemes for smooth solutions.
\end{itemize}

This paper is organized as follows: Section \ref{sec2} proposes the OE-HWENO method. Section \ref{sec3} conducts extensive benchmarks to illustrate the accuracy, high resolution, and robustness of the OE-HWENO method. Section \ref{sec4} concludes the paper. 

\section{Numerical schemes}\label{sec2} 

This section is organized into four subsections. Section \ref{framework} introduces the basic framework of the OE-HWENO method. Section \ref{sec:1D} presents the 1D OE-HWENO method, covering its computational details, accuracy analysis, scale invariance, evolution invariance, and the ADR analysis. Section \ref{sec:2D} details the 2D extension of the OE-HWENO method. Section \ref{sec:BP} analyzes the BP property of the OE-HWENO method via optimal convex decomposition.

\subsection{Framework of OE-HWENO approach}\label{framework} 

This subsection presents an outline of the OE-HWENO method for a general $d$-dimensional system of hyperbolic conservation laws: 
\begin{equation}\label{sec1:HCLs}
	\left\{
	\begin{aligned}
		&\bm{u}_t+ \nabla \cdot \bm{f}(\bm{u})=0,~(\bm{x}, t)\in \Omega\times [0,T],
		\\&
		{\bm{u}({\bm{x}},0)=\bm{u}_0(\bm{x})},~~\;\bm{x} \in \Omega.
	\end{aligned}
	\right.
\end{equation}   
Here, $\Omega\subset\mathbb{R}^{d}$ represents the (bounded) spatial domain, $\bm{x}$ denotes the spatial variables, $t$ is the time variable, $\bm{u}\in\mathbb{R}^N$  are the conservative variables, and $\bm{f}(\bm{u})\in\mathbb{R}^{N \times d}$ represents the flux functions corresponding to each spatial direction. Let $\mathcal{T}_h$ be a partition of the domain $\Omega$. To derive the weak form of equation \eqref{sec1:HCLs}, we multiply it by a test function  ${\phi}(\bm{x})$ and integrate over each cell 
$K\in\mathcal{T}_h$.  After applying the divergence theorem, we obtain
\begin{equation}\label{sec_weakform}
	\int_{K} \bm{u}_t {\phi}(\bm{x}) \mathrm{d} \bm{x} - \int_{K} \bm{f} \nabla {\phi}(\bm{x})\mathrm{d}\bm{x}+
	\int_{\partial K} (\bm{f} \bm{n}){\phi}(\bm{x}) \mathrm{d}S=0\quad \forall{\phi}(\bm{x})\in\mathbb{P}^{1}{(K)},
\end{equation}
where $\bm{n}\in\mathbb{R}^{d}$ is the unit outward normal of the boundary $\partial K$.
To achieve $(k+1)$-th order accuracy in space, we seek a piecewise polynomial vector function $\bm{u}_h(\bm{x},t)\in \mathbb{V}^k_h$ to approximate the exact solution $\bm{u}(\bm{x},t)$ for any fixed $t$, where
\begin{equation*}
	\mathbb{V}^k_h:=\big\{\bm{u}=(u_1,\cdots,u_N)^{\top}: u_\ell\big|_{K} \in \mathbb{P}^k(K),\ 1\le \ell\le N,\ \forall K\in\mathcal{T}_h \big\},
\end{equation*}
where $\mathbb{P}^k(K)$ is the space of polynomials of total degree up to $k$ in cell $K$.
For moment-based HWENO schemes, the test function ${\phi}(\bm{x})$ is taken from $\mathbb{P}^1(K)$ to derive the discrete evolution equations for the zeroth-order (cell-average) and first-order moments. 
It follows from  \eqref{sec_weakform} that
\begin{equation}\label{sec:semi}
	\frac{\mathrm{d}}{\mathrm{d}t} \int_{K}\bm{u}(\bm{x},t) {\phi}^{(\ell)}_K(\bm{x}) \mathrm{d}\bm{x}=  \int_{K} \bm{f} \nabla {\phi}^{(\ell)}_K(\bm{x})\mathrm{d}\bm{x}-
	\int_{\partial K} (\bm{f} \bm{n}){\phi}^{(\ell)}_K(\bm{x}) \mathrm{d}S,\quad 0 \le \ell \le d,
\end{equation}  
where $\{{\phi}^{(\ell)}_K(\bm{x})\}_{\ell=0}^{d}$ represent a basis of $\mathbb{P}^1(K)$ with ${\phi}^{(0)}_K(\bm{x}) \equiv \frac{1}{|K|}$. Let $\{\bar{\bm{u}}^{(\ell)}_K(t)\}^d_{\ell=0}$ denote the approximations to the zeroth-order ($\ell=0$) and first-order ($1 \le \ell \le d$) moments, namely,
\begin{equation*}
	\bar{\bm{u}}^{(\ell)}_K(t) \approx \int_{K}\bm{u}(\bm{x},t) {\phi}^{(\ell)}_K(\bm{x}) \mathrm{d}\bm{x}, \quad 0 \le \ell \le d.
\end{equation*}
Let ${\mathcal{L}}^{(\ell)}_{K}({\bm u}_h(\bm{x},t))$ denote the numerical approximation to $\int_{K}\bm{f}\nabla {\phi}^{(\ell)}_K(\bm{x})\mathrm{d}\bm{x} - \int_{\partial K} (\bm{f} \bm{n})\phi^{(\ell)}_K(\bm{x}) \mathrm{d}S$ with suitable quadrature rules and numerical flux on $\partial K$, and let $\bm{u}_h(\bm{x},t)$ be the approximate piecewise polynomial solution obtained by the HWENO reconstruction from $\{{\bm{U}}_K\}$. 
Then the semi-discrete FV HWENO schemes can be expressed as 
\begin{equation}\label{sec_semidiscrete}
	\frac{\mathrm{d}}{\mathrm{d}t}{{\bm{U}}_K(t)}= 
	\bm{\mathcal{L}}_{K}(\bm{u}_h(\bm{x},t)) \quad \mbox{with} 
	\quad \bm{\mathcal{L}}_{K}:=(\mathcal{L}^{(0)}_{K},\cdots,\mathcal{L}^{(d)}_{K}),
\end{equation}
where ${\bm{U}}_K=(\bar{\bm{u}}^{(0)}_K,\cdots,\bar{\bm{u}}^{(d)}_K)$.  
The detailed expressions for equations \eqref{sec:semi} and \eqref{sec_semidiscrete} will be presented in Section \ref{sec:1D} for the 1D case and in Section \ref{sec:2D} for the 2D case. In fact, the computations of $\bm{\mathcal{L}}_{K}(\bm{u}_h(\bm{x},t))$ involve only the values of $\bm{u}_h(\bm{x},t)$, rather than the full polynomials. For hyperbolic systems, the HWENO reconstruction at certain quadrature points should be performed in characteristic variables \cite{ZCQ,ZQ1,ZQ2}. This enhances the essentially non-oscillatory property, similar to the classical WENO schemes  in \cite{JS}. When characteristic decomposition is applied, we reconstruct only the point values rather than the full polynomials; however, we continue to use $\bm{u}_h(\bm{x},t)$ broadly to represent the HWENO reconstructed values at the quadrature points. 

Typically, the semi-discrete HWENO scheme \eqref{sec_semidiscrete} is further discretized in time using an explicit $r$th-order $s$-stage RK method:
\begin{equation*}
	\begin{aligned} 
		1.\quad &\text{Set}~~{\bm{U}}^{n,0}_K={\bm{U}}^n_K,\\
		2.\quad &\text{For}~~\ell=0,\ldots,s-1,~\text{compute the intermediate values:}\\
		~~\quad&\quad\bm{u}^{n,\ell}_h = {\mathcal H} \{  {\bm{U}}^{n,\ell}_K \},\\
		~~\quad&~~~{\bm{U}}^{n,\ell+1}_K=\sum\limits_{0\le m \le \ell}c_{\ell m}[
		{\bm{U}}^{n,m}_K + {\Delta t} d_{\ell m} \bm{\mathcal{L}}_{K}(\bm{u}_h^{n,m})  
		],\quad  \\
		3.\quad &\text{Set}~~{\bm{U}}^{n+1}_K={\bm{U}}^{n,s}_K,
	\end{aligned}
\end{equation*} 
where $\bm{u}_h^{n,m}$ is the HWENO reconstructed piecewise polynomial solution at the $m$-th stage of the RK method, the operator $\mathcal{H}$ is the standard HWENO reconstruction based on the zeroth-order and first-order moments in the cell $K$ and its adjacent cells, ${\Delta t}$ denotes the time step-size, and $\sum_{0\le m\le \ell}c_{\ell m}=1$.
The resulting HWENO scheme works well for smooth solutions. However, for strong discontinuities, it may generate spurious oscillations and suffer from nonlinear instability. This is due to the uncontrolled large variations of the first-order moments near discontinuities.

To address this critical issue, we introduce an OE procedure after each RK stage to control the magnitudes of the first-order moments, resulting in the novel OE-HWENO scheme:
\begin{subequations}
	\begin{align}
		1.\quad &\text{Set}~~{\bm{U}}^{\sigma,0}_K={\bm{U}}^n_K,\notag\\
		2.\quad &\text{For}~~\ell=0,\ldots,s-1,~\text{compute the intermediate values:}\notag\\
		~~\quad&\quad\quad\bm{u}^{\sigma,\ell}_h = {\mathcal H} \{  {\bm{U}}^{\sigma,\ell}_K \},\notag\\
		~~\quad&\quad\quad {\bm{U}}^{n,\ell+1}_K=\sum\limits_{0\le m\le\ell}c_{\ell m}\big[
		{\bm{U}}^{\sigma,m}_K + {\Delta t} d_{\ell m} \bm{\mathcal{L}}_K(\bm{u}_h^{\sigma,m})
		\big], \notag\\   
			~~\quad&\quad\quad {\bm{U}}^{\sigma,\ell+1}_K = {\mathcal F}_{\rm OE} \left \{  {\bm{U}}^{n,\ell+1}_{K'} \right \}_{K' \in \Lambda (K) }, \notag \\
		3.\quad&\text{Set}~~{\bm{U}}^{n+1}_K={\bm{U}}^{\sigma,s}_K, \notag
	\end{align}
\end{subequations}  
where $\Lambda (K) $ denotes the local stencil of the high-order linear Hermite reconstruction for cell $K$, and  
$\mathcal{F}_{\rm OE}$ denotes the proposed OE procedure.  

The OE procedure ${\bm{U}}^{\sigma,\ell+1}_K = {\mathcal F}_{\rm OE} \left \{  {\bm{U}}^{n,\ell+1}_{K'} \right \}_{K' \in \Lambda (K) }$ is founded upon the solution operator of a novel damping equation, which is exactly solvable without any discretization. 
A similar decoupled OE procedure, based on the technique of using damping equations to suppress spurious oscillations, was proposed in \cite{PSW} and has been applied to solve compressible Euler, two-phase flow, and MHD equations \cite{PSW, YAW2024, LW2024MHD, ZW2024MHD}. This significantly distinguishes it from the moment-limiting techniques utilized in existing HWENO schemes. 
More precisely, the modified moments  ${\bm{U}}^{\sigma,\ell+1}_K$ are defined as the zeroth- and first-order moments of $\bm{u}_\sigma(\bm{x},{\Delta t})$, where $\bm{u}_\sigma(\bm{x},\tau) \in \mathbb V_h^k$ ($0 \le \tau \le \Delta t$) represents the solution to the following damping equations:
\begin{equation}\label{sec:DampingEq}
	\left\{
	\begin{aligned}
		&\frac{\mathrm{d} }{\mathrm{d} \tau} \int_{K} \bm{u}_\sigma \phi \mathrm{d}\bm{x} + \sigma_{K}(\bm{u}_h^*) \int_{K}(\bm{u}_\sigma-P^{0}\bm{u}_\sigma)  \phi \mathrm{d}\bm{x}= 0 \quad \forall \phi \in \mathbb{P}^1(K),
		\\
		&\bm{u}_\sigma(\bm{x},0) = \bm{u}_h^{*} (\bm{x}) = {\Pi_h}  \left \{  {\bm{U}}^{n,\ell+1}_{K} \right \}, 
	\end{aligned}
	\right. 
\end{equation}
Here, $\tau$ represents a pseudo-time distinct from $t$, and the operator ${\Pi_h}$ denotes a simple linear Hermite reconstruction. 
The operator $P^0$ represents the standard $L^2$ projection into $\mathbb{V}^0_h$. Specifically,   $P^0\bm{u}\big|_{K}=\bar{\bm{u}}_K^{(0)}=:\bar{\bm{u}}_K$ denotes the cell average of $\bm{u}$ over $K$.  
In the damping equations \eqref{sec:DampingEq}, $\sigma_{K}(\bm{u}_h^*)$ represents the damping coefficient. It should be carefully chosen to be small in smooth regions and large near discontinuities, as defined in \eqref{sec3:sigma_1d} and \eqref{sec3:sigma_2d} for the 1D and 2D cases, respectively.   
Notably, the damping coefficient $\sigma_{K}(\bm{u}_h^*)$ in \eqref{sec:DampingEq} only depends on the ``initial" solution $\bm{u}_\sigma(\bm{x},0)= \bm{u}_h^*(\bm{x})$. Consequently, the damping ODE system \eqref{sec:DampingEq} is linear, and its exact solution can be explicitly formulated without any time discretization, facilitating easy implementation of the OE procedure with very low computational cost. 
 Although the operator ${\Pi_h}$ is involved in the initial value of \eqref{sec:DampingEq}, the final expression of the OE procedure does not require the implementation of ${\Pi_h}$. Thanks to this remarkable feature and the exact solver of \eqref{sec:DampingEq}, the OE procedure is highly efficient and easy to implement. 
 The simple expression of the OE procedure 
  will be detailed in subsections \ref{sec:1D} and \ref{sec:2D} for the 1D and 2D cases, respectively.

\subsection{One-dimensional OE-HWENO method}\label{sec:1D}	

Consider the 1D scalar conservation law 
\begin{equation}\label{sec2:1dHCLS}
	\begin{cases}
		u_t + f(u)_x=0,~(x, t)\in \Omega\times [0,T],
		\\
		u_0(x)=u(x,0),~x\in \Omega.
	\end{cases}
\end{equation}
A uniform  partition of computational domain $\Omega=[a,b]$ is defined as $\Omega=\cup^{N_x}_{i=1}I_i$ 
with $I_i=[x_{i-\frac12},x_{i+\frac12}]$ and $h_x=x_{i+\frac12}-x_{i-\frac12}$ denoting the constant spatial step-size. 
Define $x_i=\frac12(x_{i-\frac12}+x_{i+\frac12})$ as the center of cell $I_i$. 
Following \eqref{sec:semi} with $\phi^{(0)}_i(x)=\frac{1}{h_x}$ and $\phi^{(1)}_i(x)=\frac{x-x_i}{h^2_x}$, we obtain
\begin{equation}\label{sec2:WeakForm1d}
	\left \{
	\begin{aligned} 
        & 	\frac{{\rm d} }{{\rm d} t} \bigg(\int_{I_i}u(x,t) \frac{1}{h_x}\mathrm{d}x\bigg)
		=- \frac 1 {h_x}\;\; \bigg( {f\big(u(x_{i+\frac12},t)\big)-f\big(u(x_{i-\frac12},t)} \big) \bigg),\\ 
		&\frac{{\rm d} }{{\rm d} t}  \bigg(\int_{I_i} u(x,t) \frac{x-x_i}{h^2_x}\mathrm{d}x\bigg)
		=- \frac 1 {2h_x} \bigg( f\big(u(x_{i-\frac12},t)\big) +f\big(u(x_{i+\frac12},t)\big)\bigg) +\frac 1 {h_x^2} \int_{I_i}f(u)\mathrm{d}x.
	\end{aligned}
	\right.
\end{equation} 
Let $\bar{u}_i(t)$ and $\bar{v}_i(t)$ denote the approximations to the zeroth- and first-order moments of $u(x,t)$ in $I_i$, respectively, that is,  
\begin{equation*} 
		\bar{u}_i(t) \approx\frac{1}{h_x}\int_{I_i}u(x,t)\mathrm{d}x,\quad \bar{v}_i(t) \approx \frac{1}{h_x} \int_{I_i} u(x,t) \frac{x-x_i}{h_x}\mathrm{d}x.
\end{equation*} 
Based on $\{\bar{u}_i, \bar{v}_i\}$, the $(k + 1)$th-order HWENO method constructs a piecewise polynomial solution $u_h(x,t)$ of degree $k$ to approximate the exact solution. By approximating the flux $ f\big(u(x_{i+\frac12},t)$ with an appropriate numerical flux and evaluating the integral  $\int_{I_i}f(u)\mathrm{d}x$ with an $L$-point Gauss--Lobatto quadrature (where $L = \lceil \frac{k+3}2 \rceil $), we obtain the following semi-discrete FV HWENO scheme:  
\begin{equation}\label{sec2:1dSemi_HWENO}
	\left\{
	\begin{aligned}
		&\frac{{\rm d} \bar{u}_i(t)}{{\rm d} t}=-\frac{1}{h_x}\;(\hat{f}_{i+\frac12}-\hat{f}_{i-\frac12}) =: \mathcal{L}_i^{(0)}(u_h(x,t)),
		\\
		&\frac{{\rm d} \bar{v}_i(t)}{{\rm d} t}=-\frac{1}{2h_x}(\hat{f}_{i-\frac12}+\hat{f}_{i+\frac12})+\frac{1}{h_x} \sum\limits_{\ell=1}^L {\omega}^{\rm GL}_\ell f(u_h({x}^{\rm GL}_{i,\ell},t)) =: \mathcal{L}_i^{(1)}(u_h(x,t)),
	\end{aligned}
	\right.
\end{equation} 
which corresponds to \eqref{sec_semidiscrete} in the 1D scalar case. 
Here, $\{x^{\rm GL}_{i,\ell}\}^L_{\ell = 1}$ represent the Gauss--Lobatto quadrature nodes in $I_i$, with the normalized weights $\{\omega^{\rm GL}_{\ell}\}^L_{\ell = 1}$ satisfying $\sum_{\ell = 1}^{L}\omega^{\rm GL}_\ell = 1$. For example, in the sixth-order OE-HWENO scheme, we use a four-point Gauss--Lobatto quadrature with 
\begin{align*}
	& {x}^{\rm GL}_{i,1}=x_{i-\frac12},\quad {x}^{\rm GL}_{i,2}=x_{i-\frac{\sqrt{5}}{10}},\quad {x}^{\rm GL}_{i,3}=x_{i+\frac{\sqrt{5}}{10}},\quad {x}^{\rm GL}_{i,4}=x_{i+\frac12},
	\\
	&{\omega}^{\rm GL}_{1}={\omega}^{\rm GL}_{4}=\frac{1}{12}, \quad {\omega}^{\rm GL}_{2}={\omega}^{\rm GL}_{3}=\frac{5}{12}.
\end{align*} 
For the numerical flux $\hat{f}_{i + \frac{1}{2}}$, we opt for the simple Lax--Friedrichs flux: 
 \begin{equation}\label{sec2:Def_LF1d}
 	\hat{f}_{i+\frac12}=\frac12\left( f(u_{i+\frac12}^-) + f(u_{i+\frac12}^+)  - {\alpha} ( u_{i+\frac12}^+  - u_{i+\frac12}^- ) \right),
 \end{equation}
 where $u_{i+\frac12}^{-}$ and $u_{i+\frac12}^{+}$ represent the left-hand and right-hand limits of $u_h(x,t)$ at $x=x_{i+\frac12}$, respectively, and $\alpha=\max_{1\le i\le N_x}|f'(\bar{u}_i)|$. Other suitable numerical fluxes, such as Godunov, HLL, or HLLC fluxes, can also be used.
  
Let $\bm{U}_i(t)=\big(\bar{{u}}_i(t),\bar{{v}}_i(t)\big)^\top$ and $\bm{\mathcal{L}}_{i}(u_h) =(\mathcal{L}_i^{(0)}(u_h),\mathcal{L}_i^{(1)}(u_h))$. Then the semi-discrete HWENO scheme \eqref{sec2:1dSemi_HWENO} can be rewritten as 
\begin{equation}
	\frac{\mathrm{d}}{\mathrm{d}t}\bm{U}_i(t)=\bm{\mathcal{L}}_{i}(u_h),
\end{equation}
which can be further discretized in time using a RK method. To suppress spurious oscillations near discontinuities, we introduce an OE procedure after each RK stage to obtain the 1D OE-HWENO schemes. For instance, the 1D OE-HWENO scheme, in conjunction with the classic third-order explicit  strong-stability-preserving (SSP) RK method, is 
\begin{equation}\label{sec2:1dH}
\left\{
\begin{aligned}
	{\bm{U}}^{\sigma,0}_i=&~{\bm{U}}^{n}_i,\\
	\bm{U}^{n,1}_i=&~{\bm{U}}^{\sigma,0}_i+\Delta t \bm{\mathcal{L}}_{i}(u^{\sigma,0}_h),\quad\quad\quad\quad\quad~~~~
	\bm{U}^{\sigma,1}_i = \mathcal{F}_{\rm OE} \{\bm{U}^{n,1}_{j} \}_{j\in\Lambda_i},
	\\ 
	\bm{U}^{n,2}_i=&~\frac34{\bm{U}}^{\sigma,0}_i+\frac14( {\bm{U}}^{\sigma,1}_i +\Delta t \bm{\mathcal{L}}_{i}(u^{\sigma,1}_h) ),\quad 
	\bm{U}^{\sigma,2}_i = \mathcal{F}_{\rm OE} \{\bm{U}^{n,2}_{j} \}_{j\in\Lambda_i},
	\\ 
	\bm{U}^{n,3}_i=&~\frac13{\bm{U}}^{\sigma,0}_i+\frac23( {\bm{U}}^{\sigma,2}_i +\Delta t \bm{\mathcal{L}}_{i}(u^{\sigma,2}_h) ), \quad 
	\bm{U}^{\sigma,3}_i = \mathcal{F}_{\rm OE} \{\bm{U}^{n,3}_{j} \}_{j\in\Lambda_i},
	\\
	 {\bm{U}}^{n+1}_i  = &~ {\bm{U}}^{\sigma,3}_i, 
\end{aligned}
\right.
\end{equation}
where ${u}^{\sigma,\ell}_h (x) = {\mathcal H} \{  {\bm{U}}^{\sigma,\ell}_i \}$, $\ell=0,1,2$;  
the operator $\mathcal{H}$ represents the standard 1D HWENO reconstruction based on the values of $\{{\bm{U}}^{\sigma,\ell}_i\}$; and the operator $\mathcal{F}_{\rm OE}$ denotes the 1D OE procedure with $\Lambda_i=\{i-1,i,i+1\}$. We will introduce the operators $\mathcal{H}$ and $\mathcal{F}$ in Subsections \ref{sec:1dH} and \ref{sec:1dF}, respectively. 
 
\subsubsection{HWENO operator $\mathcal{H}$}\label{sec:1dH}
Taking the sixth-order OE-HWENO method as an example,  
the piecewise polynomial functions $\{u_h^{\sigma,\ell}\}^2_{\ell=0}$ in \eqref{sec2:1dH} are reconstructed as follows: based on the values $\{\bm{U}_{i+k}^{\sigma,\ell}\}^1_{k=-1}$, we construct a quintic polynomial $p_0(x)$, a cubic polynomial $p_1(x)$, and two linear polynomials $\{p_k(x)\}^3_{k=2}$ in $I_i$. Then we compute the smoothness indicators of $\{p_k(x)\}^3_{k=0}$ in $I_i$. Finally,  through the nonlinear HWENO weights, we obtain 
\begin{equation}\label{sec3:1dHWENO_rec} 
	{u_h^{\sigma,\ell}(x)={\mathcal H}\{  {\bm{U}}^{\sigma,\ell}_i \}:=\omega^\text{H}_0\bigg(\frac{1}{\gamma^\text{H}_0}p_0(x) - \frac{\gamma^\text{H}_1}{\gamma^\text{H}_0}\tilde{q}_1(x) \bigg) + \omega^\text{H}_1\tilde{q}_1(x) \qquad \forall x \in I_i,} 
\end{equation} 
where $\tilde{q}_1(x)=\omega^\text{L}_1\big(\frac{1}{\gamma^\text{L}_1}p_1(x) - \frac{\gamma^\text{L}_2}{\gamma^\text{L}_1}p_2(x) - \frac{\gamma^\text{L}_3}{\gamma^\text{L}_1}p_3(x)\big) + \omega^\text{L}_2p_2(x)  + \omega^\text{L}_3p_3(x)$, and $\{\omega^\text{H}_k\}^1_{k=0}$ and $\{\omega^\text{L}_k\}^3_{k=1}$ are the nonlinear weights. 
The linear weights $\{\gamma^\text{H}_k\}^1_{k=0}$ and $\{\gamma^\text{L}_k\}^3_{k=1}$ are positive, with $\sum^{1}_{k=0}\gamma^\text{H}_k=1$ and $\sum^{3}_{k=1}\gamma^\text{L}_k=1$. 
For the reader's convenience, the detailed procedure of HWENO reconstruction \eqref{sec3:1dHWENO_rec} is provided in \ref{sec:A_1dHWENO}. 
Note that the standard HWENO operator, ${\mathcal H}$, is not scale-invariant and cannot consistently suppress spurious oscillations in problems spanning various scales. To address this issue, a scale-invariant dimensionless HWENO operator, ${\mathcal H}_{\mathcal D}$, will be introduced in Section \ref{sec2:SI} as an effective alternative to ${\mathcal H}$.

\subsubsection{OE operator $\mathcal{F}_{\rm OE}$} \label{sec:1dF}

We now detail the OE procedure  $\bm{U}^{\sigma,\ell+1}_i=\mathcal{F}_{\rm OE}\{\bm{U}^{n,\ell+1}_{j}\}_{j\in\Lambda_i}$  $(\ell=0,1,2)$ in \eqref{sec2:1dH}. For clarity and  simplicity, we denote $\bm{U}^{\sigma}_i=\bm{U}^{\sigma,\ell+1}_i$ and $\bm{U}_{i}=\bm{U}^{n,\ell+1}_{i}$ in the following. 
The OE modified moments $\bm{U}^{\sigma}_i$ are defined as the zeroth-order and first-order moments of ${u}_\sigma({x},{\Delta t})$, where ${u}_\sigma({x},\tau) \in \mathbb V_h^k$ ($0 \le \tau \le \Delta t$) represents the solution to the following damping equations:
\begin{equation}\label{sec2:OE1d_ODE}
	\left\{
	\begin{aligned}
		&\frac{\mathrm{d} }{\mathrm{d} \tau} \int_{I_i} {u}_\sigma \phi \mathrm{d}{x} + \sigma_{i}({u}_h^*) \int_{I_i}({u}_\sigma-P^{0}{u}_\sigma)  \phi \mathrm{d}{x}= 0 \quad \forall \phi \in \mathbb{P}^1(I_i),
		\\
		&{u}_\sigma({x},0) = {u}_h^* ({x}) = {\Pi_h}  \left \{  {\bm{U}}_{i} \right \}, 
	\end{aligned}
	\right. 
\end{equation}
where  the operator ${\Pi_h}$ denotes the $(k+1)$th-order linear Hermite reconstruction, and  
the damping coefficient  
$$\sigma_{i}(u_h^*)=\frac{ \alpha }{ {h_x} } {\widehat{\sigma}_{i}(u_h^*)}.$$ 
To make the damping strength scale-invariant, we define $\widehat{\sigma}_{i}(u_h^*)$ as
\begin{equation}\label{sec3:sigma_1d}
	\widehat {\sigma}_{i}(u_h^*)= 
	\begin{cases}
		\displaystyle
		0,\quad &\text{if}~{\max\limits_{1\le i \le N_x}} |\bar{u}_i -\overline u_\Omega | = 0,
		\\ 
		\displaystyle
		\frac{\sum\limits_{m\in\{0,1\}} h_x^m\left( \big| [\![\partial^m_x u_h^*]\!]_{i-\frac12}\big|+\big|[\![\partial^m_x u_h^*]\!]_{i+\frac12}\big| \right) }{ \max_{1\le i \le N_x} |\bar{u}_i -\overline u_\Omega |},\quad &\mbox{otherwise},
	\end{cases} 
\end{equation}
where $[\![ \partial^m_x u_h^* ]\!]_{i+\frac12}=\partial^m_x u_h^* (x_{i+\frac12}^+) - \partial^m_x u_h^* (x_{i+\frac12}^-)$ denotes the jump of $\partial^m_x u_h^*$ at the cell interface $x_{i+\frac12}$, and $\overline u_\Omega$ represents the average of $u_h^*$ over the whole domain $\Omega$, namely, 	for uniform meshes, 
$$\overline u_\Omega = \frac{1}{N_x} \sum_{i=1}^{N_x} \bar{u}_i.$$ 	
Note that $\max_{1\le i \le N_x} |\bar{u}_i -\overline u_\Omega |$ in \eqref{sec3:sigma_1d} is a global constant over all the cells and is computed  only once in each OE step.

Since the damping coefficient $\sigma_{i}({u}_h^*)$ only depends on the ``initial" value ${u}_\sigma({x},0) = {u}_h^* ({x})$, the damping equations \eqref{sec2:OE1d_ODE} are essentially a linear system of ODEs and are exactly solvable without requiring discretization. Moreover, we discover that the final expression of the OE procedure can be formulated in a simple form without implementing the linear Hermite reconstruction ${\Pi_h}$. Specifically, we  have the following conclusion.
	
\begin{theorem}[Exact solver of OE procedure]\label{thm:OE}
	Denote $(\bar{u}_i^{\sigma},\bar{v}_i^{\sigma}):=\bm{U}^{\sigma}_i$. 
	The OE procedure $ \bm{U}^{\sigma}_i=\mathcal{F}_{\rm OE}\{\bm{U}_{j}\}_{j\in\Lambda_i} $ can be exactly solved and explicitly expressed as 
	\begin{equation}\label{sec:1dOE_sol}
		\left\{\begin{aligned}
			\bar{u}_i^{\sigma}&=\bar{u}_i,
			\\
			\bar{v}_i^{\sigma}&=\bar{v}_i\exp\bigg(-\alpha\frac{{\Delta t}}{h_x} \widehat \sigma_{i}(u_h^*) \bigg),
		\end{aligned} \right.
	\end{equation}
	where the coefficient $\widehat \sigma_{i}(u_h^*)$ is defined in \eqref{sec3:sigma_1d}. 
	For the 1D sixth-order OE-HWENO scheme, 
	the jumps $[\![\partial^m_x u_h^*]\!]_{i+\frac12}$ in \eqref{sec3:sigma_1d} can be explicitly expressed as 
	\begin{equation}\label{eq:jump1}
		[\![\partial^m_x u_h^*]\!]_{i+\frac12}=
		\left\{  
		\begin{aligned} 
			&\frac{-13\bar{{u}}_{i-1}-31\bar{{u}}_i+31\bar{{u}}_{i+1}+13\bar{{u}}_{i+2}-50\bar{{v}}_{i-1}-370\bar{{v}}_{i}-370\bar{{v}}_{i+1}-50\bar{{v}}_{i+2}}{108},~m=0
			\\
			&\frac{-5\bar{{u}}_{i-1}+5\bar{{u}}_{i}+5\bar{{u}}_{i+1}-5\bar{{u}}_{i+2}-22\bar{{v}}_{i-1}-54\bar{{v}}_{i}+54\bar{{v}}_{i+1}+{22\bar{{v}}_{i+2}}}{36h_x},\quad\quad\quad~ m=1.
		\end{aligned}
		\right.
	\end{equation}
	Note that the above explicit exact solver \eqref{sec:1dOE_sol} of the OE procedure only involves the values $\bar{u}_i$ and $\bar{v}_i$, without the need to formulate the linear Hermite reconstruction $u_h^*$. 
\end{theorem}

\begin{proof}
	 The proof largely follows the analysis in section 2.2 of \cite{PSW}.
	Let $\{\phi^{(\ell)}_i(x)\}^k_{\ell=1}$ be a local orthogonal basis  of $\mathbb{P}^{k}(I_i)$, for instance, the scaled
	Legendre polynomials: 
\begin{equation*}
	\begin{aligned}
		&\phi_i^{(0)}(x)=1, \quad 
		\phi_i^{(1)}(x)=\xi_i, \quad 
		\phi_i^{(2)}(x)=\xi_i^2-\frac{1}{12}, \quad 
		\phi_i^{(3)}(x)=\xi_i^3-\frac{3}{20}\xi_i, \\&
		\phi_i^{(4)}(x)=\xi_i^4-\frac{3}{14}\xi_i^2+\frac{3}{560}, \quad 
		\phi_i^{(5)}(x)=\xi_i^5-\frac{5}{18}\xi_i^3+\frac{5}{336}\xi_i,\cdots,
	\end{aligned}
\end{equation*}
where $\xi_i:= {(x-x_i)}/{h_x}$. 

For the sixth-order linear Hermite reconstruction  $u_h^*(x) = {\Pi_h} \{ {\bm U}_i\}$, 
the reconstructed solution $u_h^*(x)$
 in the cell $I_i$ is a quintic polynomial 
\begin{equation}\label{eq:uhstar}
	u_h^*(x)\Big|_{I_i} =: p_{0}^{(i)}(x)=\sum_{\ell=0}^k c_{0,\ell} \phi^{(\ell)}_i(x),
\end{equation}
 where the coefficients $\{c_{0,\ell}\}$ are determined by matching the relations in \eqref{eq:p0}, and the expressions of $\{c_{0,\ell}\}$ are listed in Table \ref{secA:coe1D} for the sixth-order OE-HWENO scheme. 
 
	Assume that the solution $u_\sigma(x,\tau) \in \mathbb V_h^k$ of the damping equations \eqref{sec2:OE1d_ODE} can be expressed as
	\begin{equation*}
		u_\sigma(x,\tau)=\sum_{\ell=0}^k c_{\ell}(\tau) \phi^{(\ell)}_i(x) \qquad \forall x \in I_i, ~~ 0 \le \tau \le \Delta t, 
	\end{equation*} 
	with $c_{\ell} ( 0) = c_{0,\ell} $ and 
	$$\bar{u}^\sigma_i= \frac{1}{h_x}\int_{I_i} u_\sigma(x,\Delta t) \mathrm{d}x, \qquad \bar{v}^{\sigma}_i= \frac{1}{h_x}\int_{I_i} u_\sigma(x,\Delta t) \frac{x-x_i}{h_x} \mathrm{d}x = \frac{c_{1}(\Delta t)}{h_x}  \int_{I_i} \left( {\phi^{(1)}_i(x)} \right)^2  \mathrm{d}x.$$
	Note that   
	\[
	({u}_\sigma-P^0{u}_\sigma)(x,\tau)=\sum^k_{\ell=1} c_{\ell}(\tau)\phi^{(\ell)}_{i}(x).
	\]
	Taking $\phi(x)= \phi_i^{(0)}(x) =\frac{1}{h_x}$ in \eqref{sec2:OE1d_ODE}, we have 
	\begin{align*} 
		\frac{\mathrm{d}}{\mathrm{d}\tau} \left( \frac{1}{h_x}\int_{I_i} u_\sigma(x,\tau) \mathrm{d}x \right) & = - \sigma_{i}({u}_h^*) \int_{I_i}({u}_\sigma-P^{0}{u}_\sigma)  \phi \mathrm{d}{x}
		= 0, 
	\end{align*} 
	which yields
	\begin{equation}\label{sec2:1dOE_uave}
		\bar{u}_i^{\sigma} = \frac{1}{h_x}\int_{I_i} u_\sigma(x,\Delta t) \mathrm{d}x = \frac{1}{h_x}\int_{I_i} u_\sigma(x,0) \mathrm{d}x  =\bar{u}_i.
	\end{equation} 
	Taking $\phi(x)=\frac{x-x_i}{h_x}$ in \eqref{sec2:OE1d_ODE}, we derive 
	\begin{equation}\label{sec2:OEstepVi} 
		\frac{\mathrm{d} }{\mathrm{d} \tau} c_{1} (\tau) 
		+ c_{1} (\tau)  \alpha  \frac{{\widehat \sigma_{i}(u_h^*)}}{h_x}
		= 0.
	\end{equation}
	Integrating \eqref{sec2:OEstepVi} from $\tau=0$ to $\Delta t$ gives
	\begin{equation*} 
		c_{1} (\Delta t) = c_1(0) \exp\bigg(-\alpha\frac{{\Delta t}}{h_x} \widehat  \sigma_{i}(u_h^*) \bigg),
	\end{equation*}
	or equivalently, 
	\begin{equation*}
		\bar{v}_i^{\sigma} =\bar{v}_i\exp\bigg(-\alpha\frac{{\Delta t}}{h_x} \widehat  \sigma_{i}(u_h^*) \bigg),
	\end{equation*}
    where the coefficient $\widehat  \sigma_{i}(u_h^*)$ is defined in \eqref{sec3:sigma_1d}. 
    For the 1D sixth-order OE-HWENO scheme, 
    using \eqref{eq:uhstar} and the expressions of $\{c_{0,\ell}\}$ listed in Table \ref{secA:coe1D}, we obtain 
	\begin{equation*}
		\begin{aligned}
			p_{0}^{(i)}(x^-_{i+\frac12})&=\frac{13}{108}\bar{u}_{i-1}+\frac{7}{12}\bar{u}_{i}+\frac{8}{27}\bar{u}_{i+1}+\frac{25}{54}\bar{v}_{i-1}+\frac{241}{54}\bar{v}_{i}-\frac{28}{27}\bar{v}_{i+1},
			\\
			p_{0}^{(i+1)}(x^+_{i+\frac12})&=\frac{8}{27}\bar{u}_{i}+\frac{7}{12}\bar{u}_{i+1}+\frac{13}{108}\bar{u}_{i+2}+\frac{28}{27}\bar{v}_{i}-\frac{241}{54}\bar{v}_{i+1}-\frac{25}{54}\bar{v}_{i+2},
			\\
			\partial_x p_{0}^{(i)} (x^-_{i+\frac12})&=\frac{1}{h_x}\left(\frac{5}{36}\bar{u}_{i-1}-\frac{9}{4}\bar{u}_{i}+\frac{19}{9}\bar{u}_{i+1}+\frac{11}{18}\bar{v}_{i-1}-\frac{97}{18}\bar{v}_{i}-\frac{62}{9}\bar{v}_{i+1} \right),
			\\
			\partial_x p_{0}^{(i+1)} (x^+_{i+\frac12})&=\frac{1}{h_x} \left(-\frac{19}{9}\bar{u}_{i}+\frac{9}{4}\bar{u}_{i+1}-\frac{5}{36}\bar{u}_{i+2}-\frac{62}{9}\bar{v}_{i}-\frac{97}{18}\bar{v}_{i+1}+\frac{11}{18}\bar{v}_{i+2} \right).
		\end{aligned}
	\end{equation*}
	Therefore, the jumps $[\![\partial^m_x u_h^*]\!]_{i+\frac12}=\partial^m_x p_0^{(i+1)} (x^+_{i+\frac12}) - \partial^m_x p_0^{(i)}(x^-_{i+\frac12})$ can be explicitly expressed as 
	\begin{equation*}
		[\![\partial^m_x u_h^*]\!]_{i+\frac12}=
		\left\{  
		\begin{aligned} 
			&\frac{-13\bar{{u}}_{i-1}-31\bar{{u}}_i+31\bar{{u}}_{i+1}+13\bar{{u}}_{i+2}-50\bar{{v}}_{i-1}-370\bar{{v}}_{i}-370\bar{{v}}_{i+1}-50\bar{{v}}_{i+2}}{108},~m=0
			\\
			&\frac{-5\bar{{u}}_{i-1}+5\bar{{u}}_{i}+5\bar{{u}}_{i+1}-5\bar{{u}}_{i+2}-22\bar{{v}}_{i-1}-54\bar{{v}}_{i}+54\bar{{v}}_{i+1}+22\bar{{v}}_{i+2}}{36h_x},\quad\quad\quad~ m=1.
		\end{aligned}
		\right.
	\end{equation*}
The proof is completed. 
\end{proof}

Some notable advantages of the proposed OE technique are summarized as follows.

\begin{remark}[{Stability}]\label{sec2:rmk1}
	Thanks to the simple exact solver \eqref{sec:1dOE_sol} of the OE procedure, the OE-HWENO method remains stable when coupled with standard explicit RK time discretization using a normal CFL number, even in the presence of highly stiff damping terms associated with strong shocks. Unlike the damping-based oscillation-free HWENO method \cite{ZQ2}, our OE approach does not require empirical, problem-dependent parameters or (modified) exponential time discretizations.
\end{remark}

\begin{remark}[{Conservation}]\label{sec2:rmk2}
	Given that $\bar{u}_i^{\sigma}=\bar{u}_i$ in \eqref{sec:1dOE_sol}, it is clear that the zeroth-order moment (i.e., the cell averages) remains unchanged in the OE modification. This means that the OE procedure preserves the local conservation of the HWENO solutions.
\end{remark}

\begin{remark}[{Efficiency and Simplicity}]\label{sec2:rmk3}
	The OE procedure is non-intrusive and completely independent of the RK stage update. This design allows for the seamless integration of the OE technique into existing HWENO codes as an independent module with only very slight adjustments. The implementation of the OE procedure \eqref{sec:1dOE_sol} is highly simple and efficient, as it involves only the multiplication of first-order moments by a damping factor. This approach significantly differs from the moment-limiting techniques \cite{LSQ2,ZCQ,ZQ1} used in existing HWENO schemes.
\end{remark}

Besides the above-mentioned features, we can prove that the OE procedure maintains the original high-order accuracy of the HWENO schemes, as shown in the following Theorem \ref{thm:accuracy}. 
This is different from the moment-limiting techniques \cite{LSQ2,ZCQ,ZQ1} in the literature, which retained at most fifth-order accuracy for originally sixth-order HWENO schemes. 

	\begin{theorem}[Maintain accuracy]\label{thm:accuracy}
		Consider the $(k+1)$th-order 1D OE-HWENO scheme \eqref{sec2:1dH} under a CFL condition $\alpha\frac{{\Delta t}}{h_x} \le C_\text{cfl}$ with $C_\text{cfl} < 1$. 
		Assume that the exact solution $u(x,t) \in C^{k+1}(\Omega)$ for a given $t \in [0,T]$, and the boundary conditions are periodic. 
		Let $\bm{U}_i^e=(\bar u_i^e, \bar v_i^e)$ denote the exact zeroth- and first-order moments of $u(x,t)$ on cell $I_i$. 
		If $\bm{U}_i=(\bar u_i, \bar v_i)$ are $(k+1)$th-order accurate approximations to 
		$\bm{U}_i^e$ for all $i$, then the OE modified moments $\bm{U}^{\sigma}_i=\mathcal{F}_{\rm OE}\{\bm{U}_{j}\}_{j\in\Lambda_i} $
		are also $(k+1)$th-order accurate approximations to 
		$\bm{U}_i^e$, namely, 
		\begin{equation}\label{eq:OEarracy}
		   \max_{1\le i\le N_x}	\big\| \bm{U}^{\sigma}_i - \bm{U}_i^e\big\| \lesssim h_x^{k+1}. 
		\end{equation}
	This means the OE procedure maintains the original high-order accuracy of the HWENO schemes. 
	\end{theorem}

\begin{proof}
	According to the solution \eqref{sec:1dOE_sol} of the OE procedure (Theorem \ref{thm:OE}), we have
	\begin{equation*}
		\bm{U}^{\sigma}_i - \bm{U}_i = \begin{pmatrix} 0\\ \exp\bigg(-\alpha\frac{{\Delta t}}{h_x} \widehat  \sigma_{i}(u_h^*) \bigg)-1 \end{pmatrix}\bar{v}_i.
	\end{equation*}
It follows that 
		\begin{align*}\nonumber
		\big\| \bm{U}^{\sigma}_i - \bm{U}_i\big\| & = |\bar v_i| \left( 1 - \exp\bigg(-\alpha\frac{{\Delta t}}{h_x} \widehat  \sigma_{i}(u_h^*) \bigg)   \right)
		 \le |\bar v_i|  \alpha\frac{{\Delta t}}{h_x} \widehat \sigma_{i}(u_h^*)
		 \le |\bar v_i|  {{\widehat \sigma_{i}}(u_h^*)}, 
\end{align*}
	where we have used the elementary inequality $1 -e^{-x} \le x$ for all $x\ge 0$ in the second step and the CFL condition in the third step. 
	If $\max_{1\le i \le N_x} |\bar{u}_i -\overline u_\Omega | = 0$, then $\widehat \sigma_{i}(u_h^*)=0$ and the conclusion holds evidently. 
	In the following, we assume that $\max_{1\le i \le N_x} |\bar{u}_i -\overline u_\Omega | \ge C> 0$. 
	According to the definition of $\widehat \sigma_{i}(u_h^*)$ in \eqref{sec3:sigma_1d}, we obtain 
			\begin{align}\nonumber
		\big\| \bm{U}^{\sigma}_i - \bm{U}_i\big\| &  \le |\bar v_i|  \widehat \sigma_{i}(u_h^*)
		\\  \nonumber
		&= |\bar v_i| \left( 
		\frac{ \big| [\![ u_h^*]\!]_{i-\frac12}\big|+\big|[\![ u_h^*]\!]_{i+\frac12}\big| }{ \max_{1\le i \le N_x} |\bar{u}_i -\overline u_\Omega |}
		+ h_x\frac{ \big| [\![\partial_x u_h^*]\!]_{i-\frac12}\big|+\big|[\![\partial_x u_h^*]\!]_{i+\frac12}\big| }{ \max_{1\le i \le N_x} |\bar{u}_i -\overline u_\Omega |}
		\right)
		\\  \label{eq:WKL55}
		& = \frac{ |\bar v_i| }{ \max_{1\le i \le N_x} |\bar{u}_i -\overline u_\Omega | }  \left( 
		\Xi_{i-\frac12} + \Xi_{i+\frac12} 
		\right) 
	\end{align}
	with 
	$$
	\Xi_{i+\frac12} := \big|[\![ u_h^*]\!]_{i+\frac12}\big| + h_x \big|[\![\partial_x u_h^*]\!]_{i+\frac12}\big|.
	$$
	Let $\widehat u(x) \in \mathbb V_h^k$ denote the $(k+1)$th-order linear Hermite reconstruction from the exact zeroth- and first-order moments $\bm{U}_i^e=(\bar u_i^e, \bar v_i^e)$, namely, 
	$$
	\widehat u(x) = {\Pi_h}  \{ \bm{U}_i^e  \}. 
	$$
	According to the approximation accuracy of Hermite reconstruction for the exact solution $u(x,t) \in C^{k+1}(\Omega)$ with a given $t \in [0,T]$, we have 
\begin{equation}\label{eq:WKL2}
		  \left\| \widehat u(x) - u(x,t) \right\|_{L^\infty(\Omega)} \lesssim h_x^{k+1}, \qquad  \left\| \partial_x \widehat u(x) - \partial_x u(x,t) \right\|_{L^\infty(\Omega)} \lesssim h_x^{k}. 
\end{equation}
	Note that  $u(x,t) \in C^{k+1}(\Omega)$ satisfies 
\begin{equation}\label{eq:WKL3}
		\big|[\![ u]\!]_{i+\frac12}\big| + h_x \big|[\![\partial_x u]\!]_{i+\frac12}\big| = 0. 
\end{equation}
Using \eqref{eq:WKL2} and \eqref{eq:WKL3}, we derive 
\begin{align} \nonumber
	\max_{0\le i \le N_x} \left\{   \big|[\![ \widehat u]\!]_{i+\frac12}\big| + h_x \big|[\![\partial_x \widehat u]\!]_{i+\frac12}\big| \right\}
	& =  \max_{0\le i \le N_x} \left\{   \big|[\![ \widehat u  -  u  ]\!]_{i+\frac12}\big| + h_x \big|[\![\partial_x \widehat u - \partial_x  u ]\!]_{i+\frac12}\big|  \right\}
	\\  \nonumber
	& \le 2 \left\| \widehat u(x) - u(x,t) \right\|_{L^\infty(\Omega)} + 2 h_x  \left\| \partial_x \widehat u(x) - \partial_x u(x,t) \right\|_{L^\infty(\Omega)} 
	\\ \label{eq:WKL22}
	& \lesssim h_x^{k+1}. 
\end{align}
Similar to \eqref{eq:jump1}, we have 
		\begin{equation}\label{eq:jump2}
			\begin{aligned}
			[\![ \widehat u]\!]_{i+\frac12} &= 	\frac{-13\bar{{u}}_{i-1}^e-31\bar{{u}}_i^e+31\bar{{u}}_{i+1}^e+13\bar{{u}}_{i+2}^e-50\bar{{v}}_{i-1}^e-370\bar{{v}}_{i}^e-370\bar{{v}}_{i+1}^e-50\bar{{v}}_{i+2}^e}{108},
			\\
		 [\![\partial_x \widehat u]\!]_{i+\frac12}	& = \frac{-5\bar{{u}}_{i-1}^e+5\bar{{u}}_{i}^e+5\bar{{u}}_{i+1}^e-5\bar{{u}}_{i+2}^e-22\bar{{v}}_{i-1}^e-54\bar{{v}}_{i}^e+54\bar{{v}}_{i+1}^e+{22\bar{{v}}_{i+2}^e}}{36h_x}.
			\end{aligned}
\end{equation}
Combining \eqref{eq:jump1} with \eqref{eq:jump2}, we obtain 
\begin{align*}
	\left| [\![  u_h^* - \widehat u   ]\!]_{i+\frac12} \right| 
&\le \frac{13+31+31+13}{108} \max_{ i-1 \le j \le i+2 } \big\{ | \bar{{u}}_{j} - \bar{{u}}_{j}^e  | \big \} +  \frac{50+370+370+50}{108} \max_{ i-1 \le j \le i+2 } \big \{ | \bar{{v}}_{j} - \bar{{v}}_{j}^e  | \big \}
\\
& \le  {\frac{70}{9}} \left(   \max_{ i-1 \le j \le i+2 } \big\{ | \bar{{u}}_{j} - \bar{{u}}_{j}^e  | \big \} +   \max_{ i-1 \le j \le i+2 } \big \{ | \bar{{v}}_{j} - \bar{{v}}_{j}^e  | \big \} \right),
\end{align*}
and similarly, we get 
$$
\left| [\![ \partial u_h^* - \partial \widehat u   ]\!]_{i+\frac12} \right| \le {\frac{38}{9 h_x}} \left(   \max_{ i-1 \le j \le i+2 } \big\{ | \bar{{u}}_{j} - \bar{{u}}_{j}^e  | \big \} +   \max_{ i-1 \le j \le i+2 } \big \{ | \bar{{v}}_{j} - \bar{{v}}_{j}^e  | \big \} \right). 
$$
Under the hypothesis that 
$\bm{U}_i=(\bar u_i, \bar v_i)$ are $(k+1)$th-order accurate approximations to 
$\bm{U}_i^e$ for all $i$, we have 
\begin{equation}\label{eq:WKL33}
 \max_{0\le i\le N_x} \left\{ 	\left| [\![  u_h^* - \widehat u   ]\!]_{i+\frac12} \right| \right \} 
  \lesssim h_x^{k+1}, \qquad  \max_{0\le i\le N_x} \left\{ 	\left| [\![ \partial_x u_h^* - \partial_x \widehat u   ]\!]_{i+\frac12} \right| \right \} 
  \lesssim h_x^{k}. 
\end{equation}
Using the estimates \eqref{eq:WKL33} and \eqref{eq:WKL22}, we obtain 
\begin{align}
	\max_{0\le i \le N_x}\Xi_{i+\frac12} &=  \max_{0\le i \le N_x} \left\{  \big|[\![ u_h^*]\!]_{i+\frac12}\big| + h_x \big|[\![\partial_x u_h^*]\!]_{i+\frac12}\big| \right\} \nonumber
	\\
	& \le \max_{0\le i \le N_x} \left\{  \big|[\![ u_h^* - \widehat u  ]\!]_{i+\frac12}\big| + h_x \big|[\![\partial_x  u_h^*  - \partial_x \widehat u  ]\!]_{i+\frac12}\big| \right\} \nonumber
	\\
	& \quad + \max_{0\le i \le N_x} \left\{   \big|[\![ \widehat u]\!]_{i+\frac12}\big| + h_x \big|[\![\partial_x \widehat u]\!]_{i+\frac12}\big| \right\} \nonumber
	\\
	& \lesssim h_x^{k+1}, \nonumber
\end{align}
which together with \eqref{eq:WKL55} implies \eqref{eq:OEarracy}. This proof is completed. 
\end{proof}

\subsubsection{Extension to 1D hyperbolic systems} 

For the OE-HWENO method to solve the 1D hyperbolic system of conservation laws $\bm{u}_t+\bm{f}(\bm{u})_x=0$,  we propose the OE procedure by the following damping equations:  
\begin{equation}\label{eq:OE-1Dsystem}
	\left\{
	\begin{aligned}
		&\frac{\mathrm{d} }{\mathrm{d} \tau } \int_{I_i} \bm{u}_\sigma \cdot \bm{\phi} \mathrm{d}x + \sigma_{i}(\bm{u}_h^*) \int_{I_i}(\bm{u}_\sigma-P^{0}\bm{u}_\sigma)\cdot\bm{\phi} \mathrm{d}x= 0 \quad \forall \bm{\phi} \in [\mathbb{P}^{1}(I_i)]^N,
		\\
		&\bm{u}_\sigma(x,0) = \bm{u}_h^*(x)= \Pi_h \{\bm{U}_{i}\}, 
	\end{aligned}
	\right. 
\end{equation} 
where ${\bm U}_i=(\bar{\bm{u}}_i,\bar{\bm{v}}_i)$, the damping coefficient $\sigma_{i}(\bm{u}_h^*)= \frac{\alpha}{h_x} \widehat \sigma_{i}(\bm{u}_h^*) $, $\alpha$ denotes the (estimated) maximum wave speed in the $x$-direction. Here, $ \widehat \sigma_{i}(\bm{u}_h^*)$  is defined as
\begin{equation}\label{eq:sigma2}
	\widehat \sigma_{i}(\bm{u}_h^*):=\max\limits_{1\le \ell\le N}  \widehat \sigma_{i}( {u}_{h}^{*,\ell}),
\end{equation}
where ${u}_{h}^{*,\ell}$ is the $\ell$-th component of $\bm{u}_h^*$, and $ \widehat \sigma_{i}( {u}_{h}^{*,\ell})$ is computed by \eqref{sec3:sigma_1d} and \eqref{eq:jump1}. 

Similar to Theorem \ref{thm:OE}, one can obtain the exact solver of the OE procedure defined by  \eqref{eq:OE-1Dsystem}. 

\begin{theorem}\label{thm:OE2}
	The OE procedure $(\overline{\bm u}_i^{\sigma},\overline{\bm v}_i^{\sigma}):=\bm{U}^{\sigma}_i=\mathcal{F}_{\rm OE}\{\bm{U}_{j}\}_{j\in\Lambda_i} $ for 1D hyperbolic systems can be exactly solved and explicitly expressed as 
	\begin{equation*} 
		\left\{\begin{aligned}
			\overline{\bm u}_i^{\sigma}&=\bar{\bm u}_i,
			\\
			\overline{\bm v}_i^{\sigma}&=\bar{\bm v}_i\exp\bigg(-\alpha\frac{{\Delta t}}{h_x} \widehat \sigma_{i}({\bm u}_h^*) \bigg),
		\end{aligned} \right.
	\end{equation*}
	where the coefficient $\widehat \sigma_{i}({\bm u}_h^*)$ is defined in \eqref{eq:sigma2}. 
\end{theorem}

The proof of Theorem \ref{thm:OE2} is similar to that of Theorem \ref{thm:OE} and thus is omitted here.

\subsubsection{Scale Invariance}\label{sec2:SI}
In this subsection, we introduce the concept of scale invariance and its importance in consistently suppressing spurious oscillations across various scales. We will demonstrate that the OE operator, $\mathcal{F}_{\rm OE}$, satisfies the scale-invariant property. However, the nonlinear HWENO and standard WENO operators generally do not exhibit scale invariance, as discussed in \cite{CW, Deng2023, DLWW} with specific techniques to address these issues. A common strategy in \cite{CW, Deng2023, DLWW} is to achieve the scale-invariant property by modifying the nonlinear weights. 
In the following, we propose a different, simple yet universal technique---a generic dimensionless transformation. This transformation can render any nonlinear reconstruction operator, such as WENO and HWENO operators, scale-invariant.

\begin{definition}[Scale invariance]\label{sec2:si_def} 
	For any $\bm{U}_j=(\bar{\bm u}_j,\bar{\bm v}_j)$, define the affine transformation: 
	$$
	{\mathcal A}_{\lambda,c} \bm{U}_j = \left( \lambda \bar{\bm u}_j +c,  \lambda \bar{\bm v}_j  \right), \qquad {\mathcal A}_{\lambda,c} {\bm u}_h =  \lambda {\bm u}_h +c. 
	$$
	A operator  $\mathcal{P}$ is termed scale-invariant if 
	$\mathcal{P}$ and ${\mathcal A}_{\lambda,c}$ are commutative for any $\lambda\ne0$ and $c\in \mathbb{R}$.  
\end{definition}

\begin{theorem}\label{sec2:thm_F}
	The OE operator $\mathcal{F}_{\rm OE}$ is scale-invariant. Specifically, we have $\mathcal{F}_{\rm OE}\{ {\mathcal A}_{\lambda,c} \bm{U}_j \}_{j\in\Lambda_i}={\mathcal A}_{\lambda,c} \{ \mathcal{F}_{\rm OE} \{  \bm{U}_j \}_{j\in\Lambda_i} \}$ for any $\lambda\ne0$ and $c\in \mathbb{R}$. 
\end{theorem}
\begin{proof}
	Due to the linearity, 
	the linear Hermite reconstruction operator $\Pi_h$ satisfies $\Pi_h\{ {\mathcal A}_{\lambda,c} \bm{U}_j \} = \lambda \Pi_h\{\bm{U}_j\}+c = \lambda {\bm u}^*_h+c$. 
	From \eqref{sec3:sigma_1d} and \eqref{eq:sigma2}, we observe that the damping coefficient is dimensionless, satisfying 
	$\widehat {\sigma}_{i}(\lambda {\bm u}^*_h+c) = \widehat {\sigma}_{i}( {\bm u}^*_h)$. 
	Hence, based on the formulation \eqref{sec:1dOE_sol}, we know that the OE modified moments for the scaled data satisfy 	
\begin{align*}  
		 \mathcal{F}_{\rm OE}\{ {\mathcal A}_{\lambda,c} \bm{U}_j \}_{j\in\Lambda_i} & 
		 = \left( \lambda \bar{\bm u}_i +c,  \lambda \bar{\bm v}_i \exp\bigg(-\alpha\frac{{\Delta t}}{h_x} \widehat \sigma_{i}(\lambda{\bm u}_h^*+c) \bigg)    \right) 
		 \\
		 & = \left( \lambda \bar{\bm u}_i +c,  \lambda \bar{\bm v}_i \exp\bigg(-\alpha\frac{{\Delta t}}{h_x} \widehat \sigma_{i}({\bm u}_h^*) \bigg)    \right)
		 \\
		 & =  \left( \lambda \bar{\bm u}_i^\sigma,  \lambda \bar{\bm v}_i^\sigma  \right) = {\mathcal A}_{\lambda,c} \{ \mathcal{F}_{\rm OE} \{  \bm{U}_j \}_{j\in\Lambda_i} \}. 
\end{align*}
		 The proof is completed. 
\end{proof}

Numerical schemes that lack scale invariance, such as the damping-based HWENO schemes proposed in \cite{ZQ2}, may produce spurious oscillations near discontinuities for problems across various scales, as illustrated in Figures \ref{sec3:Fig_LWR} and \ref{sec3:Fig_Lax}. In fact, both the damping terms and the HWENO operators in \cite{ZQ2} are not scale-invariant.

To make the nonlinear HWENO reconstruction operator ${\mathcal H}$ scale-invariant, we propose a simple yet universal dimensionless transformation, $\mathcal{D}$. This approach utilizes dimensionless variables and parameters that normalize the problem's scale, thus allowing the reconstruction operators to function effectively regardless of the absolute scale of the underlying physical quantities. Consequently, our dimensionless transformation $\mathcal{D}$ can render any nonlinear reconstruction operator, including WENO and HWENO operators, scale-invariant.

\begin{definition}[Generic dimensionless transformation $\mathcal{D}$] 
	The dimensionless transformation $\mathcal{D}$ is an affine transformation that normalizes the involved quantities by 
	the average, maximum, and minimum of the solution values in a stencil. For all the zeroth- and first-order moments $\{ \bar{u}_j, \bar{v}_j \}$ in the stencil $\Lambda_i$,  the transformation $\mathcal{D}: \{ \bar{u}_j, \bar{v}_j \} \rightarrow \{ \widehat{u}_j, \widehat{v}_j \} $ is defined by 
	\begin{equation}\label{eq:D}
		\widehat{{u}_j}=\frac{  \bar{u}_j-u_{\rm{ave}}  }{u_{\rm{max}}-u_{\rm{min}}+\epsilon}, \quad 
		\widehat{{v}_j}=\frac{\bar{v}_j}{u_{\rm{max}}-u_{\rm{min}}+\epsilon} \quad \forall j \in \Lambda_i
	\end{equation}
	with
	\begin{equation}\label{sec2:def_AveMaxMin}
		u_{\rm ave}=\frac{1}{\# \Lambda_i}\sum_{j \in \Lambda_i}\bar{u}_j,\quad 
		u_{\rm max}=\max_{j \in \Lambda_i}\{\bar{u}_j\},\quad 
		u_{\rm min}=\min_{j \in \Lambda_i}\{\bar{u}_j\},
	\end{equation}
	where $\# \Lambda_i$ denotes the number of cells in $\Lambda_i$, and $\epsilon$ is the machine epsilon to avoid division by zero, e.g., $\epsilon=10^{-15}$ for double precision.
\end{definition} 

\begin{definition}[Dimensionless HWENO operator $\mathcal{H_D}$]\label{sec3:defH_AveMaxMin} 
	The dimensionless HWENO operator is defined as 
	\begin{equation}\label{eq:HD}
		\mathcal{H_D}:=\mathcal{D}^{-1}_{u} \mathcal{H}\mathcal{D},
	\end{equation}
	where $\mathcal{D}: \{ \bar{u}_j, \bar{v}_j \} \rightarrow \{ \widehat{u}_j, \widehat{v}_j \} $ 
	is the dimensionless transformation defined by \eqref{eq:D}, 
	$\mathcal{H}$ is the standard HWENO operator which maps to the normalized moments $\{ \widehat{u}_j, \widehat{v}_j \}_{j \in \Lambda_i}$ to the dimensionless reconstructed solution $\widehat u_h(x)\big|_{I_i} = \widehat u_{h,i}(x)$, and 
	$\mathcal{D}^{-1}_{u}$ is the ``inverse'' transformation
	defined as	 
\begin{equation}\label{inverseD}
		  u_{h,i}(x) = \mathcal{D}^{-1}_{u} \widehat u_{h,i}(x) = (u_{\rm{max}}-u_{\rm{min}}+\epsilon)  \widehat u_{h,i}(x)  +u_{\rm{ave}}. 
\end{equation}
	Note that $\mathcal{D}^{-1}_{u}$ depends on the maximum, minimal, and average values of the original data $u$.  
\end{definition}
	
Based on above definitions, we have the following conclusion. 

\begin{theorem}\label{sec2:thm_si}
	The dimensionless HWENO operator $\mathcal{H_D}$ defined in \eqref{eq:HD} is scale-invariant. 
\end{theorem} 
\begin{proof} 
	According to the definition of $\mathcal{D}: \{ \bar{u}_j, \bar{v}_j \} \rightarrow \{ \widehat{u}_j, \widehat{v}_j \} $ in \eqref{eq:D}--\eqref{sec2:def_AveMaxMin}, we have  
	$$
	\mathcal{D} \{ \lambda \bar{u}_j + c, \lambda \bar{v}_j \} =  \{ \widehat{u}_j, \widehat{v}_j \}  = \mathcal{D} \{  \bar{u}_j,  \bar{v}_j \} \quad \forall \lambda \neq 0,~~\forall c \in \mathbb R.
	$$
	It follows that 
	$$
	{\mathcal H}	\mathcal{D} \{ \lambda \bar{u}_j + c, \lambda \bar{v}_j \} = {\mathcal H} \mathcal{D} \{  \bar{u}_j,  \bar{v}_j \} = \widehat u_h(x). 
	$$
	From \eqref{inverseD}, we can observe that 	
\begin{align*}
			\mathcal{D}^{-1}_{\lambda u + c} \widehat u_h(x) & = ( \lambda  u_{\rm{max}}- \lambda  u_{\rm{min}}+ \lambda \epsilon)  \widehat u_h(x)  + \lambda  u_{\rm{ave}} + c 
			\\
		& =  \lambda (  u_{\rm{max}}-   u_{\rm{min}}+\epsilon)  \widehat u_h(x)  + \lambda  u_{\rm{ave}} + c 
		= \lambda \mathcal{D}^{-1}_{u} \widehat u_h(x) + c,
\end{align*}
	where in the second step we have treated the tiny number $\epsilon$ as machine zero. 
	This indicates that $\mathcal{H_D}$ and ${\mathcal A}_{\lambda,c}$ are commutative for any $\lambda\ne0$ and $c\in \mathbb{R}$. 
	The proof is completed.
\end{proof}
	
Let $\mathcal{S}_t$ denotes the solution operator of equation \eqref{sec2:1dHCLS}, namely, $\mathcal{S}_t(u(x,0))=u(x,t)$.
Let $\mathcal{E}_n$ represent the solution operator of the OE-HWENO schemes, i.e., $\mathcal{E}_n\bm{U}_i^0)=\bm{U}_i^n$. For the homogeneous flux, namely, $f(\lambda u)=\lambda f(u)$, the exact solution operator $\mathcal{S}_t$ is also homogeneous: $\mathcal{S}(\lambda u(x,0))=\lambda \mathcal{S}(u(x,0)).$

\begin{theorem}[Homogeneous]
	If the flux function $f(u)$ and the numerical flux $\hat{f}_{i+\frac12}$ are both homogeneous, then the OE-HWENO solution operator with the dimensionless HWENO operator $\mathcal{H_D}$ is homogeneous:
	\begin{align*}
		\mathcal{E}_n(\lambda\bm{U}_i^0)=\lambda\mathcal{E}_n(\bm{U}_i^0)\qquad \forall \lambda \in \mathbb{R}.
	\end{align*}
\end{theorem}
\begin{proof}
	When $f(u)$ and $\hat{f}_{i+\frac12}$ are homogeneous, the $\bm{\mathcal{L}}_i(u_h^{n,\ell})$ is homogeneous  for the OE-HWENO schemes with $\mathcal{H_D}$, demonstrating that the local solution operator for each RK stage in \eqref{sec2:1dH} is homogeneous. According to Theorem \ref{sec2:thm_F}, the OE operator $\mathcal{F}$ is homogeneous. Consequently, $\mathcal{E}_n$ is homogeneous.
\end{proof}	
	
\subsubsection{Evolution invariance} 
	
We now discuss another important invariant property, namely, evolution invariance. This property ensures that for problems featuring slow-propagating waves (resulting in large time steps) and fast-propagating waves (resulting in small time steps), if their solutions are consistent, then the OE-HWENO method with the same CFL number will produce identical results after a fixed number of steps. More specifically, for any constant $\lambda>0$, let $\mathcal{S}^{\lambda}_t$ be the solution operator of the equation $u_t+\lambda f(u)_x=0$, which is essentially the reformulation of equation \eqref{sec2:1dHCLS} with the time unit adjusted. The operator $\mathcal{S}^{\lambda}_t$ adheres to the following evolution-invariant property:
\begin{equation*}
	\mathcal{S}^{\lambda}_t=\mathcal{S}_{\lambda t}\quad\forall \lambda>0,~\forall t>0,
\end{equation*}
where $\mathcal{S}_t$ denotes the solution operator of equation \eqref{sec2:1dHCLS}.
	
\begin{theorem}[Evolution invariance]\label{thm:EI}
	Let $\mathcal{E}^{\lambda}_n$ denote the solution operator of the OE-HWENO schemes solving $u_t+\lambda f(u)_x=0$ for $\lambda>0$ on a fixed mesh with a same CFL number, namely, $\mathcal{E}^{\lambda}_n( {\bm{U}}^0_i)= {\bm{U}}^n_i$ at time $\tau^\lambda_n=n\tau_\lambda$ where $\tau_\lambda=\tau/\lambda$ is the time step-size. Then we have 
	\begin{equation}\label{sec2:EI_pro}
		\mathcal{E}^\lambda_n=\mathcal{E}_n\quad\forall \lambda>0,
	\end{equation}
	where $\mathcal{E}_n$ denotes the solution operator of the OE-HWENO schemes solving equation \eqref{sec2:1dHCLS}.
\end{theorem}
\begin{proof}
	It is easy to verify that each RK update in the scale-invariant HWENO schemes is evolution-invariant, namely, it is 
	identical regardless of the value of $\lambda$. The OE procedure also maintains the evolution invariance. This is because $\alpha=\lambda\big|f'(\bar{{u}}_i)\big|$ for $u_t+\lambda f(u)_x=0$, and we have
	\begin{equation*}
		\mathcal{F}^{\tau_\lambda}_{{\rm OE},\lambda}=\mathcal{F}_{{\rm OE},1}^{\lambda\tau_\lambda}=\mathcal{F}^{\tau}_{{\rm OE},1}\quad\forall\lambda>0,
	\end{equation*}   
	where $\mathcal{F}^{\tau_\lambda}_{{\rm OE},\lambda}$ denotes the OE operator for $u_t+\lambda f(u)_x=0$ with the time step-size $\tau_\lambda$. 
	Consequently, both the RK update and the OE procedure yield identical outcomes independent on the value
	of $\lambda$. Hence, the OE-HWENO schemes satisfy the evolution invariance \eqref{sec2:EI_pro}.
\end{proof}

\subsubsection{Analysis of approximate dispersion relation (ADR)}\label{sec:ADR}

To investigate the dispersion and dissipation (spectral) properties of nonlinear shock-capturing schemes, Pirozzoli introduced a numerical methodology \cite{Pirozzoli} to study the ADR for a linear wave equation. The ADR serves as a tool to predict the spectral characteristics of a general nonlinear scheme. This is accomplished by evolving a single Fourier mode in a $2\pi$-periodic domain for a linear advection equation with constant velocity over a short period, followed by analyzing the solution in the Fourier space to determine the spectral property of the scheme. 
In this subsection, we will analyze the spectral property of the sixth-order OE-HWENO scheme using the ADR technique.

Consider the linear advection equation with $f(u)=u$ in \eqref{sec2:1dHCLS}, with monochromatic sinusoidal initial conditions of wavelength $\lambda$ (and wavenumber $\kappa =2\pi/\lambda$), 
\begin{equation*}%
	\begin{cases}
		u_t + u_x=0,~~~~-\infty<x<+\infty,~ t>0,
		\\
		u(x,0)=\hat{u}_0e^{\mathrm{i}\kappa x},-\infty<x<+\infty,
	\end{cases}
\end{equation*}
where $\mathrm{i}=\sqrt{-1}$ is the imaginary unit.
Consider a uniform grid with $x_j=jh_x,j=0,\cdots,N$ in a $2\pi$-periodic domain, where $h_x=\frac{2\pi}{N}$ is the space step-size. Following \cite{Pirozzoli}, one can define a modified wavenumber
\begin{equation*}
	\Phi(\phi)= \frac{\mathrm{i}h_x}{t} \ln
	\left(\frac{\hat{u}(\phi; t)}{\hat{u}_0}\right), 
\end{equation*}
where $\hat{u}(\phi; t)$ is the Fourier spectrum of the computed solution at a vary short time $t\rightarrow0$, with the given reduced wavenumber $\phi=\kappa h_x$. 
In practice, the above procedure is repeated for the $\phi_n=nh_x\le\pi,n=0,\ldots,\frac{N}{2}$ to obtain the corresponding modified wavenumber  $\Phi(\phi_n)$ defined by
\begin{equation*}
	\Phi\left(\phi_{n}\right)= \frac{\mathrm{i}h_x}{t} \ln
	\left(\frac{\hat{u}(\phi_{n} ; t)}{\hat{u}\left(\phi_{n};0\right)}\right), 
\end{equation*}
where $\hat{u}(\phi_{0} ; t)$ and $\hat{u}(\phi_{n} ; t)$ are computed by the discrete Fourier transform of the computed solution at $\phi_{n}$, for example,
\begin{equation}\label{DFT}
	\widehat{{u}}\left(\phi_{n} ; t\right)=
	\frac{1}{N} \sum_{j=0}^{N-1} {u}_{j}(t) \mathrm{e}^{-\mathrm{i}j \phi_{n}}.
\end{equation}

\begin{figure}[!thb]
	\centering
	\begin{subfigure}{0.48\textwidth}
		{\includegraphics[width=7.5cm,angle=0]{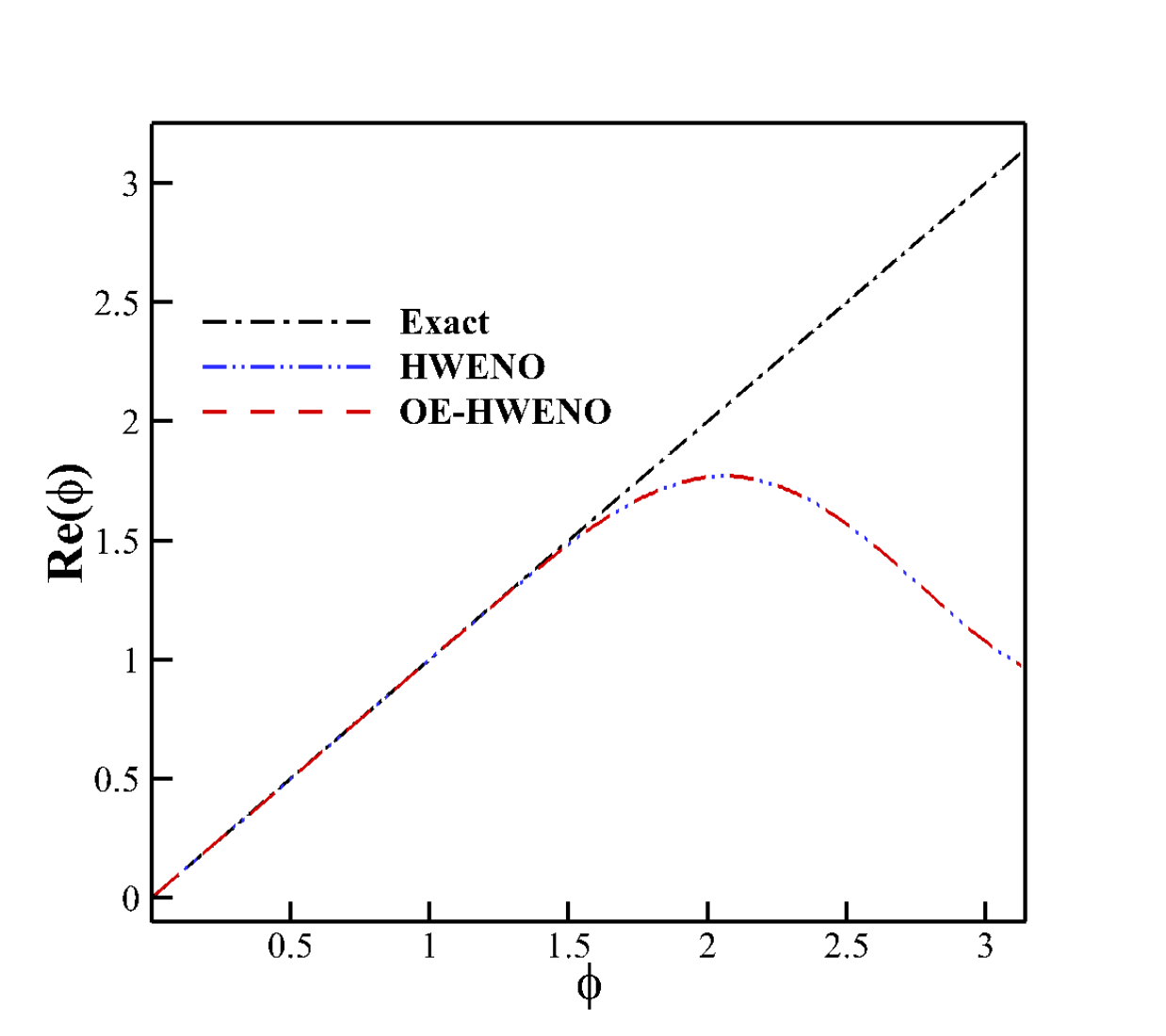}}
		\caption{Dispersion}
	\end{subfigure} 
	\begin{subfigure}{0.48\textwidth}
		{\includegraphics[width=7.5cm,angle=0]{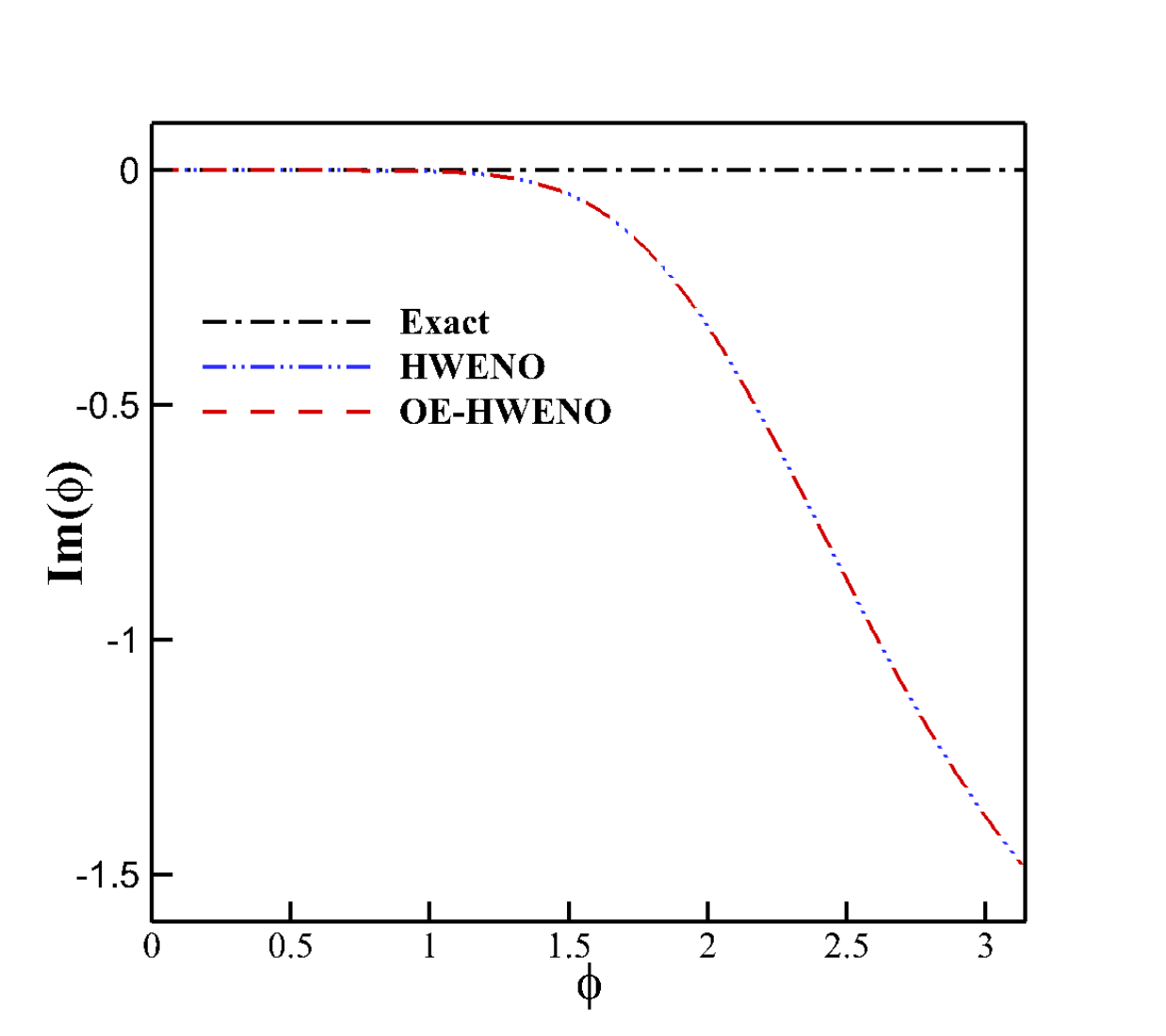}} 
		\caption{Dissipation}
	\end{subfigure} 
	\caption{The dispersion and dissipation errors of HWENO and OE-HWENO schemes. 
	} \label{sec3:ADR_HWENO}
\end{figure}  
Following the above process, one can derive the ADR to analyze the spectral property of nonlinear schemes. While spectral schemes have $\Phi(\phi)=\phi$, this generally does not hold for nonlinear shock-capturing schemes. The imaginary part $\mathbf{Im}(\Phi(\phi_n))$ represents the dissipation properties, while the real part $\mathbf{Re}(\Phi(\phi_n))$ represents the dispersion properties. Therefore, simple harmonic waves in the wave packet may have different wave velocities. As a result, the overall waveform may change with time, leading to numerical oscillations.

To obtain the ADR for the OE-HWENO scheme, we first use the corresponding HWENO reconstruction to obtain the point value ${u}_{j}$ in each cell, based on the cell-average of the computed solution at each $\phi_{n}$. We then utilize the aforementioned procedure \eqref{DFT} to determine the dispersion and dissipation properties of the OE-HWENO scheme. Fig.~\ref{sec3:ADR_HWENO} presents the dispersion and dissipation errors of the HWENO scheme without the OE procedure and the proposed OE-HWENO scheme. These two schemes exhibit nearly identical results, demonstrating  
that our OE procedure does not influence the spectral property of the HWENO schemes for smooth solutions.

\subsection{Two-dimensional OE-HWENO method}\label{sec:2D}   
Consider the 2D scalar conservation law
\begin{equation}\label{sec2:2dHCLS}
	\begin{cases}
		u_t+ f(u)_x+g(u)_y=0,~(x,y, t)\in \Omega\times [0,T], \\
		u(x,y,0)=u_0(x,y),~~~~~(x,y)\in \Omega. \\
	\end{cases}
\end{equation}
A uniform partition of the domain $\Omega=[a,b]\times[c,d]$ is defined as $\Omega=\cup^{N_x,N_y}_{i=1,j=1}I_{i,j}$ with $~I_{i,j}=I_i \times I_j$, $I_i=[x_{i-\frac12},x_{i+\frac12}]$, $I_j=[y_{j-\frac12},y_{j+\frac12}]$, $h_x=x_{i+\frac12}-x_{i-\frac12}$, and $h_y=y_{j+\frac12}-y_{j-\frac12}$. Define $(x_i,y_j)$ as the center of cell $I_{i,j}$ with $x_i=\frac12(x_{i-\frac12}+x_{i+\frac12})$ and $y_j=\frac12(y_{j-\frac12}+y_{j+\frac12})$.
Following \eqref{sec:semi} with $\phi^{(0)}_{i,j}(x,y)=\frac{1}{h_xh_y}$, $\phi^{(1)}_{i,j}(x,y)=\frac{x-x_i}{(h_x)^2h_y}$ and $\phi^{(2)}_{i,j}(x,y)=\frac{y-y_j}{h_x(h_y)^2}$, we obtain
\begin{align*}
	&\begin{aligned}
		&\frac{{\rm d}}{{\rm d} t}\left(\int_{I_{i,j}}u(x,y,t)\frac{1}{h_xh_y} \mathrm{d}x\mathrm{d}y\right)=
		-\frac{1} {h_x h_y}  \int_{I_j}\bigg[ f\big(u(x_{i+\frac12},y,t)\big)- f\big(u(x_{i-\frac12},y,t)\big)\bigg]\mathrm{d}y
		\\&\quad\quad\quad\quad\quad\quad\quad\quad\quad\quad\quad~~ 
		-\frac{1} {h_xh_y}  \int_{I_i}\big[ g\big(u(x,y_{j+\frac12},t)\big)- g\big(u(x,y_{j-\frac12},t)\big)\big]\mathrm{d}x,
	\end{aligned}
	\\		 
	&\begin{aligned}
		&\frac{{\rm d}}{{\rm d} t}\left(\int_{I_{i,j}}u(x,y,t)\frac{x-x_i}{(h_x)^2h_y} \mathrm{d}x\mathrm{d}y\right)=
		-\frac{1} {2h_xh_y}  \int_{I_j}\bigg[ f\big(u(x_{i-\frac12},y,t)\big) + f\big(u(x_{i+\frac12},y,t)\big)\bigg]\mathrm{d}y
		\\& 
		+\frac{1}{(h_x)^2 h_y}\int_{I_{i,j}}f(u)\mathrm{d}x\mathrm{d}y
		-\frac{1}{h_xh_y}\int_{I_i}\bigg[g\big(u(x,y_{j+\frac12},t)\big)-g\big(u(x,y_{j-\frac12},t)\big)\bigg]\frac{(x-x_i)}{h_x}\mathrm{d}x,
	\end{aligned}
	\\ 
	&\begin{aligned}
		&\frac{{\rm d}}{{\rm d} t}\left(\int_{I_{i,j}}u(x,y,t)\frac{y-y_j}{h_x(h_y)^2} \mathrm{d}x\mathrm{d}y\right)=
		-\frac{1}{h_xh_y}  \int_{I_j}\bigg[ f\big(u(x_{i+\frac12},y,t)\big)- f\big(u(x_{i-\frac12},y,t)\big)\bigg]\frac{(y-y_j)}{h_y}\mathrm{d}y
		\\&
		\quad-\frac{1} {2h_xh_y}  \int_{I_i}\bigg[ g\big(u(x,y_{j-\frac12},t)\big)+ g\big(u(x,y_{j+\frac12},t)\big)\bigg]dx+\frac 1{h_x (h_y)^2}\int_{I_{i,j}}g(u)\mathrm{d}x\mathrm{d}y.
	\end{aligned}
\end{align*}		
Let $\bar{u}_{i,j}(t)$, $\bar{v}_{i,j}(t)$, and $\bar{w}_{i,j}(t)$ denote the approximations to the zeroth-order moment, the first-order moment in the $x$-direction, and the first-order moment in the $y$-direction of $u(x,y,t)$, respectively, that is,
\begin{align*}
	&\bar{u}_{i,j}(t)\approx\frac{1}{h_xh_y}\int_{I_{i,j}} u(x,y,t)\mathrm{d}x\mathrm{d}y,
	\\
	&\bar{v}_{i,j}(t)\approx\frac{1}{h_xh_y}\int_{I_{i,j}}u(x,y,t)\frac{x-x_i}{h_x} \mathrm{d}x\mathrm{d}y,
	\\
	&\bar {w}_{i,j}(t)\approx\frac{1}{h_xh_y}\int_{I_{i,j}}u(x,y,t)\frac{y-y_j}{h_y} \mathrm{d}x\mathrm{d}y.
\end{align*}
Similar to the 1D case, based on $\{\bar{u}_{i,j},\bar{v}_{i,j},\bar{w}_{i,j}\}$, the ($k$+1)th-order HWENO method constructs a piecewise polynomial solution $u_h(x,y,t)$ of degree $k$ to approximate the exact solution. Though approximating the fluxes $f\big(u(x_{i+\frac12},y,t)\big)$ and $g\big(u(x,y_{j+\frac12},t) \big)$ with suitable numerical fluxes, and evaluating the integrals over $I_i$, $I_j$, and $I_{i,j}$ with proper quadrature rules of sufficiently high-order accuracy,
we obtain the following semi-discrete 2D FV HWENO scheme:
\begin{equation}\label{sec2:2dSemi_HWENO}
	\left\{
	\begin{aligned}
		\frac{{\rm d} \bar{{u}}_{i,j}(t)}{{\rm d} t} =& -\frac{1}{h_x}\bigg((\hat{f}_1)_{i+\frac12,j}-(\hat{f}_1)_{i-\frac12,j}\bigg)
		-\frac{1}{h_y}\bigg((\hat{g}_1)_{i,j+\frac12}-(\hat{g}_1)_{i,j-\frac12}\bigg)
		=: \mathcal{L}^{(0)}_{i,j}(u_h(x,y,t)), \\
		\frac{{\rm d} \bar{{v}}_{i,j}(t)}{{\rm d} t}=& -\frac{1}{2h_x}\bigg((\hat{f}_1)_{i+\frac12,j}+(\hat{f}_1)_{i-\frac12,j}\bigg)
		+\frac{1}{h_x}\sum\limits_{\ell=1}^Q\sum\limits_{m=1}^Q{\omega}^{{\rm G}}_\ell{\omega}^{{\rm G}}_m f(u_h({x}_{i,\ell}^{{\rm G}},{y}_{j,m}^{{\rm G}},t))\\&
		-\frac{1}{h_y}\bigg((\hat{g}_2)_{i,j+\frac12}-(\hat{g}_2)_{i,j-\frac12}\bigg)
		=: \mathcal{L}^{(1)}_{i,j}(u_h(x,y,t)), \\
		\frac{{\rm d} \bar{{w}}_{i,j}(t)}{{\rm d}t}=&
		-\frac{1}{h_x}\bigg((\hat{f}_2)_{i+\frac12,j}-(\hat{f}_2)_{i-\frac12,j}\bigg)
		-\frac{1}{2h_y}\bigg((\hat{g}_1)_{i,j+\frac12}+(\hat{g}_1)_{i,j-\frac12}\bigg)\\&
		+\frac{1}{h_y}\sum\limits_{\ell=1}^Q\sum\limits_{m=1}^Q{\omega}^{{\rm G}}_\ell{\omega}^{{\rm G}}_m g(u_h({x}_{i,\ell}^{{\rm G}},{y}_{j,m}^{{\rm G}},t))  =: \mathcal{L}^{(2)}_{i,j}(u_h(x,y,t)), 
	\end{aligned}
	\right.
\end{equation}
which corresponds to \eqref{sec_semidiscrete} in the 2D scalar case. Here, $(\hat{f}_{\ell})_{i+\frac12,j}$ and $(\hat{g}_{\ell})_{i,j+\frac12}$ are defined as
\begin{equation}\label{eq:WKL03}
	\begin{aligned}
		(\hat{f}_1)_{i+\frac12,j}&=\sum\limits_{\ell=1}^{Q}{\omega}^{{\rm G}}_\ell \hat{f}\big(u^{-,{\rm G}}_{i+\frac12,\ell},u^{+,{\rm G}}_{i+\frac12,\ell} \big),\quad
		(\hat{f}_2)_{i+\frac12,j}=\sum\limits_{\ell=1}^{Q}{{\omega}^{\rm G}_\ell}  \hat{f}\big(u^{-,{\rm G}}_{i+\frac12,\ell},u^{+,{\rm G}}_{i+\frac12,\ell}\big)\frac{{y}^{{\rm G}}_{j,\ell}-y_j}{h_y},
		\\
		(\hat{g}_1)_{i,j+\frac12}&=\sum\limits_{\ell=1}^{Q}{\omega}_\ell^{{\rm G}} \hat{g}\big(u^{-,{\rm G}}_{\ell,j+\frac12},u^{+,{\rm G}}_{\ell,j+\frac12}\big),\quad
		(\hat{g}_2)_{i,j+\frac12}=\sum\limits_{\ell=1}^{Q}{\omega}_\ell^{{\rm G}} \hat{g}\big(u^{-,{\rm G}}_{\ell,j+\frac12},u^{+,{\rm G}}_{\ell,j+\frac12}\big)\frac{{x}^{{\rm G}}_{i,\ell}-x_i}{h_x},
	\end{aligned}
\end{equation}
where $u^{-,{\rm G}}_{i+\frac12,\ell}$ and $u^{+,{\rm G}}_{i+\frac12,\ell}$ represent the left and right limits of $u_h(x,y)$ at $(x_{i+\frac12},{y}_{j,\ell}^{{\rm G}})$; $u^{-,{\rm G}}_{\ell,j+\frac12}$ and $u^{+,{\rm G}}_{\ell,j+\frac12}$ denote the left and right limits of $u_h(x,y)$ at $({x}^{{\rm G}}_{i,\ell},y_{j+\frac12})$.  
Here, $\{x^{\rm G}_{i,\ell}\}^Q_{\ell=1}$ and $\{y^{\rm G}_{j,\ell}\}^Q_{\ell=1}$ represent the Gauss quadrature nodes in $I_{i}$ and $I_{j}$, respectively, with the normalized weights $\{\omega^{\rm G}_{\ell}\}^Q_{\ell=1}$ satisfying $\sum^Q_{\ell=1}\omega^{\rm G}_{\ell}=1$. For instance, in the sixth-order HWENO scheme, we use the Gauss quadrature with
\begin{equation*}
	\begin{aligned}
		&{x}^{{\rm G}}_{i,1}={x}_{i-\frac{\sqrt{15}}{10}},~{x}^{{\rm G}}_{i,2}={x}_i,~{x}^{{\rm G}}_{i,3}={x}_{i+\frac{\sqrt{15}}{10}},\\
		&{y}^{{\rm G}}_{j,1}={y}_{j-\frac{\sqrt{15}}{10}},~{y}^{{\rm G}}_{j,2}={y}_j,~{y}^{{\rm G}}_{j,3}={y}_{j+\frac{\sqrt{15}}{10}},\\
		&{\omega}^{{\rm G}}_{1}=\frac5{18},~{\omega}^{{\rm G}}_{2}=\frac4{9},~{\omega}^{{\rm G}}_{3}=\frac5{18}.
	\end{aligned}
\end{equation*}
For the numerical fluxes $\hat{f}\big(u^-,u^+\big)$ and  $\hat{g}\big(u^-,u^+\big)$ in \eqref{eq:WKL03}, we use the Lax--Friedrichs flux:
\begin{align*}
	&\hat{f}\big(u^-,u^+\big)=
	\frac12\Big( f(u^-)+f(u^+)-\alpha_x(u^+-u^-)\Big),
	\\
	&\hat{g}\big(u^-,u^+\big)=
	\frac12\Big( g(u^-)+g(u^+)-\alpha_y(u^+-u^-)\Big),
\end{align*}
where $\alpha_x=\max_{i,j}|f'(\bar{u}_{i,j})|$ and $\alpha_y=\max_{i,j}|g'(\bar{u}_{i,j})|$.  

Let $\bm{U}_{i,j}(t)=(\bar{u}_{i,j}(t),\bar{v}_{i,j}(t),\bar{w}_{i,j}(t))^\top$ and $\bm{\mathcal{L}}_{i,j}(u_h)=(\mathcal{L}_{i,j}^{(0)}(u_h)$, $\mathcal{L}_{i,j}^{(1)}(u_h),\mathcal{L}_{i,j}^{(2)}(u_h))$. Then the semi-discrete HWENO scheme \eqref{sec2:2dSemi_HWENO} can be rewritten as
\begin{equation*}
	\frac{\rm d}{{\rm d}t}\bm{U}_{i,j}(t)=\bm{\mathcal{L}}_{i,j}(u_h),
\end{equation*}
which can also be further discretized in time using a RK method. By introducing an OE procedure after each RK stage to suppress spurious oscillations, we obtain the 2D OE-HWENO schemes. For instance, the 2D OE-HWENO scheme, coupled with the third-order explicit SSP RK method, is 
\begin{equation}\label{sec2:2dHWENO}
	\left\{
	\begin{aligned}
		{\bm{U}}^{\sigma,0}_{i,j}=&~{\bm{U}}^{n}_{i,j},\\
		\bm{U}^{n,1}_{i,j}=&~{\bm{U}}^{\sigma,0}_{i,j}+\Delta t \bm{\mathcal{L}}_{i,j}(u^{\sigma,0}_h),\quad\quad\quad\quad\quad~~~~
		\bm{U}^{\sigma,1}_{i,j} = \mathcal{F}_{\rm OE} \{\bm{U}^{n,1}_{\kappa} \}_{\kappa\in\Lambda_{i,j}},
		\\ 
		\bm{U}^{n,2}_{i,j}=&~\frac34{\bm{U}}^{\sigma,0}_{i,j}+\frac14( {\bm{U}}^{\sigma,1}_{i,j} +\Delta t \bm{\mathcal{L}}_{i,j}(u^{\sigma,1}_h) ),\quad 
		\bm{U}^{\sigma,2}_{i,j} = \mathcal{F}_{\rm OE} \{\bm{U}^{n,2}_{\kappa} \}_{\kappa\in\Lambda_{i,j}},
		\\ 
		\bm{U}^{n,3}_{i,j}=&~\frac13{\bm{U}}^{\sigma,0}_{i,j}+\frac23( {\bm{U}}^{\sigma,2}_{i,j} +\Delta t \bm{\mathcal{L}}_{i,j}(u^{\sigma,2}_h) ), \quad 
		\bm{U}^{\sigma,3}_{i,j} = \mathcal{F}_{\rm OE} \{\bm{U}^{n,3}_{\kappa} \}_{\kappa\in\Lambda_{i,j}},
		\\
		{\bm{U}}^{n+1}_{i,j}  = &~ {\bm{U}}^{\sigma,3}_{i,j}. 
	\end{aligned}
	\right.
\end{equation}
where ${u}^{\sigma,\ell}_h (x) = {\mathcal H} \{  {\bm{U}}^{\sigma,\ell}_{i,j} \}$, $\ell=0,1,2$; the operator $\mathcal{H}$ denotes the standard 2D HWENO reconstruction based on the values of $\{{\bm{U}}^{\sigma,\ell}_{i,j}\}$; and the operator $\mathcal{F}_{\rm OE}$ denotes the 2D OE procedure
with $\Lambda_{i,j} = \{(i+m,j+n),-1\le m,n\le1\}$.
The operators $\mathcal{H}$ and $\mathcal{F}$ will be introduced in Subsections \ref{sec:2dH} and \ref{sec:2dF}, respectively.	   
	
\subsubsection{HWENO Operator $\mathcal{H}$}\label{sec:2dH} 
Taking the 2D sixth-order OE-HWENO method as an example, the piecewise polynomial functions $\{u^{\sigma,\ell}_h\}^2_{\ell=0}$ in \eqref{sec2:2dHWENO} are reconstructed as follows: based on the values values $\{\bar{{u}}^{\sigma,\ell}_{i+m,j+n}$, $\bar{{v}}^{\sigma,\ell}_{i+m,j+n}$, $\bar{{w}}^{\sigma,\ell}_{i+m,j+n}\}_{-1\le m,n\le 1}$, we construct a quintic polynomial $p_0(x,y)$, a cubic polynomial $p_1(x,y)$ and four linear polynomials $\{p_m(x,y)\}^5_{m=2}$ in $I_{i,j}$. Then we compute the smoothness indicators $\{\beta_m\}^5_{m=0}$ of functions $\{p_m(x,y)\}^5_{m=0}$ in $I_{i,j}$. Finally, thought the nonlinear weights of HWENO reconstruction, we obtain
\begin{equation}\label{sec3:2dHWENO_rec}
	{u}^{\sigma,\ell}_h (x) = {\mathcal H} \{  {\bm{U}}^{\sigma,\ell}_{i,j} \}:=
	\omega^\text{H}_0\bigg(\frac{1}{\omega^\text{H}_0}p_0(x,y) - \frac{\omega^\text{H}_1}{\omega^\text{H}_0}\tilde{q}_1(x,y) \bigg) + \gamma^\text{H}_1\tilde{q}_1(x,y)\quad \forall (x,y)\in I_{i,j},
\end{equation}
where $\tilde{q}_1(x,y)=\omega^\text{L}_1\big(\frac{1}{\gamma^\text{L}_1}p_1(x,y) - \sum\limits_{m=2}^5\frac{\gamma^\text{L}_m}{\gamma^\text{L}_1}p_m(x,y) \big) + \sum\limits_{m=2}^5\omega^\text{L}_mp_m(x,y)$; $\{\gamma^\text{H}_m\}^1_{m=0}$ and $\{\gamma^\text{L}_m\}^5_{m=2}$ are arbitrary positive linear weights with $\sum_{m=0}^{1}\gamma^\text{H}_m=1$ and $\sum_{m=1}^{5}\gamma^\text{L}_m=1$; $\{\omega^\text{H}_m\}^1_{m=0}$ and $\{\omega^\text{L}_m\}^5_{m=2}$ are nonlinear weights. For the convenience of the reader, the detailed procedure of the 2D HWENO reconstruction \eqref{sec3:2dHWENO_rec} is provided in \ref{sec:A_2dHWENO}.

As the 1D case, the standard 2D HWENO operator $\mathcal{H}$ in \eqref{sec3:2dHWENO_rec} not scale-invariant. To this end, we propose the 2D dimensionless HWENO operator $\mathcal{H_D}$ to replace $\mathcal{H}$. The definition of $\mathcal{H_D}$ is based on a dimensionless transformation $\mathcal{D}$, similar to Definition \ref{sec3:defH_AveMaxMin} in the 1D case. 
\begin{definition}[2D dimensionless transformation $\mathcal{D}$] 
	The dimensionless transformation $\mathcal{D}$ is an affine transformation that normalizes the involved quantities by 
	the average, maximum, and minimum of the solution values in a stencil. For all the zeroth- and first-order moments $\{ \bar{u}_{\kappa}, \bar{v}_{\kappa},\bar{w}_{\kappa} \}$ in the stencil $\Lambda_{i,j}$,  the transformation $\mathcal{D}: \{ \bar{u}_{\kappa}, \bar{v}_{\kappa},\bar{w}_{\kappa} \} \rightarrow \{ \widehat{u}_{\kappa}, \widehat{v}_{\kappa},\widehat{w}_{\kappa} \} $ is defined by 
	\begin{equation}\label{eq:D_2}
		\widehat{u}_{\kappa}=\frac{\bar{u}_{\kappa}-u_{\rm{ave}}}{u_{\rm{max}}-u_{\rm{min}}+\epsilon}, \quad 
		\widehat{v}_{\kappa}=\frac{\bar{v}_{\kappa}}{u_{\rm{max}}-u_{\rm{min}}+\epsilon}, \quad 
		\widehat{w}_{\kappa}=\frac{\bar{w}_{\kappa}}{u_{\rm{max}}-u_{\rm{min}}+\epsilon} \quad \forall {\kappa} \in \Lambda_{i,j}
	\end{equation}
	with
	\begin{equation*} 
		u_{\rm ave}=\frac{1}{\# \Lambda_{i,j}}\sum_{{\kappa} \in \Lambda_{i,j}}|\bar{u}_{\kappa}|,\quad 
		u_{\rm max}=\max_{{\kappa} \in \Lambda_{i,j}}\{\bar{u}_{\kappa}\},\quad 
		u_{\rm min}=\min_{{\kappa} \in \Lambda_{i,j}}\{\bar{u}_{\kappa}\},
	\end{equation*}
	where $\# \Lambda_{i,j}$ denotes the number of cells in $\Lambda_{i,j}$.
\end{definition} 

\begin{definition}[2D dimensionless HWENO operator $\mathcal{H_D}$]\label{sec3:defH_AveMaxMin_2} 
	The dimensionless HWENO operator is defined as 
	\begin{equation}\label{eq:HD_2}
		\mathcal{H_D}:=\mathcal{D}^{-1}_{u} \mathcal{H}\mathcal{D},
	\end{equation}
	where $\mathcal{D}: \{ \bar{u}_{\kappa}, \bar{v}_{\kappa},\bar{w}_{\kappa} \} \rightarrow \{ \widehat{u}_{\kappa}, \widehat{v}_{\kappa},\widehat{w}_{\kappa} \} $ 
	is the dimensionless transformation defined by \eqref{eq:D_2}, 
	$\mathcal{H}$ is the standard HWENO operator which maps to the normalized moments $\{ \widehat{u}_{\kappa}, \widehat{v}_{\kappa},\widehat{w}_{\kappa} \}_{{\kappa} \in \Lambda_{i,j}}$ to the dimensionless reconstructed solution $\widehat u_h(x,y)\big|_{I_{i,j}} = \widehat u_{h,i,j}(x,y)$, and 
	$\mathcal{D}^{-1}_{u}$ is the ``inverse'' transformation
	defined as 
	\begin{equation*} 
		u_{h,i,j}(x,y) = \mathcal{D}^{-1}_{u} \widehat u_{h,i,j}(x,y) = (u_{\rm{max}}-u_{\rm{min}}+\epsilon)  \widehat u_{h,i,j}(x,y)  +u_{\rm{ave}}. 
	\end{equation*}
\end{definition}

According to the above definitions, the dimensionless HWENO operator $\mathcal{H_D}$ defined in \eqref{eq:HD_2} is scale-invariant, with the proof similar to Theorem \ref{sec2:thm_si} and thus  omitted here.
	
\subsubsection{OE operator $\mathcal{F}$}\label{sec:2dF}

Now we will introduce the OE procedure  $\bm{U}^{\sigma,\ell+1}_{i,j}=\mathcal{F}_{\rm OE}\{\bm{U}^{n,\ell+1}_{\kappa}\}_{\kappa\in\Lambda_{i,j}}$  $(\ell=0,1,2)$ in \eqref{sec2:2dHWENO}. For simplicity, we define $\bm{U}^{\sigma}_{i,j}=\bm{U}^{\sigma,\ell+1}_{i,j}$ and $\bm{U}_{i,j}=\bm{U}^{n,\ell+1}_{i,j}$ in the following.
The OE modified moments $\bm{U}^{\sigma}_{i,j}$ are defined as the zeroth-order and first-order moments of ${u}_\sigma({x,y},{\Delta t})$, where ${u}_\sigma({x,y},\tau) \in \mathbb V_h^k$ ($0 \le \tau \le \Delta t$) represents the solution to the following damping equations:
\begin{equation}\label{sec2:OE2d_ODE}
	\left\{
	\begin{aligned}
		&\frac{\mathrm{d} }{\mathrm{d} \tau} \int_{I_{i,j}} u_\sigma  \phi \mathrm{d}x\mathrm{d}y + \sigma_{i,j}(u_h^*)\int_{I_{i,j}}(u_\sigma-P^{0}u_\sigma)\phi \mathrm{d}x\mathrm{d}y= 0\quad\forall \phi\in\mathbb{P}^{1}(I_{i,j}),
		\\
		&u_\sigma(x,y,0) = u_h^*(x,y)=\Pi_h\{\bm{U}_{i,j}\}, 
	\end{aligned}
	\right. 
\end{equation}
where the operator $\Pi_h$ denotes the $(k+1)$th-order linear Hermite reconstruction, and the damping
coefficient  
\[
\sigma_{i,j}(u_h^*)=  {\frac{\alpha_x}{h_x}\widehat{\sigma}_{i,j}(u_h^*) +   \frac{\alpha_y}{h_y}\widetilde{\sigma}_{i,j}(u_h^*).}
\]
To maintain the scale-invariant property of damping strength, we define $\widehat{\sigma}_{i,j}(u_h^*)$ and  $\widetilde{\sigma}_{i,j}(u_h^*)$ as 
\begin{equation}\label{sec3:sigma_2d}
	\begin{aligned}
		&
		\widehat{\sigma}_{{i,j}}(u_h^*)=
		\begin{cases}
			\displaystyle 
			0,\qquad &\text{if}~\max\limits_{{1\le i \le N_x}, {1\le j\le N_y}}|\bar{u}_{i,j} -\overline{u}_{\Omega} |=0,
			\\ 
			\displaystyle 
			\frac{\sum\limits_{m\in\{0,1\}} \left( h_x^m \big| [\![\partial^m_x u_h^*]\!]_{i-\frac12,j}\big|+\big|[\![\partial^m_x u_h^*]\!]_{i+\frac12,j}\big| \right) }{ \max_{1\le i \le N_x, 1\le j\le N_y}|\bar{u}_{i,j} -\bar{u}_{\Omega} |},~&\mbox{otherwise}.
		\end{cases} 
		\\&
		\widetilde{\sigma}_{{i,j}}(u_h^*)=
		\begin{cases}
			\displaystyle 
			0,\qquad&\text{if}~\max\limits_{1\le i \le N_x, 1\le j\le N_y}|\bar{u}_{i,j} -\overline{u}_{\Omega} |=0,
			\\
			\displaystyle  
			\frac{ \sum\limits_{m\in\{0,1\}} \left( h_y^m \big| [\![\partial^m_y u_h^*]\!]_{i,j-\frac12}\big|+\big|[\![\partial^m_y u_h^*]\!]_{i,j+\frac12}\big| \right) }{ \max_{1\le i \le N_x, 1\le j\le N_y}|\bar{u}_{i,j} -\overline{u}_{\Omega} |},~&\mbox{otherwise}.
		\end{cases} 
	\end{aligned}
\end{equation} 
where $[\![ \partial^m_x u_h^* ]\!]_{i+\frac12,j}=\partial^m_x u_h^*(x_{i+\frac12}^+,y_j)-\partial^m_x u_h^*(x_{i+\frac12}^-,y_j)$ denotes  the jumps of $\partial^m_x u_h^*$ across the interface at  $(x_{i+\frac12},y_j)$,
$[\![ \partial^m_y u^*_h ]\!]_{i,j+\frac12}=\partial^m_y u_h^*(x_{i},y_{j+\frac12}^+)-\partial^m_y u_h^*(x_{i},y_{j+\frac12}^-)$ denotes  the jumps of $\partial^m_y u_h^*$ across the interface at  $(x_{i},y_{j+\frac12})$, and $\overline{u}_{\Omega}$ represents the average of $u^*_h$ over the whole domain $\Omega$, namely, for uniform meshes,
\[
\overline{u}_{\Omega}=\frac{1}{N_xN_y}\sum_{i=1}^{N_x}\sum_{j=1}^{N_y}\bar{u}_{i,j}.
\]
Note that $\max_{{1\le i \le N_x}, {1\le j\le N_y}}|\bar{u}_{i,j} -\overline{u}_{\Omega} |$ in \eqref{sec3:sigma_2d} is a global constant over all the cells and is computed only once in each OE step.

Since the damping coefficient $\sigma_{i,j}({u}_h^*)$ only depends on the ``initial" value 
${u}_\sigma(x,y,0) = {u}_h^* (x,y)$, the damping equations \eqref{sec2:OE2d_ODE} are essentially a linear system of ODEs and are exactly solvable without requiring discretization. As the 1D case, we find that the final expression of the OE procedure can be formulated in a simple form without implementing the linear Hermite reconstruction ${\Pi_h}$. Specifically, we have the following conclusion.

\begin{theorem}[Exact solver of 2D OE procedure]\label{thm:2dOE}
	Denote $(\bar{u}_{i,j}^\sigma,\bar{v}_{i,j}^\sigma,\bar{w}_{i,j}^\sigma):=\bm{U}^\sigma_{i,j}$. The OE procedure $ \bm{U}^{\sigma}_{i,j}=\mathcal{F}_{\rm OE}\{\bm{U}_{\kappa}\}_{\kappa\in\Lambda_{i,j}} $ can be exactly solved and explicitly expressed as
	\begin{equation*} 
		\bar{u}_{i,j}^{\sigma}=\bar{u}_{i,j}, \quad 
		\bar{v}_{i,j}^{\sigma}=\bar{v}_{i,j} \delta, \quad 
		\bar {w}_{i,j}^{\sigma}=\bar{w}_{i,j} \delta 
	\end{equation*}
	with 		
    \begin{equation}\label{key098}
    	\delta := \exp\bigg(-\alpha_x\frac{{\Delta t}}{h_x} \widehat{\sigma}_{i,j}(u_h^{*}) -\alpha_y\frac{{\Delta t}}{h_y} \widetilde{\sigma}_{i,j}(u_h^{*})\bigg),
    \end{equation} 
	where the coefficients $\widehat{\sigma}_{i,j}(u_h^{*})$ and $\widetilde{\sigma}_{i,j}(u_h^{*})$ are defined in \eqref{sec3:sigma_2d}. For the 2D sixth-order OE-HWENO scheme, the jumps $[\![\partial^m_x u_h^{*}]\!]_{i+\frac12,j}$ and $[\![\partial^m_y u_h^{*}]\!]_{i,j+\frac12}$ are expressed as 
	\begin{equation} \label{eq:2Djump}
		\left\{
		\begin{aligned}  
			&[\![\partial^m_x u_h^{*}]\!]_{i+\frac12,j}=\frac{1}{h_x^m}
			\left(
			\left \langle {\bf A}^{(m)},{\bf U}_x\right \rangle +
			\left \langle {\bf B}^{(m)},{\bf V}_x\right \rangle + 
			\left \langle {\bf C}^{(m)},{\bf W}_x\right \rangle 
			\right),
			\\
			&[\![\partial^m_y u_h^{*}]\!]_{i,j+\frac12}=\frac{1}{h_y^m}
			\left(
			\left \langle {\bf A}^{(m)},{\bf U}^{\top}_y\right \rangle +
			\left \langle {\bf C}^{(m)},{\bf V}^{\top}_y\right \rangle + 
			\left \langle {\bf B}^{(m)},{\bf W}^{\top}_y\right \rangle 
			\right),
		\end{aligned}
		\right. 
		\quad m=0,1,
	\end{equation}
	where  the constant matrices ${\bf A}^{(m)},{\bf B}^{(m)},{\bf C}^{(m)}$  
	are presented in \ref{sec:A3_Matrices}. The notation 
	$\left \langle\cdot,\cdot\right \rangle$ denotes the inner product of two matrices; ${\bf U}_x=[\bar{u}_{s,\ell}],{\bf V}_x=[\bar{v}_{s,\ell}],{\bf W}_x=[\bar{w}_{s,\ell}]\in\mathbb{R}^{4\times3}$ with $i-1\le s \le i+2$ and $j-1\le \ell \le j+1$;
	${\bf U}_y=[\bar{u}_{s,\ell}],{\bf V}_y=[\bar{v}_{s,\ell}],{\bf W}_y=[\bar{w}_{s,\ell}]\in\mathbb{R}^{3\times4}$ with $i-1\le s\le i+1$ and $j-1 \le \ell \le j+2$. 
	Notably, the above exact solver of the OE procedure only involves the values $\bar{u}_{i,j}$, $\bar{v}_{i,j}$ and $\bar{w}_{i,j}$ without the need to formulate the linear Hermite reconstruction $u^*_h$.
\end{theorem}
\begin{proof}
	Let $\{\phi^{(\ell)}_{i,j}(x,y)\}^k_{\ell=1}$ be a local orthogonal basis  of $\mathbb{P}^{k}(I_{i,j})$, for instance, the scaled 2D	Legendre polynomials:
	\begin{equation*}
		\begin{aligned}
			&\phi_{i,j}^{(0)}(x,y)=1,\quad \phi_{i,j}^{(1)}(x,y)=\xi_i,\quad \phi_{i,j}^{(2)}(x,y)= \eta_j, 
			\\&
			\phi_{i,j}^{(3)}(x,y)=\xi_i^2-\frac{1}{12}, \quad   \phi_{i,j}^{(4)}(x,y)=\xi_i\eta_j, \quad \phi_{i,j}^{(5)}(x,y)=\eta_j^2-\frac{1}{12}, \dots ,
		\end{aligned}
	\end{equation*} 
	where $\xi=\frac{x-x_i}{h_x}$and $\eta_j=\frac{y-y_j}{h_y}$.
	
	For the sixth-order linear Hermite reconstruction $\Pi_h\{\bm{U}_{i,j}\}$, the reconstructed solution $u_h^*(x,y)$ in the cell $I_{i,j}$ is a quintic polynomial
	\begin{equation}\label{eq:uhxystar}
		u_h^*(x,y)\Big|_{I_{i,j}} =: p_{0}^{(i,j)}(x,y)=\sum_{\ell=0}^k c_{0,\ell} \phi^{(\ell)}_{i,j}(x,y),
	\end{equation}
	where the coefficients $\{c_{0,\ell}\}$ are determined by matching the relations in \eqref{sec:p0(x,y)}, and the expressions of $\{c_{0,\ell}\}$ are listed in Table \ref{secA:coe2D} for the sixth-order OE-HWENO scheme.
	
	Assume that the solution $u_\sigma(x,y,\tau) \in \mathbb V_h^k$ of the damping equations \eqref{sec2:OE2d_ODE} can be expressed as
	\begin{equation*}
		u_\sigma(x,y,\tau)=\sum_{\ell=0}^k c_{\ell}(\tau) \phi^{(\ell)}_i(x,y) \qquad \forall (x,y) \in I_{i,j}, ~~ 0 \le \tau \le \Delta t, 
	\end{equation*} 
	with $c_{\ell} ( 0) = c_{0,\ell} $ and 
	\begin{equation*}
		\begin{aligned}
			&\bar{u}^\sigma_{i,j}= \frac{1}{h_xh_y}\int_{I_{i,j}} u_\sigma(x,y,\Delta t) \mathrm{d}x\mathrm{d}y,\\
			&\bar{v}^{\sigma}_{i,j}= \frac{1}{h_xh_y}\int_{I_{i,j}} u_\sigma(x,y,\Delta t) \frac{x-x_i}{h_x} \mathrm{d}x\mathrm{d}y = \frac{c_{1}(\Delta t)}{h_xh_y}  \int_{I_{i,j}} \left( \phi^{(1)}_i(x,y) \right)^2  \mathrm{d}x\mathrm{d}y,\\
			&\bar{w}^{\sigma}_{i,j}= \frac{1}{h_xh_y}\int_{I_{i,j}} u_\sigma(x,y,\Delta t) \frac{y-y_j}{h_y} \mathrm{d}x\mathrm{d}y = \frac{{c_{2}(\Delta t)}}{h_xh_y}  \int_{I_{i,j}} \left( \phi^{(2)}_i(x,y) \right)^2  \mathrm{d}x\mathrm{d}y,
		\end{aligned}
	\end{equation*} 
	Note that  
	\[
	({u}_\sigma-P^{0}{u}_\sigma)(x,y,\tau)=\sum^k_{\ell=1} c_{\ell}(\tau)\phi^{(\ell)}_{i}(x,y).
	\]
	Taking $\phi(x,y)= \phi_i^{(0)}(x,y) =\frac{1}{h_xh_y}$ in \eqref{sec2:OE2d_ODE} gives 
	\begin{align*} 
		&\frac{\mathrm{d}}{\mathrm{d}\tau} \left( \frac{1}{h_xh_y}\int_{I_{i,j}} u_\sigma(x,y,\tau) \mathrm{d}x\mathrm{d}y \right)  = - \sigma_{i}({u}_h^*) \int_{I_{i,j}}({u}_\sigma-P^{0}{u}_\sigma)  \phi \mathrm{d}{x}\mathrm{d}y
		\\
		& = {\frac{1}{h_xh_y}} \left(-\alpha_x\frac{{\Delta t}}{h_x} \widehat{\sigma}_{i,j}(u_h^*) -\alpha_y\frac{{\Delta t}}{h_y} \widetilde{\sigma}_{i,j}(u_h^*)\right) \int_{I_{i,j}}({u}_\sigma-P^{0}{u}_\sigma)   \mathrm{d}{x} \mathrm{d}y= 0, 
	\end{align*} 
	which yields
	\begin{equation*} 
		\bar{u}_{i,j}^{\sigma} = \frac{1}{h_xh_y}\int_{I_{i,j}} u_\sigma(x,y,\Delta t) \mathrm{d}x \mathrm{d}y= \frac{1}{h_xh_y}\int_{I_{i,j}} u_\sigma(x,y,0) \mathrm{d}x\mathrm{d}y  =\bar{u}_{i,j}.
	\end{equation*} 
	Taking $\phi(x,y)=\frac{x-x_i}{h_xh_y}$ in \eqref{sec2:OE2d_ODE}, we derive 
	\begin{equation}\label{sec2:OEstepVij}
		\frac{\mathrm{d} }{\mathrm{d} \tau} c_{1} (\tau) 
		+ c_{1} (\tau)  {\left( \frac{\alpha_x }{h_x}{\widehat  \sigma_{i,j}(u_h^*)}+  \frac{\alpha_y }{h_y}{\widetilde  \sigma_{i,j}(u_h^*)}\right)}
		= 0.
	\end{equation}
	Integrating \eqref{sec2:OEstepVij} from $\tau=0$ to $\Delta t$ gives
	$
	c_{1} (\Delta t) = c_1(0) \delta,
	$  
	or equivalently, 
	$
	\bar{v}_{i,j}^{\sigma} =\bar{v}_{i,j}\delta,
	$ 
	where $\delta$ is defined in \eqref{key098}. 
	Similarly, we have
	\begin{equation*}
		\bar{w}_{i,j}^{\sigma} =\bar{w}_{i,j} \delta. 
	\end{equation*}
	For the 2D sixth-order OE-HWENO scheme, 
	using \eqref{eq:uhxystar} and the expressions of $\{c_{0,\ell}\}$ listed in Table \ref{secA:coe2D}, we obtain the expressions for the jumps 
	$[\![\partial^m_x u_h^{*}]\!]_{i+\frac12,j}$ and $[\![\partial^m_y u_h^{*}]\!]_{i,j+\frac12}$ in \eqref{eq:2Djump}. 
	The proof is completed.
\end{proof}		
		 
Similar to the 1D case, the 2D OE procedure also possesses notable advantages, including stability (Remark \ref{sec2:rmk1}), conservation (Remark \ref{sec2:rmk2}), and efficiency and simplicity (Remark \ref{sec2:rmk3}). Moreover, we can prove that the 2D OE procedure retains the original high-order accuracy of the HWENO schemes, as demonstrated in Theorem \ref{thm:accuracy2d}.
		 
\begin{theorem}[Maintain accuracy]\label{thm:accuracy2d}
	Consider the $(k+1)$th-order 2D OE-HWENO scheme \eqref{sec2:2dHWENO} under a CFL condition $\alpha_x\frac{{\Delta t}}{h_x} + \alpha_y\frac{{\Delta t}}{h_y} \le C_\text{cfl}$ with $C_\text{cfl} < 1$. 
	Assume that the exact solution $u(x,y,t) \in C^{k+1}(\Omega)$ for a given $t \in [0,T]$, and the boundary conditions are periodic. 
	Let $\bm{U}_{i,j}^e=(\bar u_{i,j}^e, \bar v_{i,j}^e, \bar{w}_{i,j}^e)$ denote the exact zeroth- and first-order moments of $u(x,y,t)$ on cell $I_{i,j}$. 
	If $\bm{U}_{i,j}=(\bar u_{i,j}, \bar v_{i,j},\bar{w}_{i,j})$ are $(k+1)$th-order accurate approximations to 
	$\bm{U}_{i,j}^e$ for all ${i,j}$, then the OE modified moments $\bm{U}^{\sigma}_{i,j}=\mathcal{F}_{\rm OE}\{\bm{U}_{\kappa}\}_{\kappa\in\Lambda_{i,j}} $
	are also $(k+1)$th-order accurate approximations to 
	$\bm{U}_{i,j}^e$, namely, 
	\begin{equation*}
		\max_{1\le i\le N_x} \max_{1\le i\le N_y}	\big\| \bm{U}^{\sigma}_{i,j} - \bm{U}_{i,j}^e\big\| \lesssim h_x^{k+1}+h_y^{k+1}. 
	\end{equation*}
	This means the OE procedure maintains the original high-order accuracy of the HWENO schemes. 
\end{theorem}

The proof of Theorem \ref{thm:accuracy2d} is similar to that of Theorem \ref{thm:accuracy} and thus is omitted here.		 

\subsubsection{Extension to 2D hyperbolic systems}
For the OE-HWENO method to solve the 2D hyperbolic system of conservation laws $\bm{u}_t+\bm{f}(\bm{u})_x+\bm{g}(\bm{u})_y=0$, we propose the OE procedure by the following damping equations: 
\begin{equation}\label{eq:OE-2Dsystem}
	\left\{
	\begin{aligned}
		&\frac{\mathrm{d} }{\mathrm{d} \tau} \int_{I_{i,j}} \bm{u}_\sigma  \cdot\bm{\phi} \mathrm{d}x\mathrm{d}y + \sigma_{i,j}(\bm{u}_h^{*}) \int_{I_{i,j}}(\bm{u}_\sigma-P^{0}\bm{u}_\sigma)\cdot\bm{\phi} \mathrm{d}x\mathrm{d}y= 0 \quad\forall \bm{\phi}\in[\mathbb{P}^{1}(I_{i,j})]^N,
		\\
		&\bm{u}_\sigma(x,y,0) = \bm{u}_h^{*}(x,y) =\Pi_h\{\bm{U}_{i,j}\},
	\end{aligned}
	\right. 
\end{equation}
where $\bm{U}_{i,j}=(\overline{\bm u}_{i,j},\overline{\bm v}_{i,j},\overline{\bm w}_{i,j})$; 
the damping coefficient $\sigma_{i,j}(\bm{u}_h^{*})= {\frac{\alpha_x}{h_x}\widehat{\sigma}_{i,j}(\bm{u}_h^{*}) +  \frac{\alpha_y}{h_y}\widetilde{\sigma}_{i,j}(\bm{u}_h^{*})}$; $\alpha_x$ and $\alpha_y$ denote the (estimated) maximum wave speeds in the $x$- and $y$-directions, respectively. Here,  $\widehat{\sigma}_{i,j}(\bm{u}_h^{*})$ and $\widetilde{\sigma}_{i,j}(\bm{u}_h^{*})$ are defined as 
\begin{equation}\label{eq:sigma22}
	\widehat{\sigma}_{i,j}(\bm{u}_h^*):=\max\limits_{1\le k\le N} \widehat{\sigma}_{i,j}( {u}_{h}^{*,\ell}), \qquad 
	\widetilde{\sigma}_{i,j}(\bm{u}_h^*):=\max\limits_{1\le k\le N} \widetilde{\sigma}_{i,j}( {u}_{h}^{*,\ell}),
\end{equation}
in which ${u}_{h}^{*,\ell}$ is the $\ell$-th component of $\bm{u}_h^*$, and $\widehat{\sigma}_{i,j}( {u}_{h}^{*,\ell})$ and $\widetilde{\sigma}_{i,j}( {u}_{h}^{*,\ell})$ are computed by \eqref{sec3:sigma_2d}.

Similar to Theorem \ref{thm:2dOE}, one can obtain the exact solver of the OE procedure defined by  \eqref{eq:OE-2Dsystem}. 

\begin{theorem}\label{thm:2dOE2}
	The OE procedure $(\overline{\bm u}_{i,j}^{\sigma},\overline{\bm v}_{i,j}^{\sigma},\overline{\bm w}_{i,j}^{\sigma}):=\bm{U}^{\sigma}_{i,j}=\mathcal{F}_{\rm OE}\{\bm{U}_{\kappa}\}_{\kappa\in\Lambda_{i,j}} $ for 2D hyperbolic systems can be exactly solved and explicitly expressed as 
	\begin{equation*} 
		\overline{\bm u}_{i,j}^{\sigma}=\bar{\bm u}_{i,j},\qquad
		\overline{\bm v}_{i,j}^{\sigma}=\bar{\bm v}_{i,j}\delta,\qquad
		\overline{\bm w}_{i,j}^{\sigma}=\bar{\bm w}_{i,j}\delta, 
	\end{equation*}
	where 
	$\delta := \exp\bigg(-\alpha_x\frac{{\Delta t}}{h_x} \widehat \sigma_{i,j}({\bm u}_h^*) -\alpha_y\frac{{\Delta t}}{h_y} \widetilde \sigma_{i,j}({\bm u}_h^*) \bigg) $ 
	with $\widehat \sigma_{i,j}({\bm u}_h^*)$ and $\widetilde \sigma_{i,j}({\bm u}_h^*)$ defined by \eqref{eq:sigma22}. 
\end{theorem}

The proof of Theorem \ref{thm:2dOE2} is similar to that of Theorem \ref{thm:2dOE} and thus is omitted here.
 
\subsection{Bound Preservation via Optimal Convex Decomposition} \label{sec:BP}

Solutions to hyperbolic conservation laws typically satisfy certain bound constraints. For instance, the entropy solutions of scalar conservation laws adhere to the maximum principle \cite{ZS1}, and the physical solutions of compressible Euler equations must maintain positive density and pressure \cite{ZS2}. In numerical simulations, preserving these bounds is crucial for both the physical significance and numerical stability of the results \cite{Wu2023Geometric}. The positivity-preserving analysis of the standard HWENO scheme for compressible Navier--Stokes equations, based on the classic cell average decomposition (CAD) \cite{ZS1}, was discussed in \cite{FZQ}. This subsection will present a rigorous BP  analysis of the OE-HWENO schemes based on the OCAD approach \cite{CDW1,CDW2}. Since the classic CAD \cite{ZS1} was proven optimal in the 1D case \cite{CDW2}, the following discussion will focus exclusively on the 2D OE-HWENO method.

Assume $G$ is a convex set composed of all admissible states that satisfy the desired bounds. For example,  $G = [U_{\text{min}}, U_{\text{max}}]$ with $ U_{\text{min}} := \min_{\bm x} u({\bm x}, 0)$ and $ U_{\text{max}} := \max_{\bm x} u_0({\bm x}, 0) $ for the scalar conservation law \eqref{sec2:2dHCLS}. We aim to develop BP OE-HWENO schemes that ensure the cell averages remain within $ G $ when the initial values belong to $ G $. Following \cite{ZS1, ZS2}, our analysis considers only the forward Euler time discretization, although it is directly extensible to high-order strong-stability-preserving time discretization, which is formally a convex combination of forward Euler steps.

Note that the OE procedure does not change the cell averages. 
The evolution equation of the cell averages for the 2D OE-HWENO method with  forward Euler time discretization can be written as 
\begin{equation}\label{sec2:BP_HWENO2d}
	\begin{aligned}
		\bar{u}_{i,j}^{n+1}=\bar{u}_{i,j}^{n}-\frac{{\Delta t}}{h_x}\sum_{\ell=1}^{Q}{\omega}^{{\rm G}}_\ell\big[\hat{f}(u^{-,{\rm G}}_{i+\frac12,\ell},u^{+,{\rm G}}_{i+\frac12,\ell})-\hat{f}(u^{-,{\rm G}}_{i-\frac12,\ell},u^{+,{\rm G}}_{i-\frac12,\ell})\big]
		\\
		-\frac{{\Delta t}}{h_y}\sum_{\ell=1}^{Q}{\omega}^{{\rm G}}_\ell\big[\hat{g}(u^{-,{\rm G}}_{\ell,j+\frac12},u^{+,{\rm G}}_{\ell,j+\frac12})-\hat{g}(u^{-,{\rm G}}_{\ell,j-\frac12},u^{+,{\rm G}}_{\ell,j-\frac12})\big],
	\end{aligned}
\end{equation}
with the values at the cell interfaces computed by
\begin{equation*}
	\begin{aligned} 
		&
		u^{+,{\rm G}}_{i-\frac12,\ell}=p_{i,j}(x_{i-\frac12},y^{{\rm G}}_{j,\ell}),\quad 
		u^{-,{\rm G}}_{i+\frac12,\ell}=p_{i,j}(x_{i+\frac12},y^{{\rm G}}_{j,\ell}),
		\\
		&
		u^{+,{\rm G}}_{\ell,j-\frac12}=p_{i,j}(x^{{\rm G}}_{i,\ell},y_{j-\frac12}),\quad 
		u^{-,{\rm G}}_{\ell,j+\frac12}=p_{i,j}(x^{{\rm G}}_{i,\ell},y_{j+\frac12}),
	\end{aligned}
\end{equation*}
where $p_{i,j}(x,y) := {u}_h \big|_{I_{i,j}}  \in\mathbb{P}^{k}(I_{i,j})$ denotes the polynomial reconstructed by the HWENO method from the OE modified moments, satisfying
\[
\bar{u}_{i,j}^{n}=\frac{1}{h_xh_y}\int_{I_{i,j}}p_{i,j}(x,y)\mathrm{d}x\mathrm{d}y:=\left\langle p_{i,j} \right\rangle.
\]

Firstly, we discuss the BP conditions for the 2D OE-HWENO schemes based on the classic CAD \cite{ZS1} in the form of 
\begin{equation*}
	\begin{aligned}
		\left\langle p_{i,j} \right\rangle=\frac{\lambda_1}{\lambda_1+\lambda_2}\omega^{\rm GL}_1\sum_{\ell=1}^{Q}\omega_\ell^{{\rm G}}\left(p_{i,j}(x_{i-\frac12},y^{{\rm G}}_{j+\ell})+p_{i,j}(x_{i+\frac12},y^{{\rm G}}_{j+\ell})\right)
		\\
		+\frac{\lambda_2}{\lambda_1+\lambda_2}\omega^{\rm GL}_1\sum_{\ell=1}^{Q}\omega_\ell^{{\rm G}}\left(p_{i,j}(x^{{\rm G}}_{i+\ell},y_{j-\frac12})+p_{i,j}(x^{{\rm G}}_{i+\ell},y_{j+\frac12})\right)
		\\
		+\sum_{\ell=2}^{L-1}\sum_{m=1}^{Q}\omega_\ell^{\rm GL}\omega_m^{{\rm G}}\left( \frac{\lambda_1}{\lambda_1+\lambda_2}p_{i,j}(x^{\rm GL}_{i,\ell},y^{{\rm G}}_{j,m})+\frac{\lambda_2}{\lambda_1+\lambda_2}p_{i,j}(x^{{\rm G}}_{i,m},y^{\rm GL}_{j,\ell})\right),
	\end{aligned}
\end{equation*}
with $\lambda_1=\frac{\alpha_x}{h_x} $ and $\lambda_2=\frac{\alpha_y}{h_y}$.
Following \cite{ZS1}, 
if $p_{i,j}(x^{\rm GL}_{i,\ell},y^{{\rm G}}_{j,m}) \in G$ and  $p_{i,j}(x^{{\rm G}}_{i,m},y^{\rm GL}_{j,\ell}) \in G$, then the scheme \eqref{sec2:BP_HWENO2d} is BP under the traditional BP CFL condition 
\begin{equation}\label{sec2:classic_CFL}
	{\Delta t}(\lambda_1+\lambda_2)\le \omega^{\rm GL}_1,
\end{equation}
 where $\omega^{\rm GL}_1=\frac{1}{12}$ for the sixth-order OE-HWENO scheme. 

Next, we discuss the BP condition for the 2D sixth-order OE-HWENO scheme based on the OCAD \cite{CDW2} in the form of
\begin{equation}\label{sec2:OCAD2d}
	\left\langle p_{i,j} \right\rangle=\overline{\omega}_{\star}\left[(1+\theta)\left\langle p_{i,j} \right\rangle_{x}+(1-\theta)\left\langle p_{i,j} \right\rangle_{y}\right] +
	\sum_{s=1}^{2} \omega_s\overline{p_{i,j}(x^{(s)}_{i},y^{(s)}_{j})},
\end{equation}
with $\theta :=\frac{\lambda_1-\lambda_2}{\lambda_1+\lambda_2} \in [-1,1],$ and 
\[
\begin{aligned} 
	&\left\langle p_{i,j} \right\rangle_{x}:=\frac12\left(\left\langle p_{i,j} \right\rangle_{x}^{-}+\left\langle p_{i,j} \right\rangle_{x}^{+}\right), \quad 
	\left\langle p_{i,j} \right\rangle_{y}:=\frac12\left(\left\langle p_{i,j} \right\rangle_{y}^{-}+\left\langle p_{i,j} \right\rangle_{y}^{+}\right),~
	\\
	&\left\langle p_{i,j} \right\rangle_x^\pm:=\frac{1}{h_y}\int_{y_{j-\frac12}}^{y_{j+\frac12}} p(x_{i\pm\frac12},y)\mathrm{d}y, \quad 
	\left\langle p_{i,j} \right\rangle_{y}^\pm:=\frac{1}{h_x}\int_{x_{i-\frac12}}^{x_{i+\frac12}} p(x,y_{j\pm\frac12})\mathrm{d}x,
	\\
	&
	\overline{p_{i,j}(x^{(s)}_i,y^{(s)}_j)}:=\frac14 \sum_{m,\ell\in\{\pm 1\}} p_{i,j}\left( x_i+m\frac{h_x}{2}x^{(s)},  y_j+\ell \frac{h_y}{2}y^{(s)} \right) ,
\end{aligned}
\]
where
\begin{equation*}
	\begin{aligned}
		&
		\overline{\omega}_{\star}=\left[\frac{14}{3}+\frac23\sqrt{78\theta^2+46}\cos\left(\frac13\arccos\frac{1476\theta^2-244}{78\theta^2+46}\right)\right]^{-1},
		\\
		&
		\omega_1=\frac{5(1-4\overline{{\omega}}_{\star}+2|\theta|\overline{{\omega}}_{\star})^2}{9(1-6\overline{{\omega}}_{\star}+4|\theta|\overline{{\omega}}_{\star})},~\omega_2=1-2\overline{{\omega}}_{\star}-\omega_1,
		\\
		&
		\left(x^{(1)}, y^{(1)}\right) =\left\{\begin{aligned} &\left(\sqrt{\frac{3(1-6\overline{{\omega}}_{\star}+4|\theta|\overline{{\omega}}_{\star})}{5(1-4\overline{{\omega}}_{\star}+2|\theta|\overline{{\omega}}_{\star})}},\sqrt{\frac{1-6\overline{{\omega}}_{\star}}{3(1-4\overline{{\omega}}_{\star}+2|\theta|\overline{{\omega}}_{\star})}}\right),~\text{if}~\theta\in[-1,0],
		\\&\left(\sqrt{\frac{1-6\overline{{\omega}}_{\star}}{3(1-4\overline{{\omega}}_{\star}+2|\theta|\overline{{\omega}}_{\star})}},\sqrt{\frac{3(1-6\overline{{\omega}}_{\star}+4|\theta|\overline{{\omega}}_{\star})}{5(1-4\overline{{\omega}}_{\star}+2|\theta|\overline{{\omega}}_{\star})}}\right),~\text{if}~\theta\in[0,1],\end{aligned}\right.
		\\
		& 
		\left(x^{(2)}, y^{(2)}\right) =\left\{\begin{aligned} &\left(0,\sqrt{\frac{1-4\overline{{\omega}}_{\star}-2|\theta|\overline{{\omega}}_{\star}-3\omega_1 (y^{(1)})^2}{3\omega_2}}\right),\quad\quad\quad\quad\quad~\text{if}~\theta\in[-1,0],
		\\&\left(\sqrt{\frac{1-4\overline{{\omega}}_{\star}-2|\theta|\overline{{\omega}}_{\star}-3\omega_1 (x^{(1)})^2}{3\omega_2}},0\right),\quad\quad\quad\quad\quad~\text{if}~\theta\in[0,1].\end{aligned}\right. 
	\end{aligned}	
\end{equation*}
Consider the Lax-Friedrichs fluxes $\hat{f}$ and $\hat{g}$,  with which the 1D three-point first-order schemes are BP under the CFL condition $\max\{\alpha_x\frac{\Delta t}{h_x}, \alpha_y\frac{\Delta t}{h_y}\}\le1$: 
\begin{equation}\label{sec:BP_1rd}
	u_2-\frac{\Delta t}{h_x}\Big(\hat{f}(u_2,u_3)-\hat{f}(u_1,u_2)\Big)\in G,\qquad
	u_2-\frac{\Delta t}{h_y}\Big(\hat{g}\left(u_2,u_3\right)-\hat{g}\left(u_1,u_2\right)\Big)\in G
\end{equation}
hold for any $u_1$, $u_2$, $u_3 \in G$.

\begin{theorem}\label{sec2:thm_BP2d}
	Consider the sixth-order OE-HWENO scheme and the OCAD \eqref{sec2:OCAD2d}. If the HWENO reconstructed values satisfy for all $i$ and $j$ that 
	\begin{equation}\label{sec2:BP_2d_1}
		{u}^{+,{\rm G}}_{i-\frac12,\ell}\in G,~~~ {u}^-_{i+\frac12,\ell}\in G,~~~  {u}^{+,{\rm G}}_{\ell,j-\frac12}\in G,~~~ {u}^-_{\ell,j+\frac12}\in G, \qquad \ell=1,\ldots,Q,
	\end{equation} 
	\begin{equation}\label{sec2:BP_2d_2}
		\Gamma_{i,j}:=\frac{\bar{u}_{i,j}-{\overline{\omega}_\star}\big[\frac{1+\theta}{2} \sum\limits_{\ell=1}^{Q}\omega^{{\rm G}}_\ell\big( {u}^{+,{\rm G}}_{i-\frac12,\ell} +{u}^{-,{\rm G}}_{i+\frac12,\ell}\big) +\frac{1-\theta}{2} \sum\limits_{\ell=1}^{Q}\omega^{{\rm G}}_\ell\big({u}^{+,{\rm G}}_{\ell,j-\frac12}+{u}^{-,{\rm G}}_{\ell,j+\frac12}\big) \big]}{1-2\overline{\omega}_\star}\in G,
	\end{equation}
	then the sixth-order {OE-HWENO} scheme preserves $\bar{u}^{n+1}_{i,j}\in G$ under the BP CFL condition
	\begin{equation}\label{sec2:BP_CFL2D}
		{\Delta t}(\lambda_1+\lambda_2)\le  \overline{{\omega}}_{\star}.
	\end{equation}
\end{theorem}
\begin{proof}
	The OCAD gives
	\begin{equation}\label{sec2:OCAD_uij}
		\bar{u}^n_{i,j}=\sum_{\ell=1}^{Q}\omega^{{\rm G}}_\ell\left({\widehat{\omega}_x}(u^{-,{\rm G}}_{i+\frac12,\ell}+u^{+,{\rm G}}_{i-\frac12,\ell}) 
		+ \widehat{\omega}_y(u^{-,{\rm G}}_{\ell,j+\frac12}+u^{+,{\rm G}}_{\ell,j-\frac12})\right)+(1-2\overline{{\omega}}_{\star})\Gamma_{i,j},
	\end{equation} 
	with $\widehat{\omega}_x=\frac{\overline{\omega}_{\star}(1+\theta)}{2}$,  $\widehat{\omega}_y=\frac{\overline{\omega}_{\star}(1-\theta)}{2}$, and 
	 $\Gamma_{i,j} = p_{i,j} (\xi) $ for some $\xi \in I_{i,j}$ according to the Mean Value Theorem. 
	Substituting the decomposition \eqref{sec2:OCAD_uij} into \eqref{sec2:BP_HWENO2d}, we obtain 
	\begin{equation}\label{OCAD2d}
		\bar{u}^{n+1}_{i,j}=\sum_{\ell=1}^{Q}\omega^{{\rm G}}_\ell\big(\widehat{\omega}_x(H^{-,{\rm G}}_{i+\frac12,\ell}+H^{+,{\rm G}}_{i-\frac12,\ell})
		+ \widehat{\omega}_y(H^{-,{\rm G}}_{\ell,j+\frac12}+H^{+,{\rm G}}_{\ell,j-\frac12})\big)+(1-2\overline{{\omega}}_{\star})\Gamma_{i,j},
	\end{equation}
	where
	\begin{equation*}
		\begin{aligned}
			H^{-,{\rm G}}_{i+\frac12,\ell}=&u^{-,{\rm G}}_{i+\frac12,\ell} - \frac{{\Delta t}}{\widehat{\omega}_x h_x}\big(\hat{f}(u^{-,{\rm G}}_{i+\frac12,\ell},u^{+,{\rm G}}_{i+\frac12,\ell})-\hat{f}(u^{+,{\rm G}}_{i-\frac12,\ell},u^{-,{\rm G}}_{i+\frac12,\ell})\big),
			\\
			H^{+,{\rm G}}_{i-\frac12,\ell}=&u^{+,{\rm G}}_{i-\frac12,\ell} - \frac{{\Delta t}}{\widehat{\omega}_x h_x}\big(\hat{f}(u^{+,{\rm G}}_{i-\frac12,\ell},u^{-,{\rm G}}_{i+\frac12,\ell})-\hat{f}(u^{-,{\rm G}}_{i-\frac12,\ell},u^{+,{\rm G}}_{i-\frac12,\ell})\big),
			\\
			H^{-,{\rm G}}_{\ell,j+\frac12}=&u^{-,{\rm G}}_{\ell,j+\frac12} - \frac{{\Delta t}}{\widehat{\omega}_y h_y}\big(\hat{g}(u^{-,{\rm G}}_{\ell,j+\frac12},u^{+,{\rm G}}_{\ell,j+\frac12})-\hat{g}(u^{+,{\rm G}}_{\ell,j-\frac12},u^{-,{\rm G}}_{\ell,j+\frac12})\big),
			\\
			H^{+,{\rm G}}_{\ell,j-\frac12}=&u^{+,{\rm G}}_{\ell,j-\frac12} - \frac{{\Delta t}}{\widehat{\omega}_y h_y}\big(\hat{g}(u^{+,{\rm G}}_{\ell,j-\frac12},u^{-,{\rm G}}_{\ell,j+\frac12})-\hat{g}(u^{-,{\rm G}}_{\ell,j-\frac12},u^{+,{\rm G}}_{\ell,j-\frac12})\big),
		\end{aligned}
	\end{equation*} 
	which take the same form as the 1D three-point first-order schemes \eqref{sec:BP_1rd}, ensuring that
	\begin{equation*}
		H^{-,{\rm G}}_{i+\frac12,\ell}\in G,\quad H^{+,{\rm G}}_{i-\frac12,\ell}\in G,\quad H^{-,{\rm G}}_{\ell,j+\frac12}\in G,\quad H^{+,{\rm G}}_{\ell,j-\frac12}\in G,
	\end{equation*}
	under the conditions \eqref{sec2:BP_2d_1}--\eqref{sec2:BP_2d_2} and the CFL conditions 
	\begin{equation*}
		\alpha_x\frac{\Delta t}{\widehat \omega_x h_x}\le 1,\qquad \alpha_y\frac{\Delta t}{\widehat \omega_y  h_y}\le 1.
	\end{equation*}
	which are equivalent to \eqref{sec2:BP_CFL2D}. Given the convex combination form in \eqref{OCAD2d} and the convexity of the set $G$, we conclude that $\bar{u}^{n+1}_{i,j}\in G$ under the CFL condition \eqref{sec2:BP_CFL2D}.
	This completes the proof.
\end{proof}

For the 2D sixth-order OE-HWENO scheme, when $\alpha_xh_y=\alpha_yh_x$ (i.e., $\theta=0$), $\overline{{\omega}}_{\star}=2-\frac{\sqrt{14}}{2}\approx0.1292$ in the BP CFL condition \eqref{sec2:BP_CFL2D}, which is notably milder than the traditional BP CFL condition \eqref{sec2:classic_CFL} with $\omega^{\rm GL}_1=\frac{1}{12}\approx0.0833$.

\begin{remark}[BP limiter] 
	Theorem \ref{sec2:thm_BP2d} establishes a sufficient condition \eqref{sec2:BP_2d_1}--\eqref{sec2:BP_2d_2} for the BP property of the OE-HWENO scheme. However, the HWENO reconstructed values may not always satisfy \eqref{sec2:BP_2d_1}--\eqref{sec2:BP_2d_2}. 
	In enforce this condition \eqref{sec2:BP_2d_1}--\eqref{sec2:BP_2d_2}, we can employ a scaling limiter \cite{ZS4}. For the scalar conservation law, the  limiter  is given by
	\begin{equation*}
		\tilde{p}_{i,j}(x,y)=\delta ({p}_{i,j}(x,y)-\bar{u}_{i,j}^n) + \bar{u}_{i,j}^n,~\delta=\min\left\{ \bigg| \frac{U_{\text{max}}-\bar{u}^n_{i,j}}{p^{\text{max}}_{i,j}-\bar{u}^n_{i,j}}\bigg|,\bigg| \frac{U_{\text{min}}-\bar{u}^n_{i,j}}{p^{\text{min}}_{i,j}-\bar{u}^n_{i,j}}\bigg|,1\right\},
	\end{equation*}
	with
	\begin{equation*}
		\begin{aligned}
			&p^{\text{max}}_{i,j}=\max\{p_{i,j}(x_{i\pm\frac12},y^{{\rm G}}_{j,\ell}),p_{i,j}(x^{{\rm G}}_{i,\ell},y_{j\pm\frac12}),\Gamma_{i,j}\},~\\
			&p^{\text{min}}_{i,j}=\min\{p_{i,j}(x_{i\pm\frac12},y^{{\rm G}}_{j,\ell}),p_{i,j}(x^{{\rm G}}_{i,\ell},y_{j\pm\frac12}),\Gamma_{i,j}\},
		\end{aligned}
	\end{equation*}
	where $\Gamma_{i,j}$ is computed by \eqref{sec2:BP_2d_2}. A similar local scaling positivity-preserving (PP) limiter has been designed for the Euler equations (to preserve the positivity of density and pressure) and related systems \cite{ZS2, ZS4, FZQ}. The PP limiter for the Euler equations will be utilized in Examples \ref{Sec3:Example_1dSedov}, \ref{sec3:Example_Leblanc}, \ref{Sec3:Example_2dSedov}, and \ref{Sec3:Example_2dMach2000}.
\end{remark}

\begin{remark}
This work is the first to employ the 2D optimal cell average decomposition to analyze the BP property of HWENO schemes. Our BP analysis also applies to the standard HWENO method without the OE procedure, because the OE procedure is separate from the standard HWENO scheme and does not alter the cell averages.
\end{remark}

\section{Numerical tests}\label{sec3}

This section presents extensive numerical results for both 1D and 2D benchmark and demanding examples to validate the accuracy, high resolution, non-oscillatory, and BP properties of the proposed sixth-order  OE-HWENO method on uniform Cartesian meshes. Our test cases include {two smooth examples} to verify the accuracy of our method, as well as several non-smooth problems. These include the Lighthill--Whitham--Richards traffic flow problem, five 1D Riemann problems, a 2D Riemann problem, a 2D double Mach reflection problem, a 2D Sedov problem, and a 2D Mach 2000 jet problem.

We will provide comparisons between our OE-HWENO method and the oscillation-free HWENO method from \cite{ZQ2} (termed OF-HWENO for convenience), as both methods employ similar damping techniques to suppress spurious oscillations. A notable difference is that the OF-HWENO method incorporates non-scale-invariant damping terms into the semi-discrete equations of first-order moments. As these damping terms are highly stiff when dealing with strong discontinuities and large-scale problems, the semi-discrete OF-HWENO method must be evolved in time using the modified exponential RK method to mitigate the highly restricted time step-size restriction, as demonstrated in Examples \ref{sec3:Example_LWR}, \ref{sec3:Example_Lax}, \ref{Sec3:Example_1dSedov}, and \ref{Sec3:Example_2dRiemann}. Thanks to the exact solver for our OE procedure, the proposed OE-HWENO method remains stable under a normal CFL constraint. Therefore, for our OE-HWENO method, we use the classic third-order SSP explicit RK method for time discretization. 
To match the sixth-order spatial accuracy, we define the time step size for the 1D accuracy tests as $\Delta t = \frac{C_\text{cfl}}{\alpha_x/h_x^2}$ and for the 2D accuracy tests as $\Delta t = \frac{C_\text{cfl}}{\alpha_x/h_x^2 + \alpha_y/h_y^2}$. For the 1D discontinuity tests, the time step size is set as $\Delta t = \frac{C_\text{cfl}}{\alpha_x/h_x}$, and for the 2D discontinuity tests, it is $\Delta t = \frac{C_\text{cfl}}{\alpha_x/h_x + \alpha_y/h_y}$, where $\alpha_x$ and $\alpha_y$ represent the maximum wave speeds in the $x$- and $y$-directions, respectively. 
To ensure a fair comparison, both OE-HWENO and OF-HWENO methods set the linear weights of the lowest degree polynomial to 0.025, and the CFL number to $C_\text{cfl} = 0.45$ for all numerical tests. 
An empirical artificial parameter $\omega^d$ is required in the OF-HWENO scheme \cite{ZQ2} to achieve satisfactory performance. This parameter can be problem-dependent. Following the suggestion in \cite{ZQ2}, we take $\omega^d = 3.5$ for the 1D case and $\omega^d = 0.75$ for the 2D case, unless otherwise stated.  
The simulations are implemented in FORTRAN 95 with double precision on an Intel(R) Xeon(R) Gold 5218R CPU @ 2.10GHz.

\subsection{Accuracy tests}\label{sec3:AccuracyTests}

\begin{example}[Burgers' equation] \label{Example:Burgers1DTestOrder} 
	The 1D and 2D Burgers' equations are employed to verify the accuracy of our OE-HWENO method. For the 1D Burgers' equation $u_t+(\frac{u^2}{2})_x=0$, the initial condition is $u(x,0) = 0.5 + \sin(\pi x)$ over the domain $\Omega = [0, 2]$ with periodic boundary conditions. For the 2D Burgers' equation  $u_t+(\frac{u^2}{2})_x+(\frac{u^2}{2})_y=0$, the initial condition is $u(x,y,0) = 0.5 + \sin(\pi (x+y)/2)$ over the domain $\Omega = [0, 2]^2$, with periodic boundary conditions in both the $x$- and $y$-directions. The simulations run until the final time $T = 0.5/\pi$, during which the solution remains smooth. Table \ref{sec3:1dBurgersTest} presents the numerical errors and convergence rates obtained using the proposed OE-HWENO method at different mesh resolutions, demonstrating that our method achieves the expected sixth-order accuracy in both 1D and 2D cases. This confirms that the OE procedure preserves the accuracy, being consistent with the theoretical analysis provided in Theorems \ref{thm:accuracy} and \ref{thm:accuracy2d}.

	\begin{table}[!thb]
		\centering
		\begin{threeparttable}
			\caption{Errors and convergence rate of the OE-HWENO method for Example \ref{Example:Burgers1DTestOrder}.}
			\label{sec3:1dBurgersTest}
			{\begin{tabular}{cccccccccccc}
					\toprule
					\multicolumn{1}{l}{\multirow{11}{*}{1D}}
					\multirow{2}{*}{Mesh Resolution}&
					\multicolumn{2}{c}{${\ell^\infty}$-norm} & \multicolumn{2}{c}{${\ell^1}$-norm}& \multicolumn{2}{c}{${\ell^2}$-norm}\cr
					\cmidrule(l){2-3} \cmidrule(l){4-5}\cmidrule(l){6-7}
					&  Error&Order& Error&Order &  Error&Order\cr
					\hline
					30&	2.21E-05&   $- $&    1.67E-06&   $- $&    5.40E-06    $- $	\\
					60&	2.56E-07&   6.43&    1.47E-08&   6.83&    4.74E-08&   6.83	\\
					90&	1.93E-08&   6.37&    1.04E-09&   6.53&    3.36E-09&   6.53	\\
					120&	2.91E-09&   6.58&    1.65E-10&   6.40&    5.46E-10&   6.40	\\
					150&	7.00E-10&   6.38&    3.97E-11&   6.37&    1.32E-10&   6.37	\\
					180&	2.30E-10&   6.10&    1.21E-11&   6.54&    4.06E-11&   6.54	\\ 
					\midrule 
                    \multicolumn{1}{l}{\multirow{6}{*}{2D}}
                    \vspace{-6.5mm} \cr
					30$\times$30&	2.99E-04&   $- $&    2.71E-05&   $- $&    7.48E-05&   $- $	\\
					60$\times$60&	7.17E-06&   5.38&    5.70E-07&   5.57&    1.71E-06&   5.57	\\
					90$\times$90&	7.24E-07&   5.66&    4.49E-08&   6.27&    1.41E-07&   6.27	\\
					120$\times$120&	1.15E-07&   6.40&    7.06E-09&   6.43&    2.30E-08&   6.43	\\
					150$\times$150&	2.74E-08&   6.41&    1.73E-09&   6.31&    5.52E-09&   6.31	\\
					180$\times$180&	8.49E-09&   6.43&    5.41E-10&   6.36&    1.69E-09&   6.36	\\	
					\bottomrule
			\end{tabular}}
		\end{threeparttable}
	\end{table}	 
\end{example}

\begin{example}[Compressible Euler equations]\label{Example:Euler1DTestOrder}
	The 1D and 2D compressible Euler equations are used to further examine the accuracy of our OE-HWENO method for hyperbolic systems.
	For the 1D compressible Euler equations in the form \eqref{sec2:1dHCLS}, with $\bm{u}=(\rho, \rho\mu,E)^{\top}$ and $\bm{f}(\bm{u})=(\rho\mu,\rho\mu^2+p, \mu(E+p))^{\top}$, the initial condition is set as $(\rho_0,\mu_0,p_0)=(1+0.2\sin(\pi x),1,1)$ over the domain $\Omega=[0,2]$ with periodic boundary conditions. For the 2D compressible Euler equations in the form \eqref{sec2:2dHCLS}, with $\bm{u}=(\rho, \rho\mu, \rho\nu,E )^{\top}$, $\bm{f}(\bm{u})=(\rho\mu, \rho\mu^2+p, \rho\mu\nu,\mu(E+p))^{\top}$, and $\bm{g}(\bm{u})=(\rho\nu,\rho\mu\nu,\rho\nu^2+p,\nu(E+p))^{\top}$, we take the initial condition as $(\rho_0,\mu_0,\nu_0,p_0)=(1+0.2\sin(\pi(x+y)),1,1,1)$ over the domain $[0,4]^2$ with periodic boundary conditions in both the $x$- and $y$-directions. 
	Here, $\rho$ denotes the density, $\mu$ and $\nu$ are the velocity in the $x$- and $y$-directions respectively, and $p$ is the pressure. The total energy is given by $E=\frac{p}{\gamma-1}+\frac{1}{2}\rho\mu^2$ in the 1D case and $E=\frac{p}{\gamma-1}+\frac12\rho(\mu^2+\nu^2)$ in the 2D case, where the adiabatic index $\gamma=1.4$ unless stated otherwise. 
	The simulations are run until a final time  $T=2$, where the exact solutions are $(\rho,\mu,p)=(1+0.2\sin(\pi (x-T)),1,1)$ in the 1D case and $(\rho,\mu,\nu,p)=(1+0.2\sin(\pi (x+y-2T)),1,1,1)$ in the 2D case. Table \ref{sec3:1dEulerTest} provides the numerical errors and convergence rate of the density computed by the OE-HWENO method, demonstrating that the OE procedure does not degenerate the high-order accuracy. 
		We observe that the high-order damping effect dominates the numerical errors, resulting in convergence rates higher than the expected sixth-order convergence rate of the HWENO scheme. This phenomenon is common and has also been observed in the OEDG schemes (Table 1 of \cite{PSW}), the sixth-order OF-HWENO scheme (Tables 3 and 5 of \cite{ZQ2}), and the OFDG schemes (Table 4.3 of \cite{LLS1}).  
	\begin{table}[!thb]
		\centering
		\begin{threeparttable}
			\caption{Errors and convergence rate of the OE-HWENO method for Example \ref{Example:Euler1DTestOrder}.} 
			\label{sec3:1dEulerTest}
			{\begin{tabular}{ccccccccccccc}
					\toprule
					\multicolumn{1}{l}{\multirow{11}{*}{1D}}
					\multirow{2}{*}{Mesh Resolution}&
					\multicolumn{2}{c}{${\ell^\infty}$-norm} & \multicolumn{2}{c}{${\ell^1}$-norm}& \multicolumn{2}{c}{${\ell^2}$-norm}\cr
					\cmidrule(l){2-3} \cmidrule(l){4-5}\cmidrule(l){6-7}
					&  Error&Order& Error&Order &  Error&Order\cr
					\midrule
					20&	3.53E-07&   $- $&    1.80E-07&   $- $&    2.10E-07&   $- $	\\
					40&	2.90E-09&   6.93&    1.66E-09&   6.76&    1.88E-09&   6.76	\\
					60&	1.72E-10&   6.97&    1.02E-10&   6.90&    1.14E-10&   6.90	\\
					80&	2.34E-11&   6.93&    1.38E-11&   6.94&    1.54E-11&   6.94	\\
					100&	4.93E-12&   6.98&    2.93E-12&   6.94&    3.27E-12&   6.94	\\
					120&	1.38E-12&   6.99&    8.23E-13&   6.96&    9.19E-13&   6.96	\\ 
					\midrule
					\multicolumn{1}{l}{\multirow{6}{*}{2D}}
					\vspace{-6.5mm} \cr
					20$\times$20 &1.10E-05&   $- $&    4.36E-06&   $- $&    5.88E-06&   $- $\\
					40$\times$40 &9.37E-08&   6.87&    4.49E-08&   6.60&    5.34E-08&   6.60\\
					60$\times$60 &5.62E-09&   6.94&    2.81E-09&   6.83&    3.29E-09&   6.83\\
					80$\times$80 &7.57E-10&   6.97&    3.88E-10&   6.89&    4.49E-10&   6.89\\
					100$\times$100&1.62E-10&   6.90&    8.30E-11&   6.91&    9.54E-11&   6.91\\
					120$\times$120&4.98E-11&   6.47&    2.37E-11&   6.88&    2.72E-11&   6.88\\
					\bottomrule
			\end{tabular}}
		\end{threeparttable}
	\end{table}
\end{example}

\subsection{Discontinuities tests}\label{sec3:NonsmoothTests}

\begin{example}[Lighthill--Whitham--Richards traffic flow model]\label{sec3:Example_LWR}  
	\begin{figure}[!htb]
		\centering  
		\begin{subfigure}{0.32\textwidth}
			{\includegraphics[width=5.25cm,angle=0]{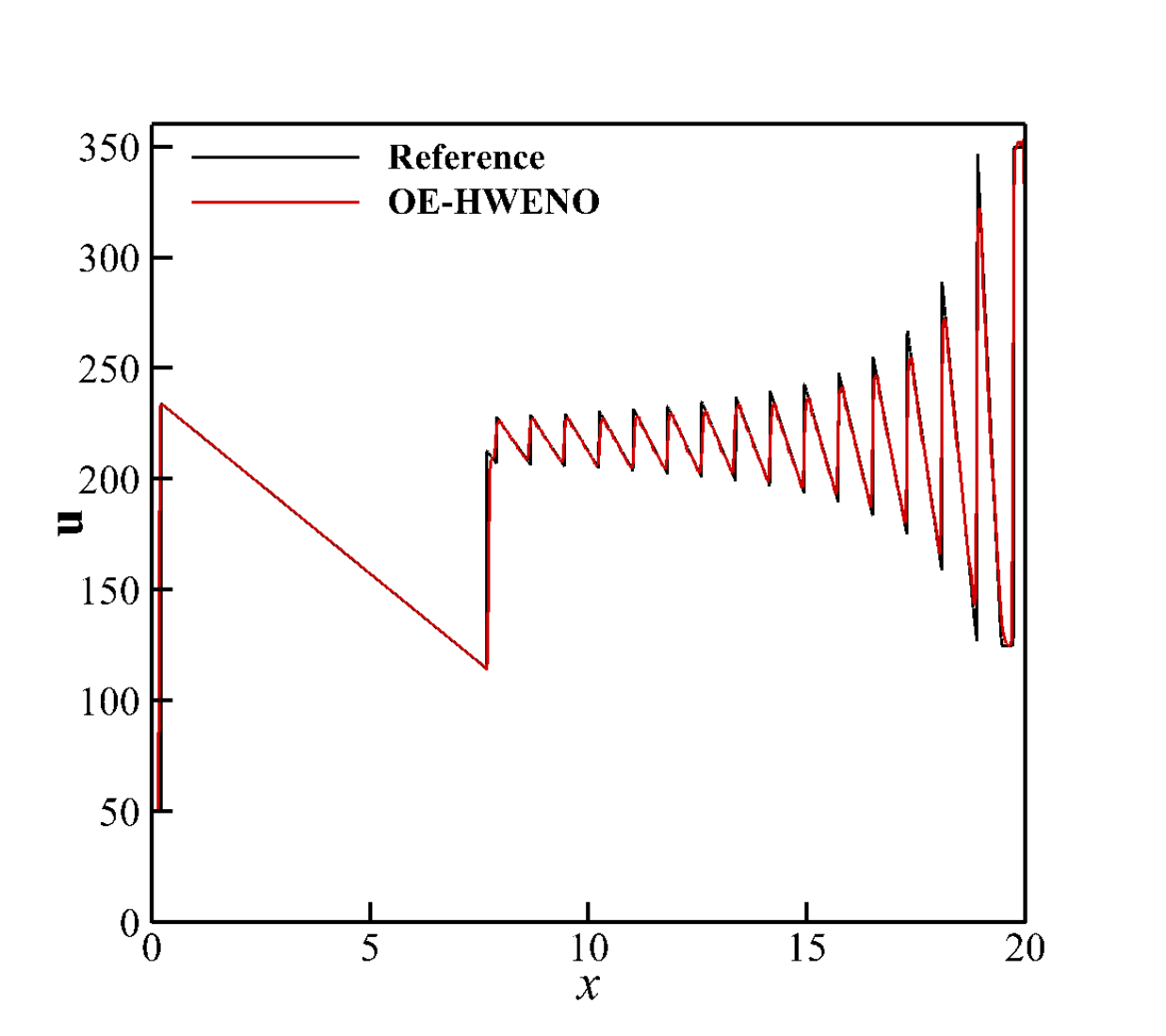}}
		\end{subfigure}  
		\begin{subfigure}{0.32\textwidth}
			{\includegraphics[width=5.25cm,angle=0]{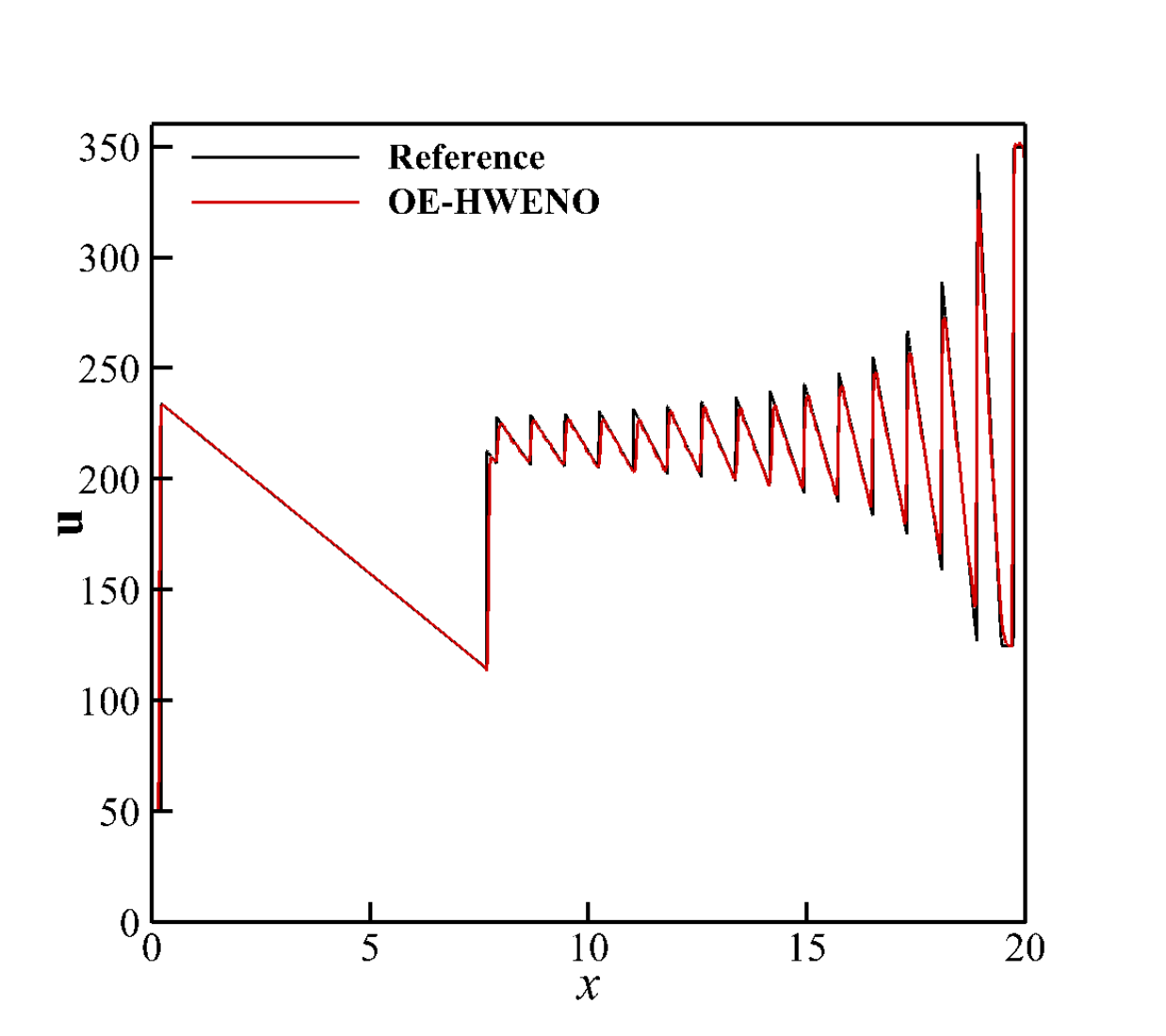}}
		\end{subfigure}  
		\begin{subfigure}{0.32\textwidth}
			{\includegraphics[width=5.25cm,angle=0]{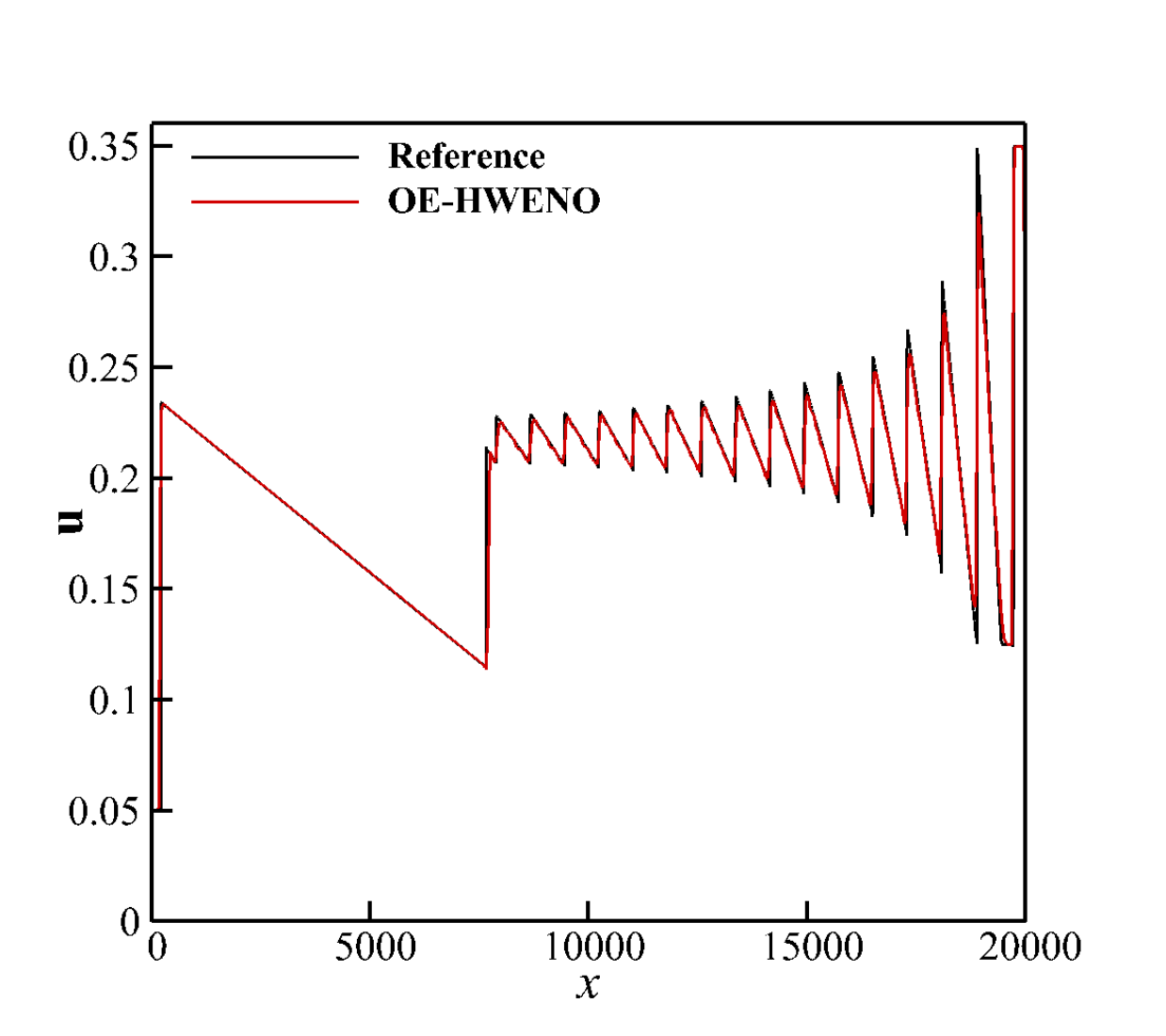}}
		\end{subfigure} 
		\begin{subfigure}{0.32\textwidth}
			{\includegraphics[width=5.25cm,angle=0]{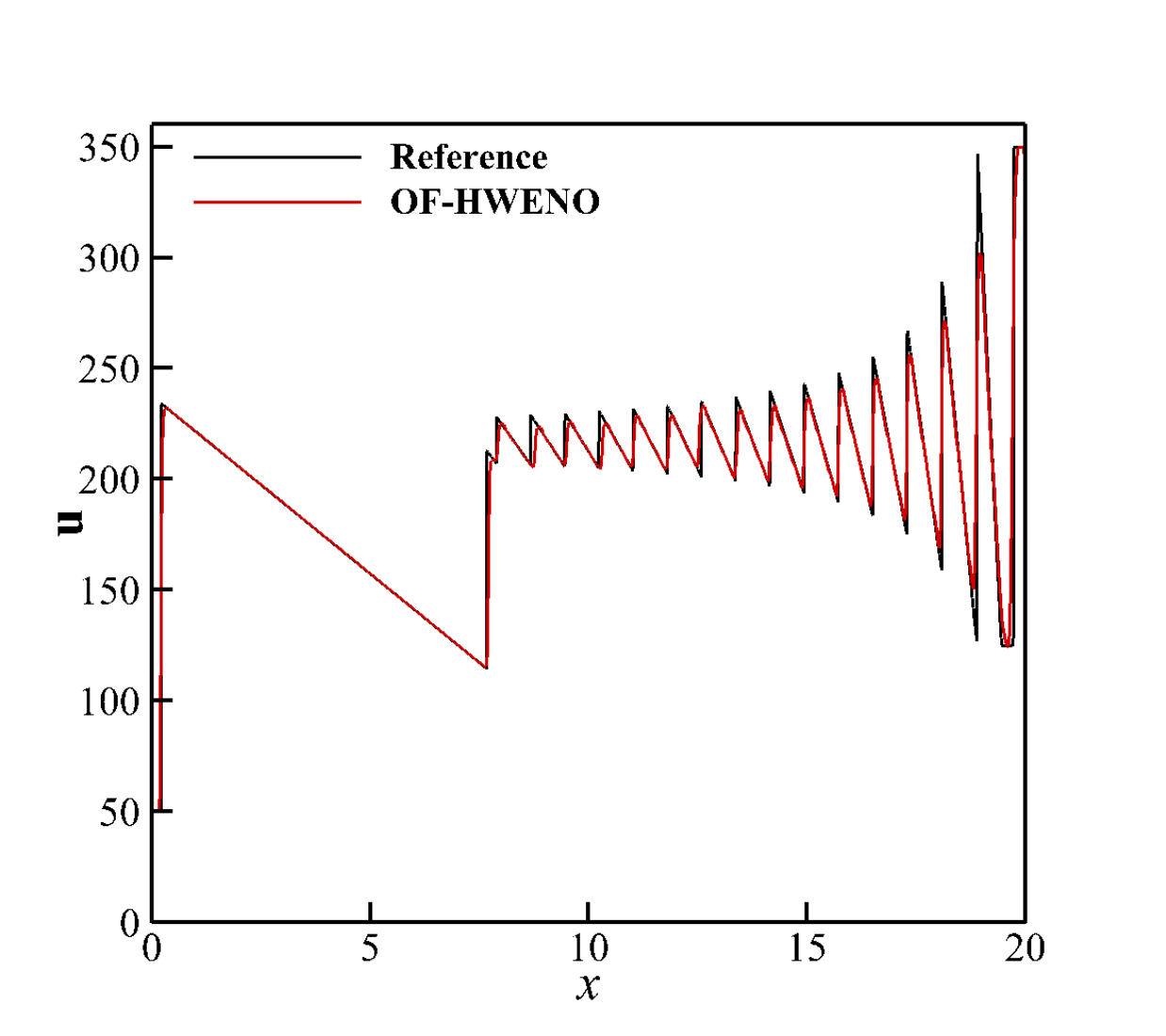}}
		\end{subfigure}  
		\begin{subfigure}{0.32\textwidth}
			{\includegraphics[width=5.25cm,angle=0]{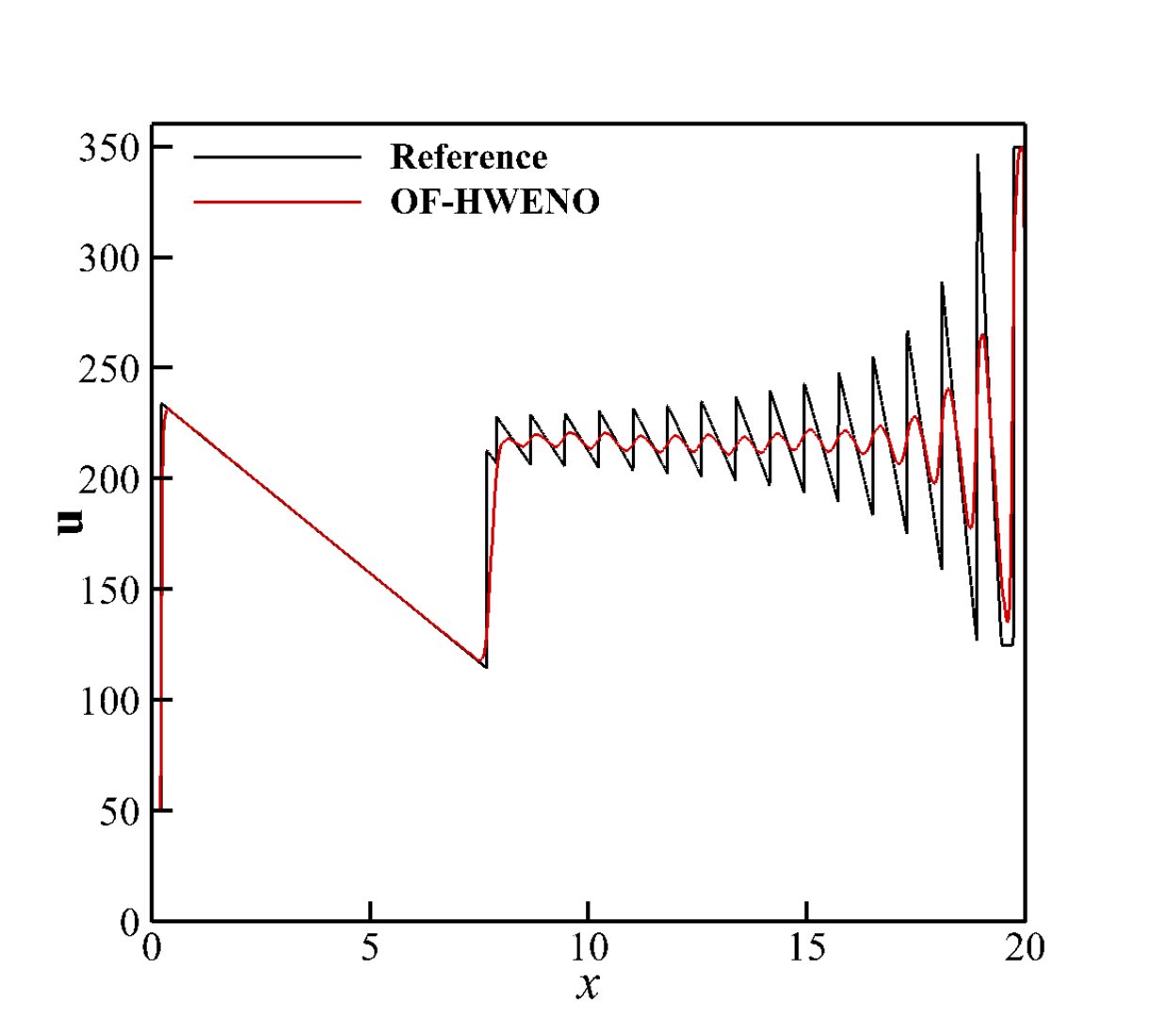}}
		\end{subfigure}  
		\begin{subfigure}{0.32\textwidth}
			{\includegraphics[width=5.25cm,angle=0]{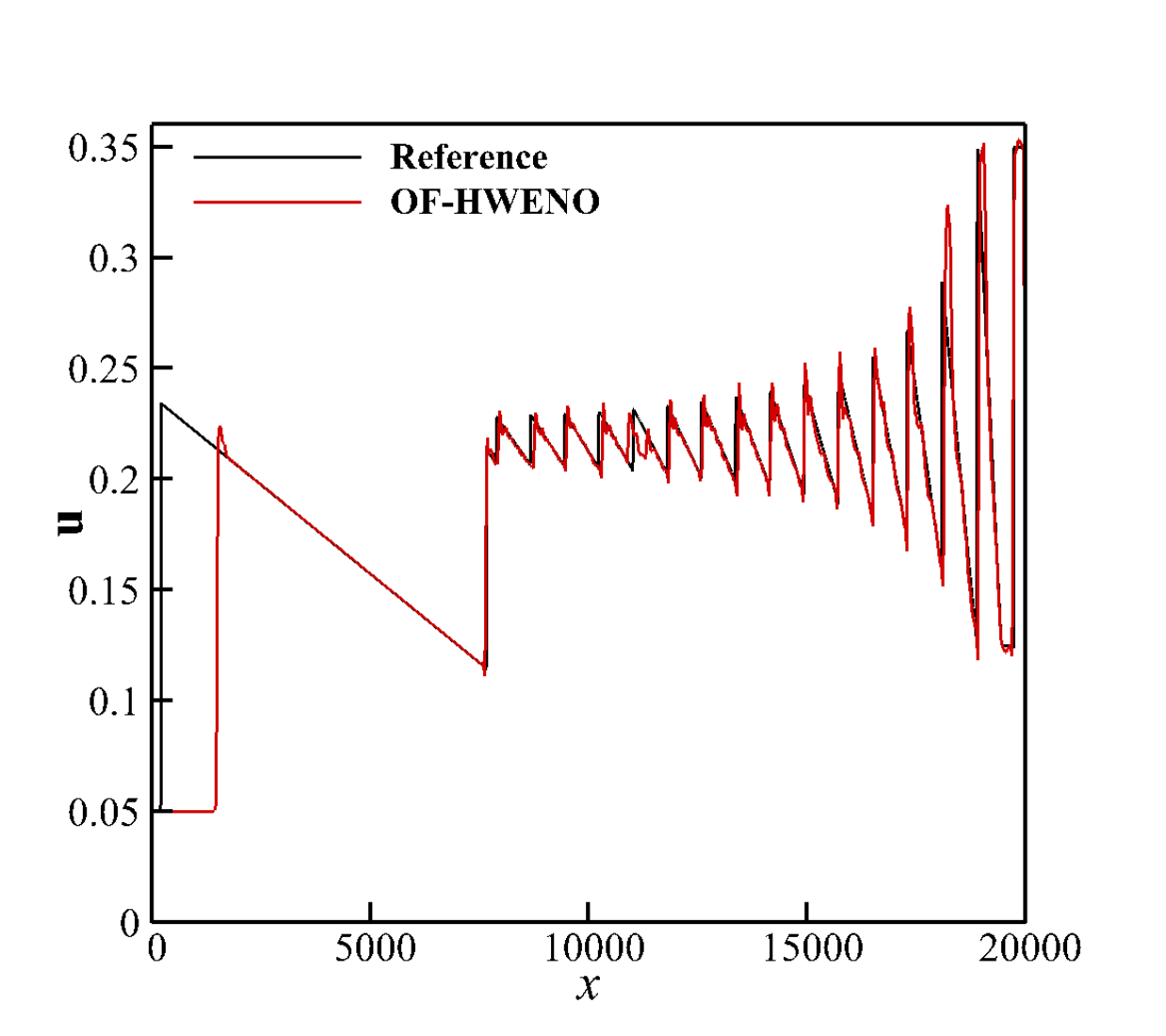}}
		\end{subfigure}   
		\caption{ Numerical results computed by the OE-HWENO and OF-HWENO methods with 800 uniform cells for the equations \eqref{sec3:LWR_Eq1}, \eqref{sec3:LWR_Eq2}, and \eqref{sec3:LWR_Eq3} (from left to right) for Example \ref{sec3:Example_LWR}. 
		} \label{sec3:Fig_LWR}
	\end{figure}
	Following \cite{LWZSC}, we simulate the traffic flow problem modeled by a scalar conservation law
	\begin{equation}\label{sec3:LWR_Eq1}
		u_t+f(u)_x=0,~f(u)=\begin{cases}
			-0.4u^2+100u,~\quad\quad\quad\quad\quad0\le u\le 50,
			\\
			-0.1u^2+15u+3500,~~\quad\quad 50\le u\le 100,
			\\
			-0.024u^2-5.2u+4760,~\quad100\le u\le 350,
		\end{cases}
	\end{equation} 
	where $t$ represents time in hours (h), $x$ denotes distance in kilometers (km), and $u$ signifies traffic density in vehicles per kilometer (veh/km). Initially, the density at the left entrance is 50 veh/km, and a traffic accident has occurred on the freeway, resulting in a piecewise linear traffic density profile as illustrated in Figure 6 of \cite{LWZSC}. The left boundary condition is specified as follows: the entrance is temporarily closed for 10 minutes to alleviate congestion. Following this closure, traffic resumes from the entrance at a density of 75 veh/km. However, after 20 minutes, the entrance flow reverts to its original density of 50 veh/km. For the right boundary, a traffic signal is positioned at the freeway exit, operating on a cyclic pattern: 2 minutes of green light (denoting zero density) followed by 1 minute of red light (denoting a jam density of 350 veh/km).
	\\ 	\indent
	To demonstrate the significance of scale-invariant and evolution-invariant properties, we reformulate equation \eqref{sec3:LWR_Eq1} in different units, resulting in the following two equivalent equations: \eqref{sec3:LWR_Eq2} and \eqref{sec3:LWR_Eq3}. The first one is that we introduce a new time variable, $\tilde{t}=60t$, measured in minutes and calculated as $\tilde{t}=60t$.  By employing this transformation, equation \eqref{sec3:LWR_Eq1} is cast to
	\begin{equation}\label{sec3:LWR_Eq2}
		u_{\tilde{t}}+\frac{1}{60}f(u)_x=0.
	\end{equation} 
	Another is to redefine the distance variable $\tilde{x}=1000x$ measured in meters (m) and transform the density as $\tilde{u}=\frac{u}{1000}$ measured in vehicles per meter (veh/m). With these transformations, we obtain
	\begin{equation}\label{sec3:LWR_Eq3}
		\tilde{u}_{{t}}+\frac{1}{60}f(1000\tilde{u})_{\tilde{x}}=0.
	\end{equation}
	We simulate this problem in the domain $[0,20~\text{km}]$ by solving the three equivalent equations \eqref{sec3:LWR_Eq1}, \eqref{sec3:LWR_Eq2} and \eqref{sec3:LWR_Eq3}, up to the time of an hour. The numerical solutions computed by the OE-HWENO and OF-HWENO schemes over 800 uniform cells are presented in Fig.~\ref{sec3:Fig_LWR}, where the reference solution is generated by the OE-HWENO scheme with 10000 cells. The numerical results of the above three equivalent equations computed by the OE-HWENO schemes are consistent without any nonphysical oscillations, thanks to the scale-invariant and evolution-invariant attributes. Conversely, the OF-HWENO method yields inconsistent numerical results for the three equivalent equations in varying units. More specifically, the OF-HWENO solutions exhibit considerable smearing in equation \eqref{sec3:LWR_Eq2}, whereas displaying notable spurious oscillations for equation \eqref{sec3:LWR_Eq3}.  Additionally, the serious nonphysical oscillations lead to a shift in the OF-HWENO solution for equation  \eqref{sec3:LWR_Eq3}, because the propagation speed of the oscillations does not match the correct speed of the discontinuity.
\end{example}

\begin{example}[Lax problem]\label{sec3:Example_Lax} 
	This is a classic Riemann problem of the 1D compressible Euler equations with discontinuous initial values. We take the scaled initial data $\bm{u}^{\lambda}(x,0)={\lambda}\bm{u}_0(x)$, where $\bm{u}_0(x)=(\rho_0,\rho_0 \mu_0,p_0)^{\top}$ is defined by 
	\begin{equation*} 
		(\rho_0,\mu_0,p_0)^{\top}=\begin{cases}
			(0.445, 0.698, 3.528)^{\top},&-0.5\le x<0,
			\\(0.5, 0, 0.571)^{\top},&0\le x\le 0.5.
		\end{cases}
	\end{equation*}	
	Outflow boundary conditions are imposed on all boundaries, and the final time is set to $T=0.16$. It is worth noting that the exact solution adheres to the relationship $\frac{1}{\lambda}\bm{u}^{\lambda}(x,t) = \bm{u}^{1}(x,t)$ for any constant $\lambda>0$. To validate the scale-invariant property, three different $\lambda$ values are chosen from the set $\{10^{-7},1,10^{7}\}$. 
	Fig.~\ref{sec3:Fig_Lax} displays the results of density obtained by the OE-HWENO and OF-HWENO methods. It is observed that the OE-HWENO solution is scale-invariant and can effectively captures the shock and contact discontinuity without any noticeable spurious oscillations, regardless of the value of $\lambda$. In contrast, the OF-HWENO method produces inconsistent numerical outputs for different $\lambda$ values. More specifically, Fig.~\ref{sec3:Fig_Lax} clearly shows that the OF-HWENO solution shows excessive smearing for $\lambda=10^{7}$ and displays spurious oscillations near shocks and contact discontinuities when $\lambda=10^{-7}$.
	\begin{figure}[!htb]
		\centering 
		\begin{subfigure}{0.35\textwidth}
			{\includegraphics[width=5.25cm,angle=0]{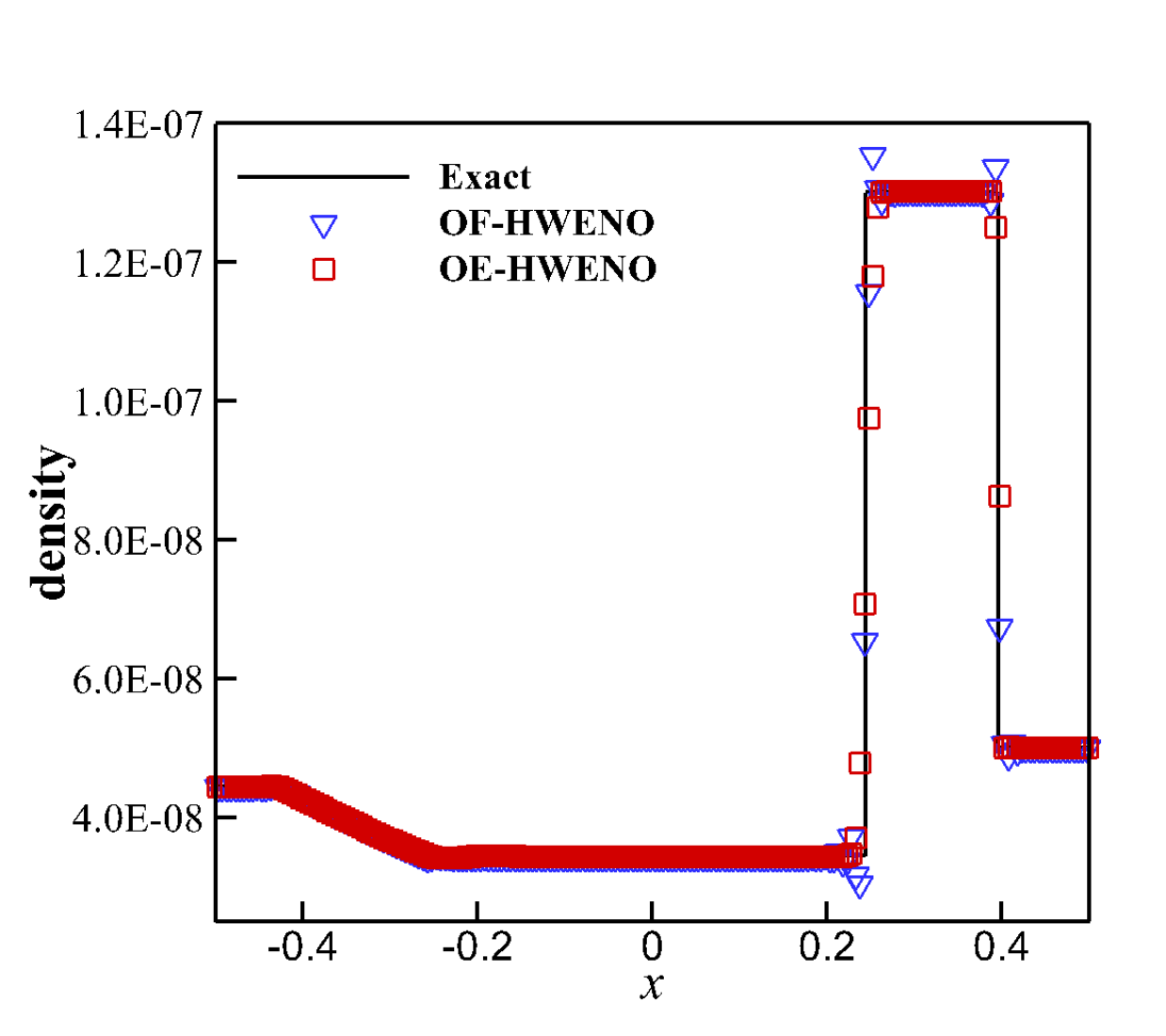}}  
		\end{subfigure} 
		\begin{subfigure}{0.33\textwidth}
			{\includegraphics[width=5.25cm,angle=0]{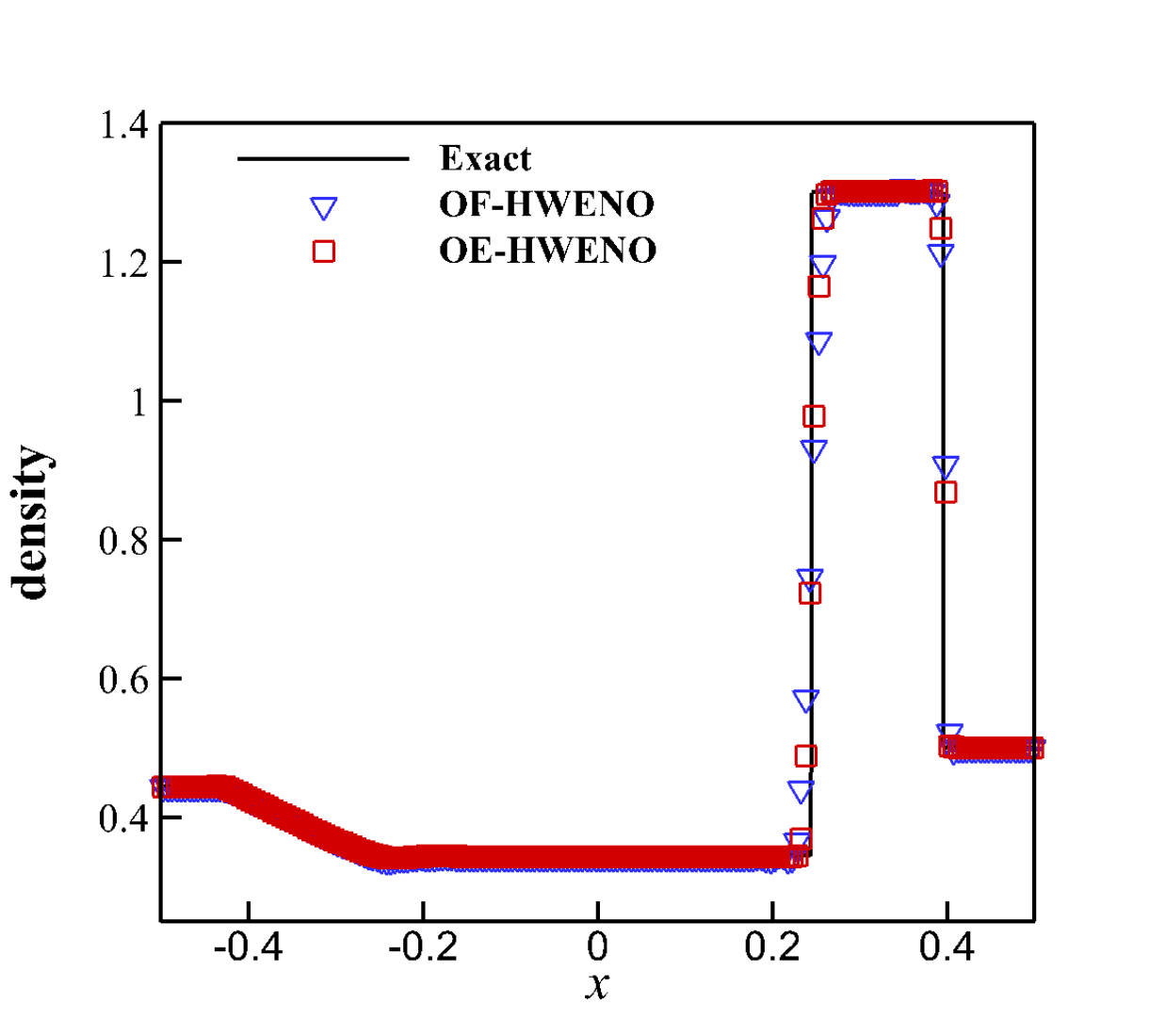}}  
		\end{subfigure} 
		\begin{subfigure}{0.3\textwidth}
			{\includegraphics[width=5.25cm,angle=0]{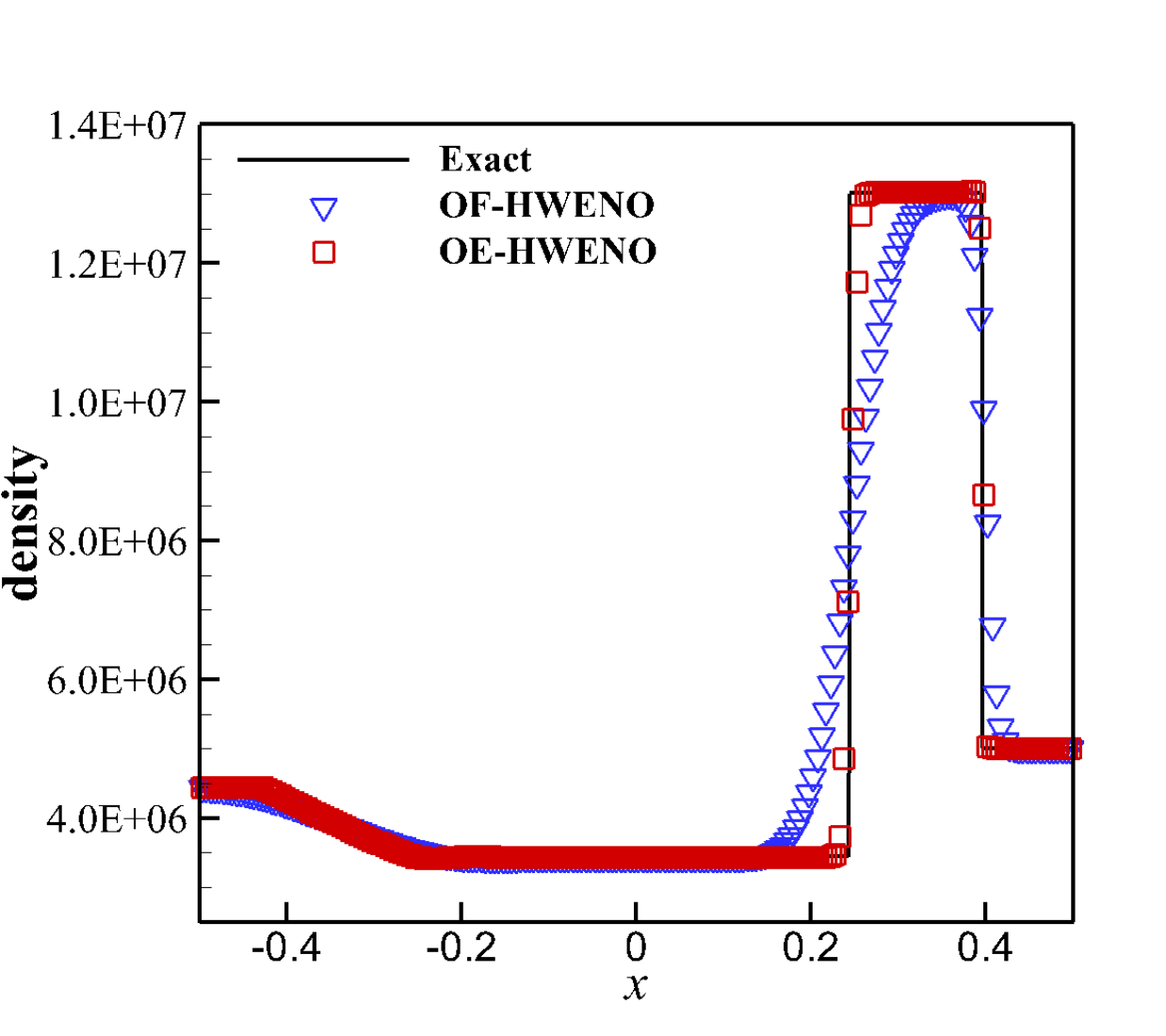}} 
		\end{subfigure} 
		\begin{subfigure}{0.35\textwidth}
			{\includegraphics[width=5.25cm,angle=0]{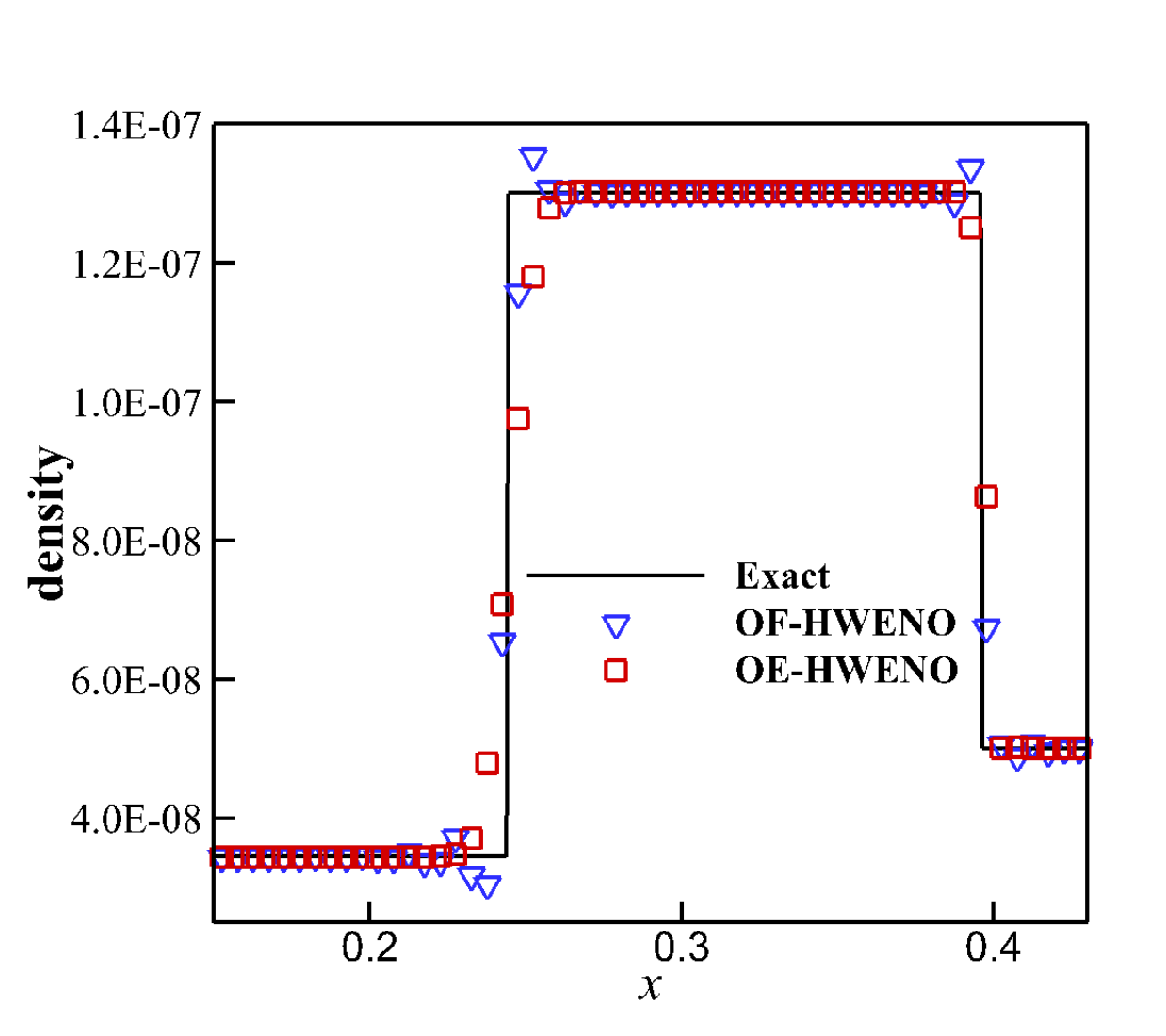}} 
			\caption{$\lambda=10^{-7}$}
		\end{subfigure} 
		\begin{subfigure}{0.33\textwidth}
			{\includegraphics[width=5.25cm,angle=0]{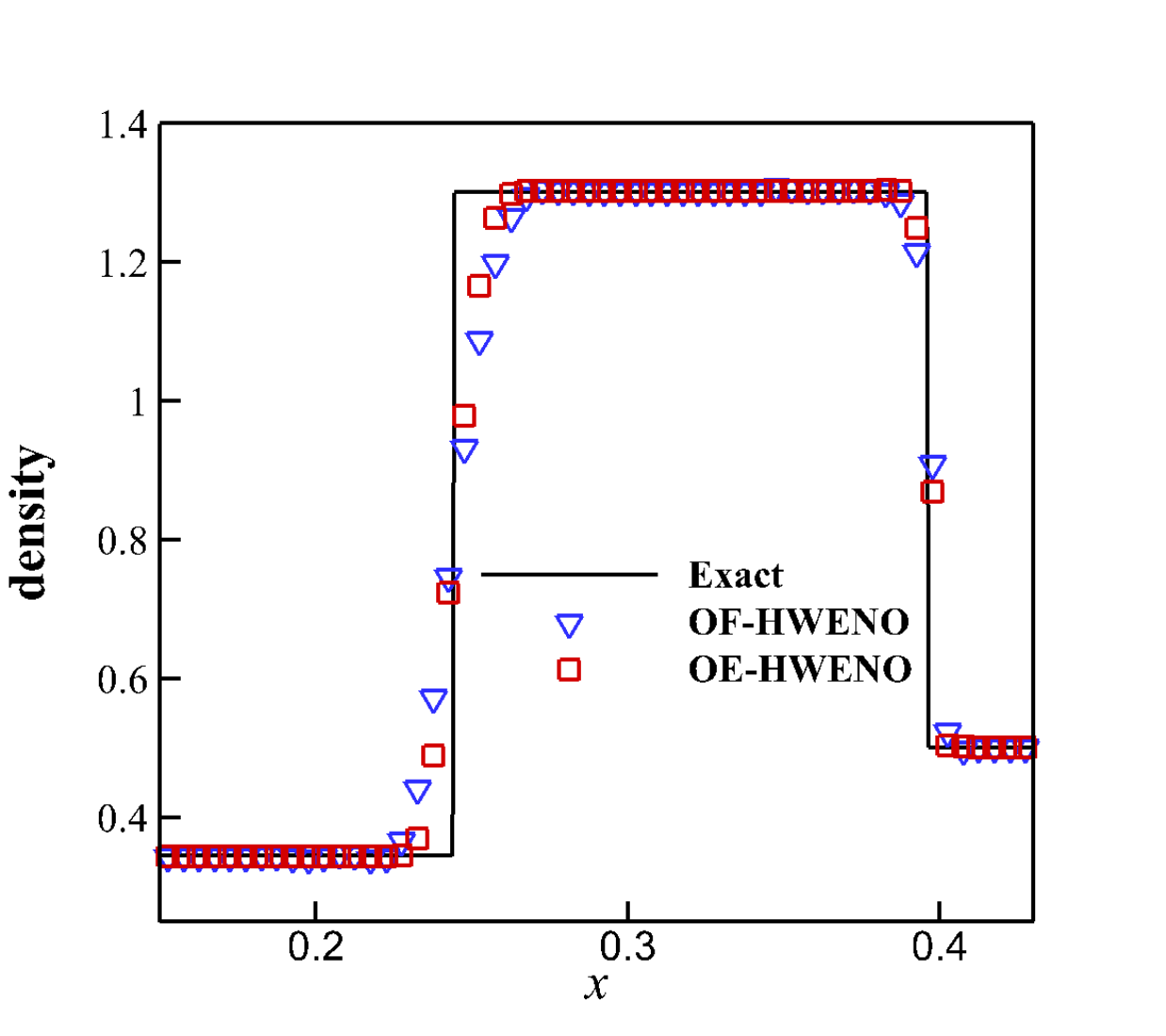}} 
			\caption{$\lambda=1$}
		\end{subfigure} 
		\begin{subfigure}{0.3\textwidth}
			{\includegraphics[width=5.25cm,angle=0]{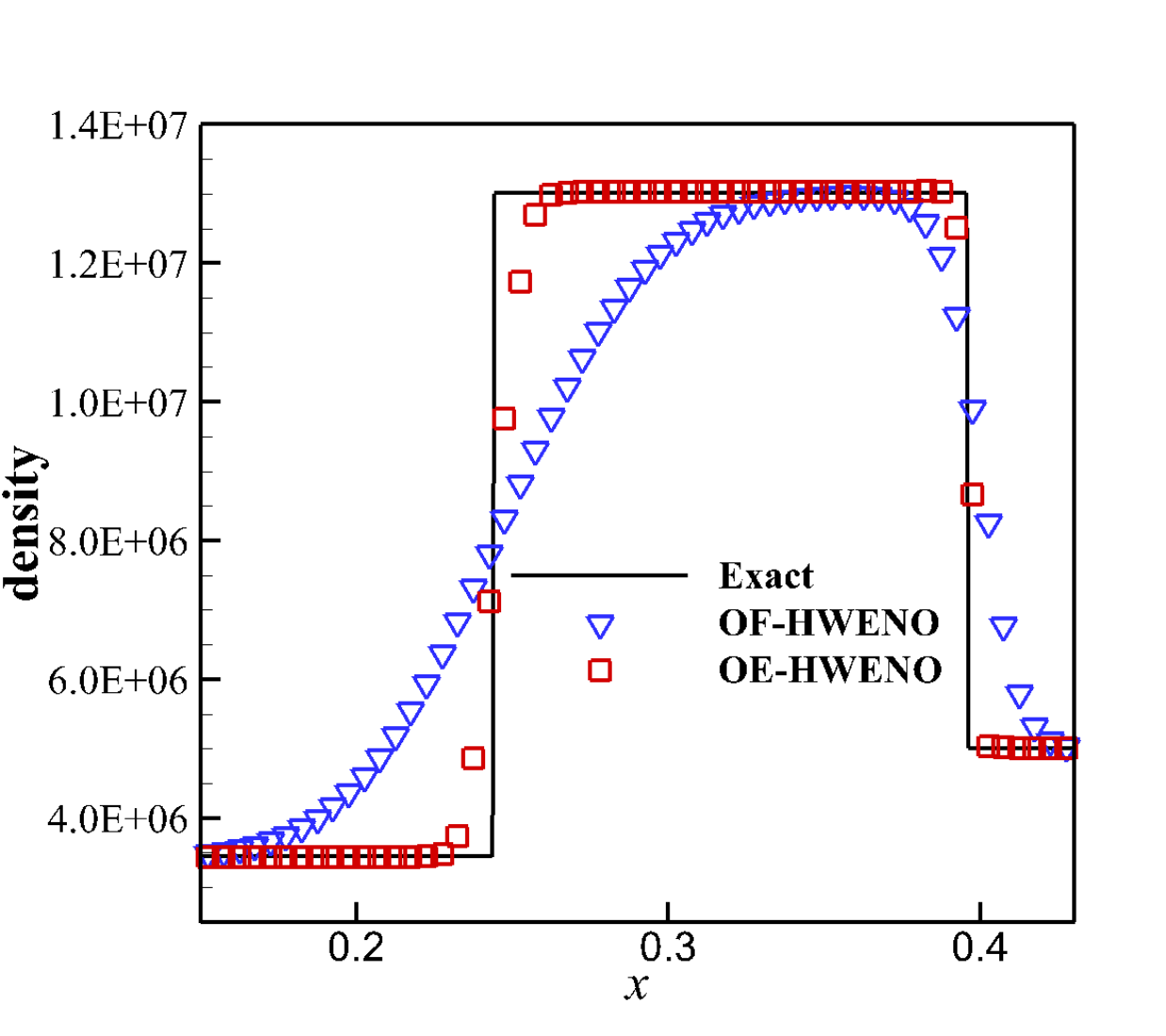}}
			\caption{$\lambda=10^{7}$}
		\end{subfigure} 
		\caption{Density of Lax problem computed by OE-HWENO and OF-HWENO schemes with $200$ uniform cells. 
		}\label{sec3:Fig_Lax}
	\end{figure}
\end{example} 

\begin{example}[Woodward--Colella blast wave problem]\label{sec3:Example_Blastwave}  
	This example simulates the interaction of two blast waves for the 1D compressible Euler equations with the initial values 
	\begin{equation*}
		(\rho_0,\mu_0,p_0)^{\top}=\begin{cases}
			(1,0,10^3)^{\top},&0<x<0.1,
			\\(1,0,10^{-2})^{\top},&0.1<x<0.9,
			\\(1,0,10^2)^{\top},&0.9<x<1.
		\end{cases}
	\end{equation*}	 
	Reflective boundary conditions are applied to all boundaries. This problem is simulated up to the final time $T=0.038$. The results of density computed by OE-HWENO and OF-HWENO methods are plotted in Fig.~\ref{sec3:Fig_Blast}, demonstrating that the OE-HWENO method exhibits higher resolution than the OF-HWENO method.
	\begin{figure}[!htb]
		\centering
		{\includegraphics[width=7.0cm,angle=0]{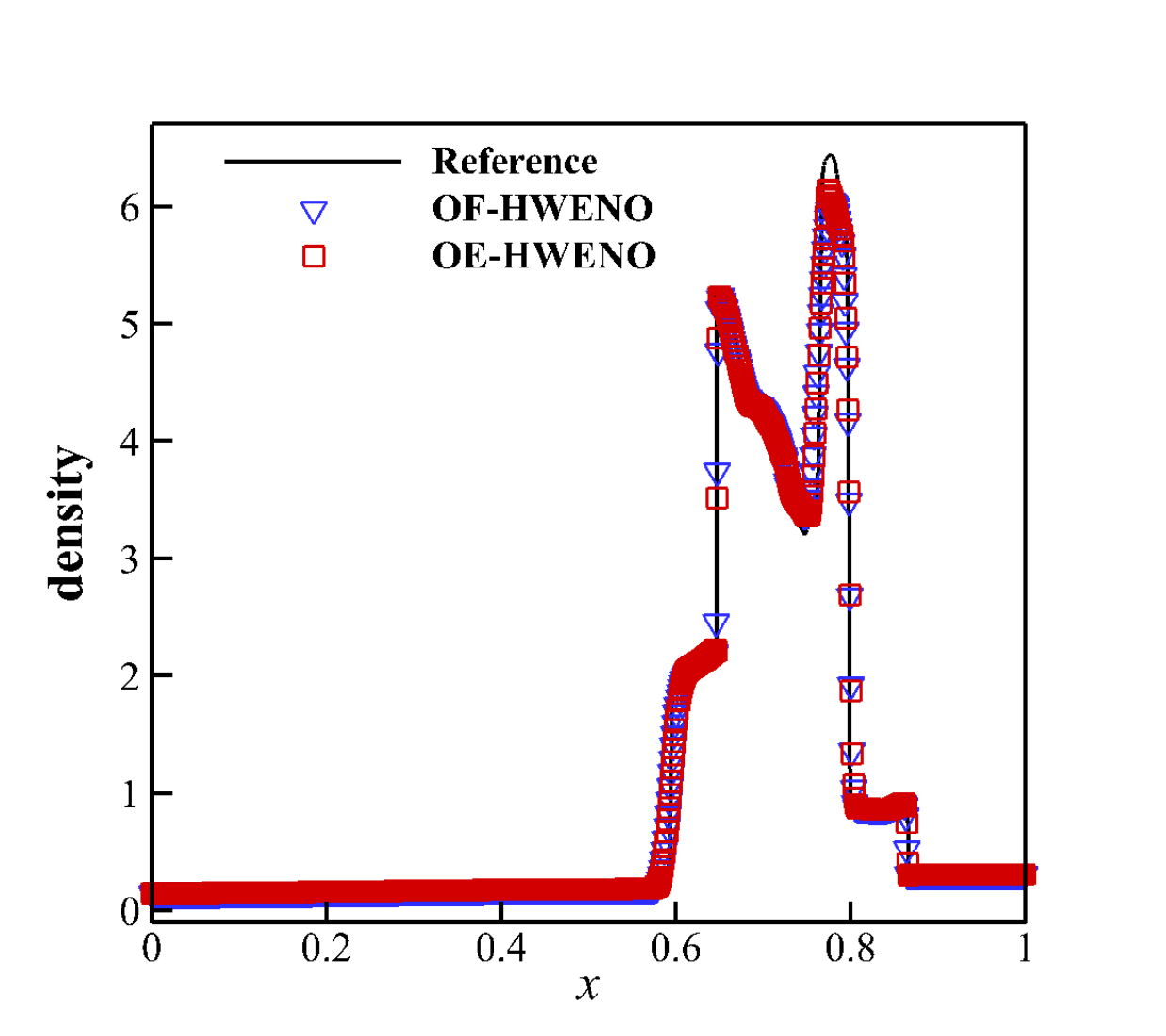}}
		{\includegraphics[width=7.0cm,angle=0]{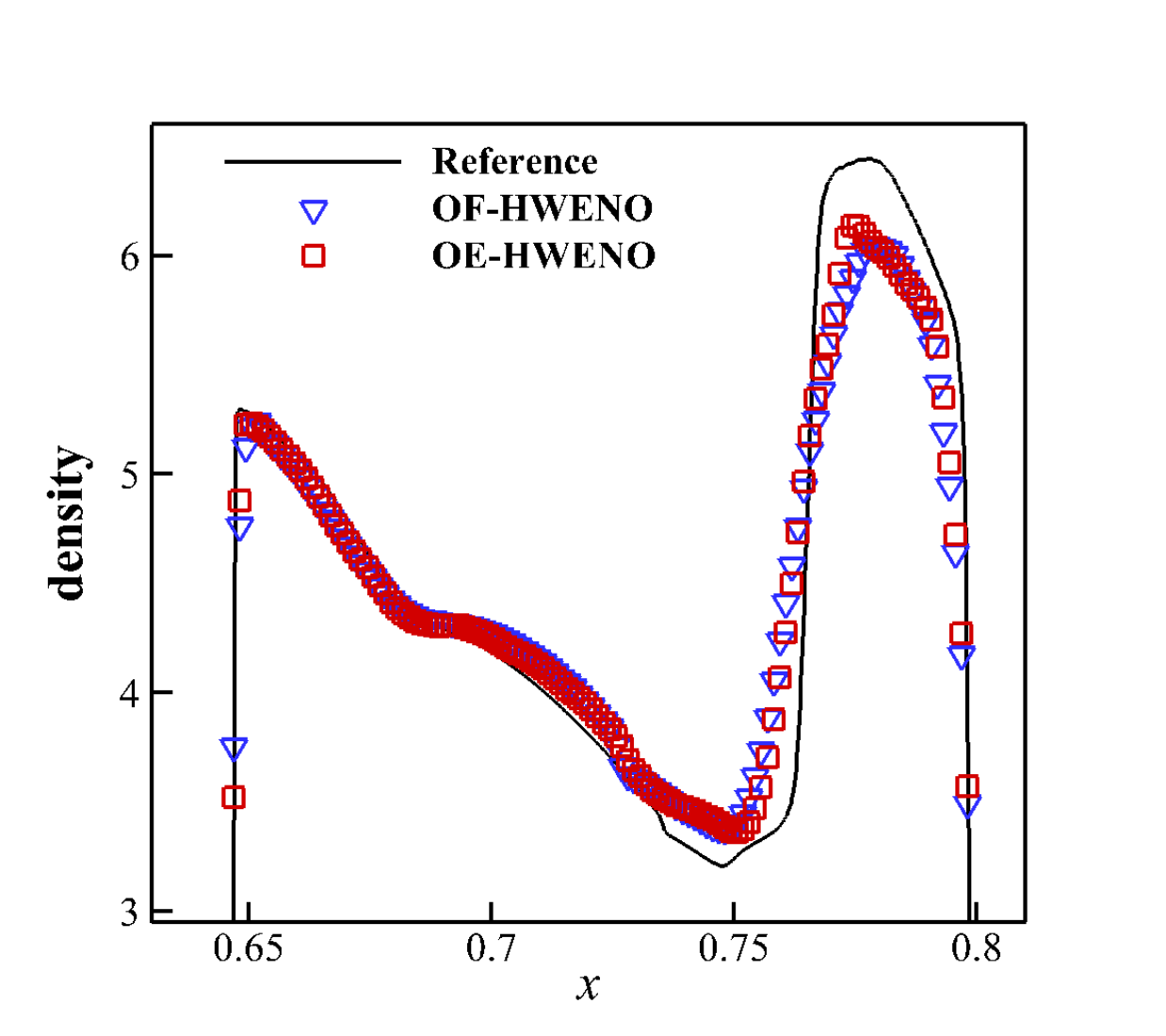}} 
		\caption{Density of the Woodward-Colella blast wave problem computed by the OE-HWENO and OF-HWENO methods with $800$ uniform cells.  
		}\label{sec3:Fig_Blast}
	\end{figure}
\end{example}

\begin{example}[1D Sedov problem]\label{Sec3:Example_1dSedov} 
	\begin{figure}[!htb]
		\centering
		\begin{subfigure}{0.32\textwidth}
			{\includegraphics[width=5.25cm,angle=0]{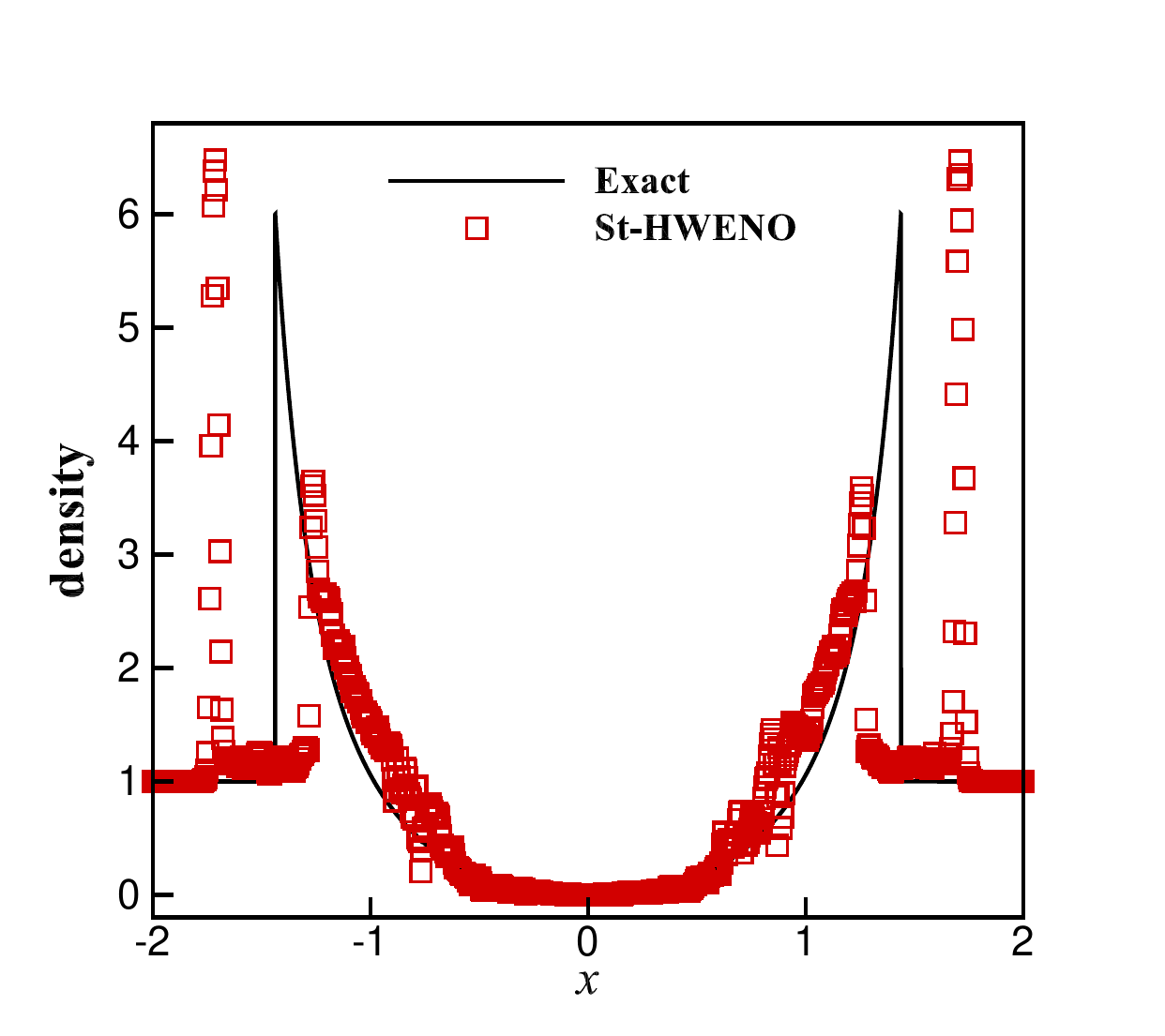}} 
		\end{subfigure} 
		\begin{subfigure}{0.32\textwidth}
			{\includegraphics[width=5.25cm,angle=0]{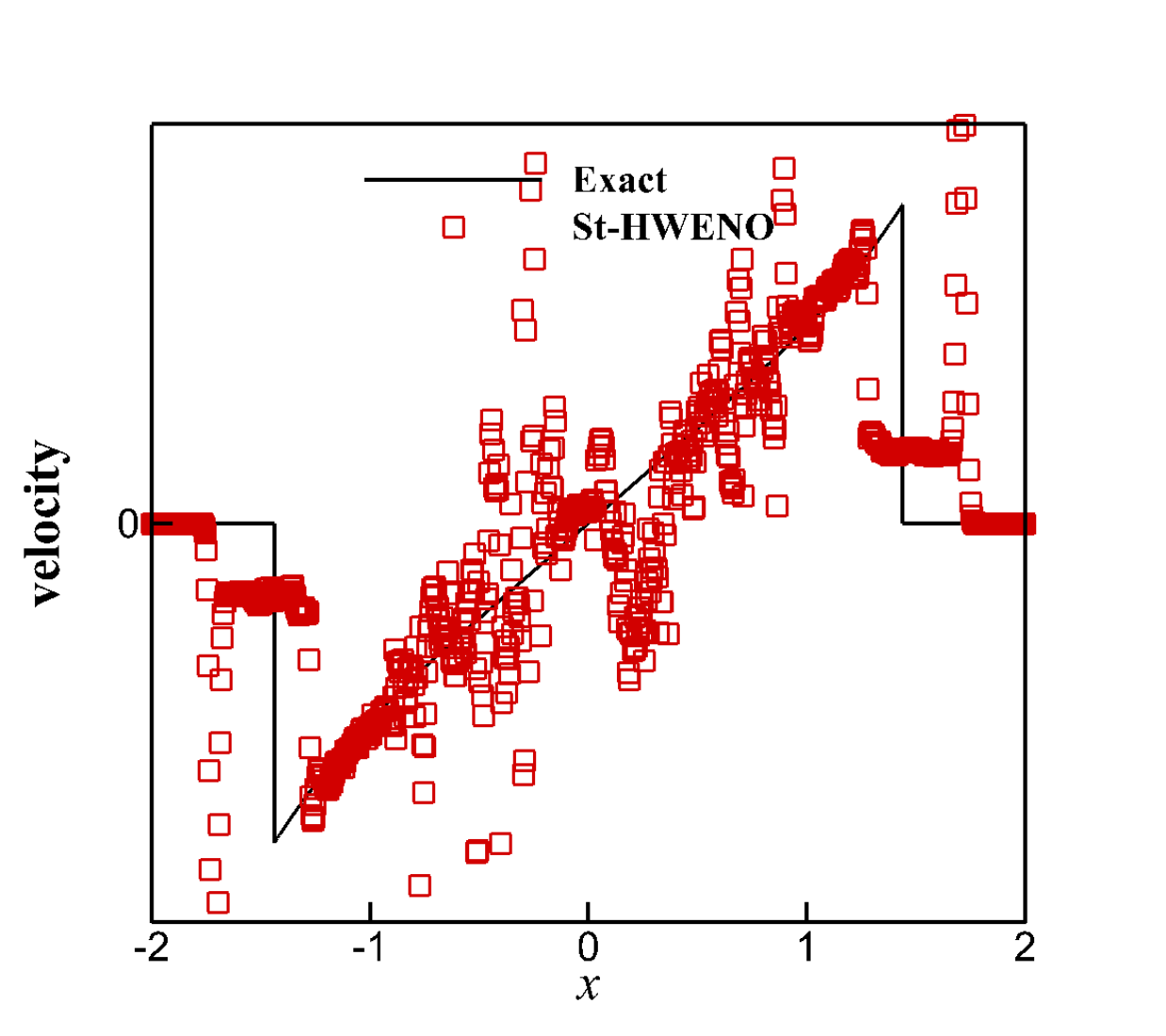}}
		\end{subfigure} 
		\begin{subfigure}{0.32\textwidth}
			{\includegraphics[width=5.25cm,angle=0]{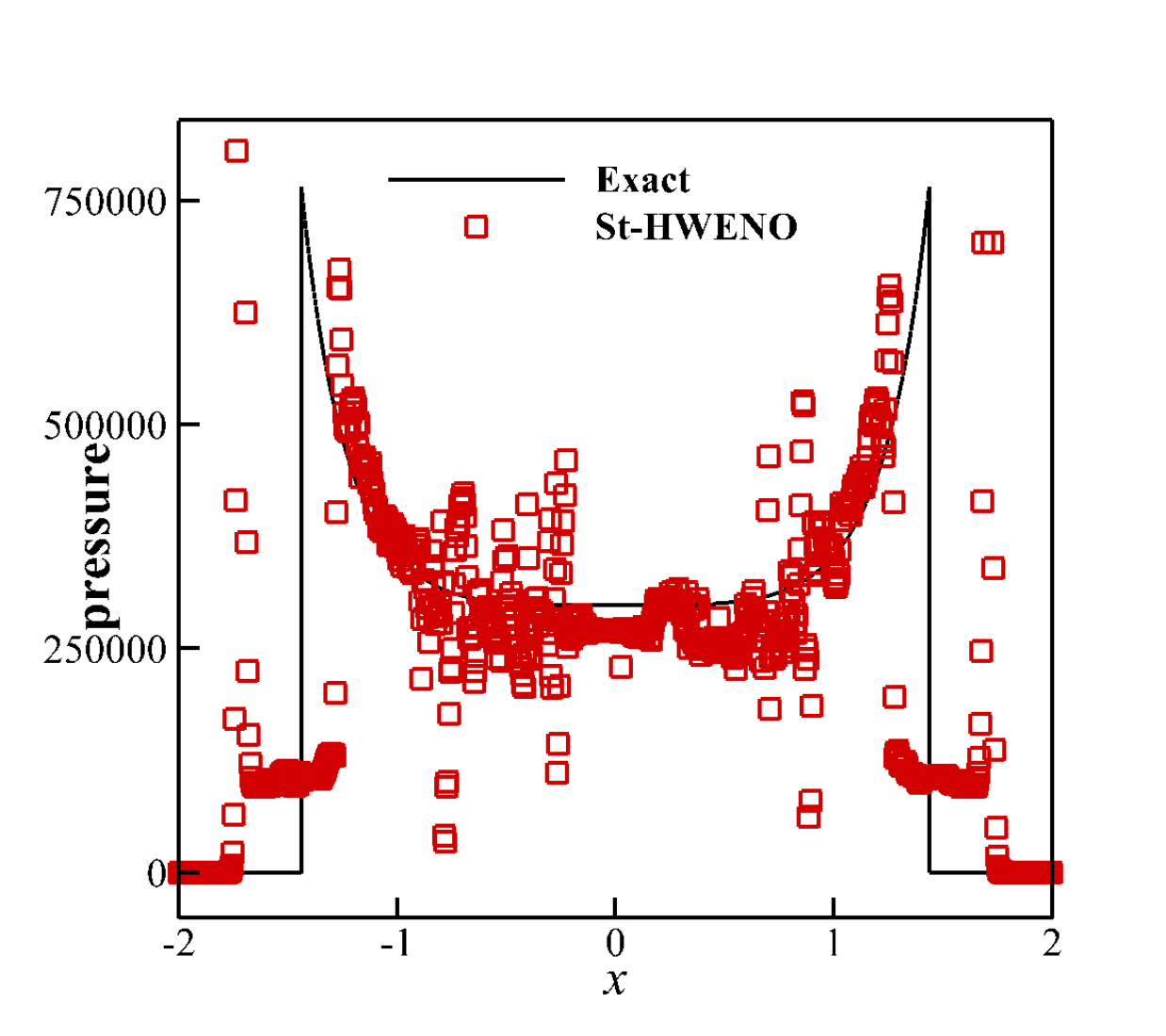}} 
		\end{subfigure} 
		\begin{subfigure}{0.32\textwidth}
			{\includegraphics[width=5.25cm,angle=0]{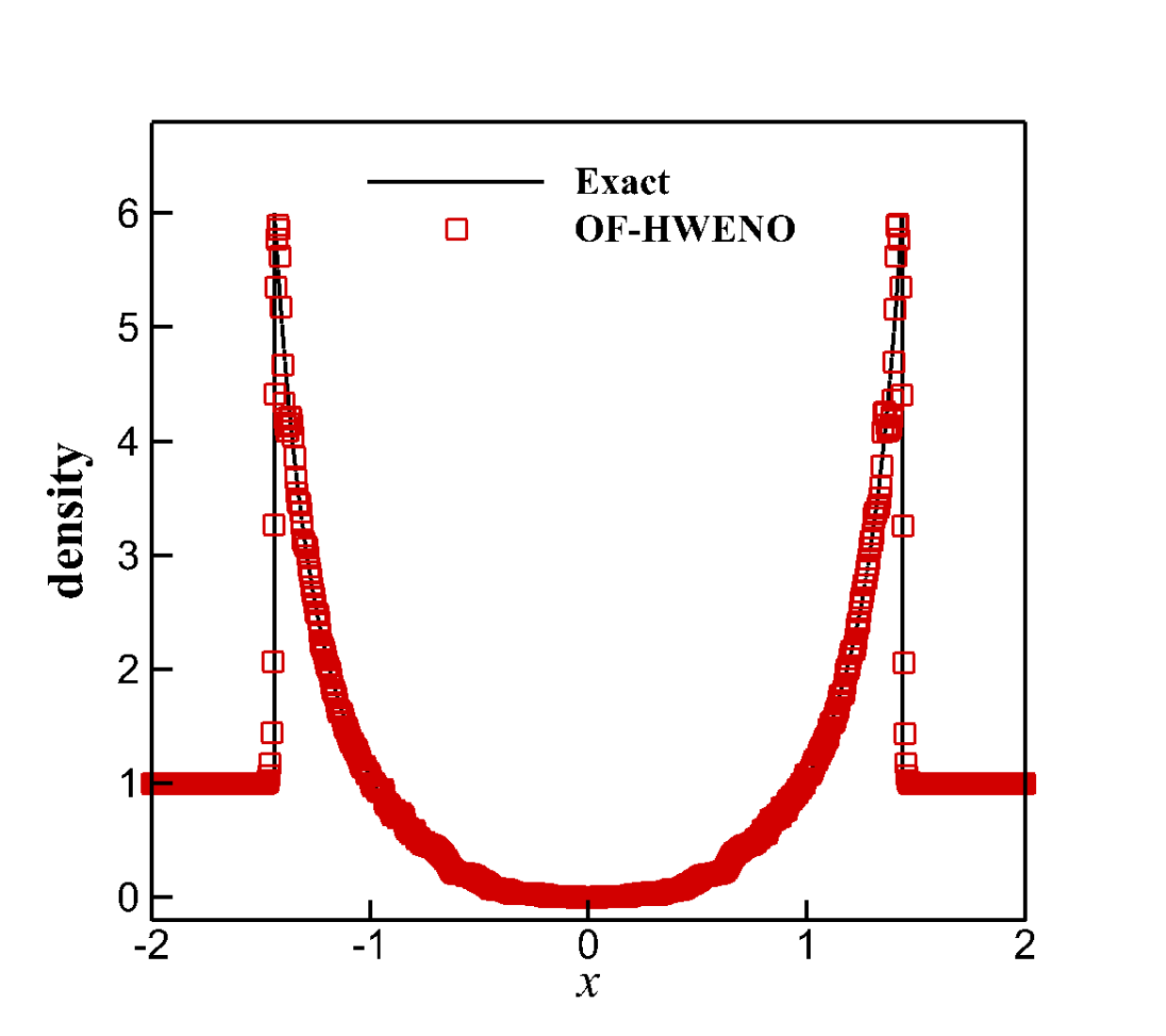}} 
		\end{subfigure} 
		\begin{subfigure}{0.32\textwidth}
			{\includegraphics[width=5.25cm,angle=0]{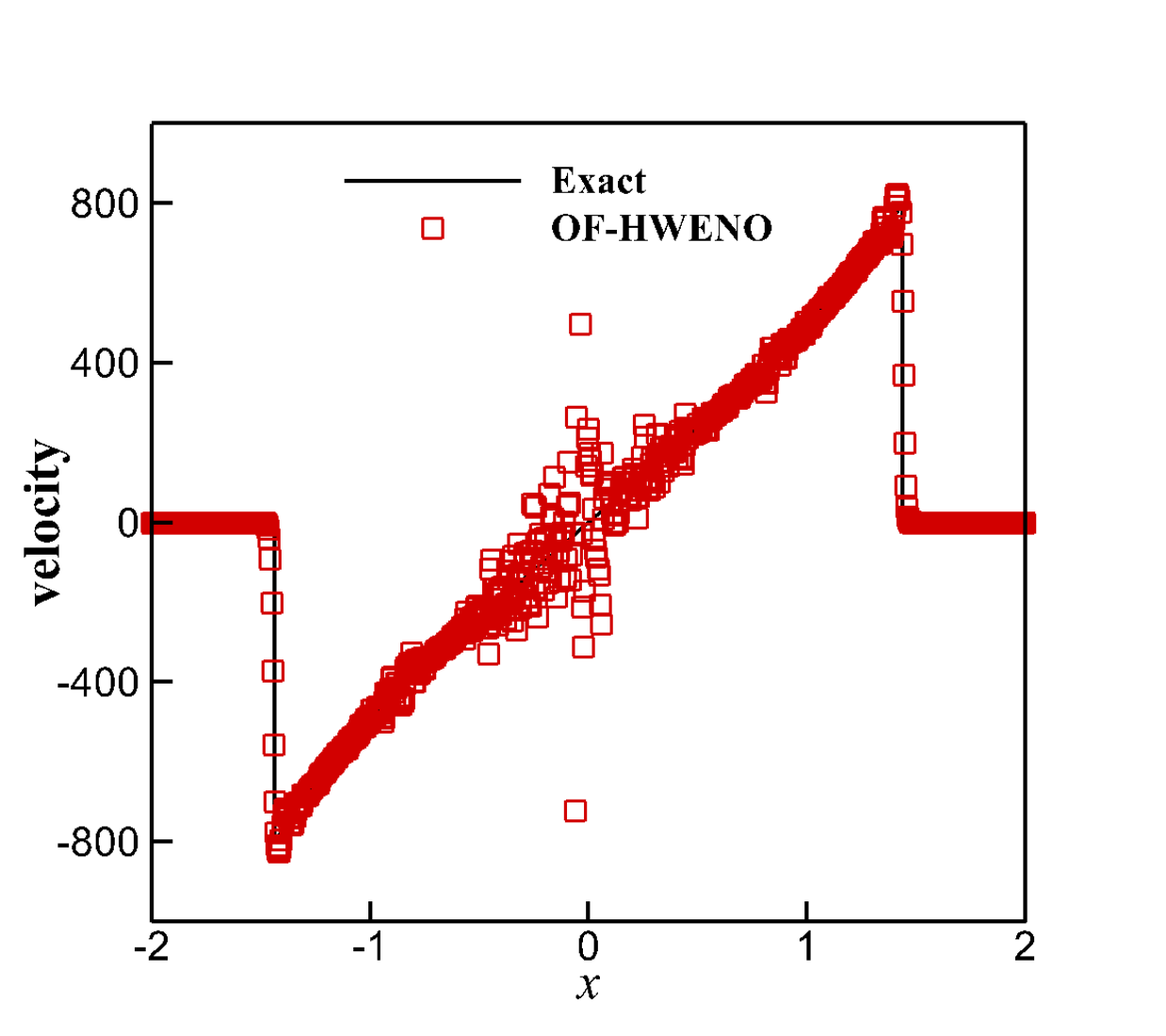}}
		\end{subfigure} 
		\begin{subfigure}{0.32\textwidth}
			{\includegraphics[width=5.25cm,angle=0]{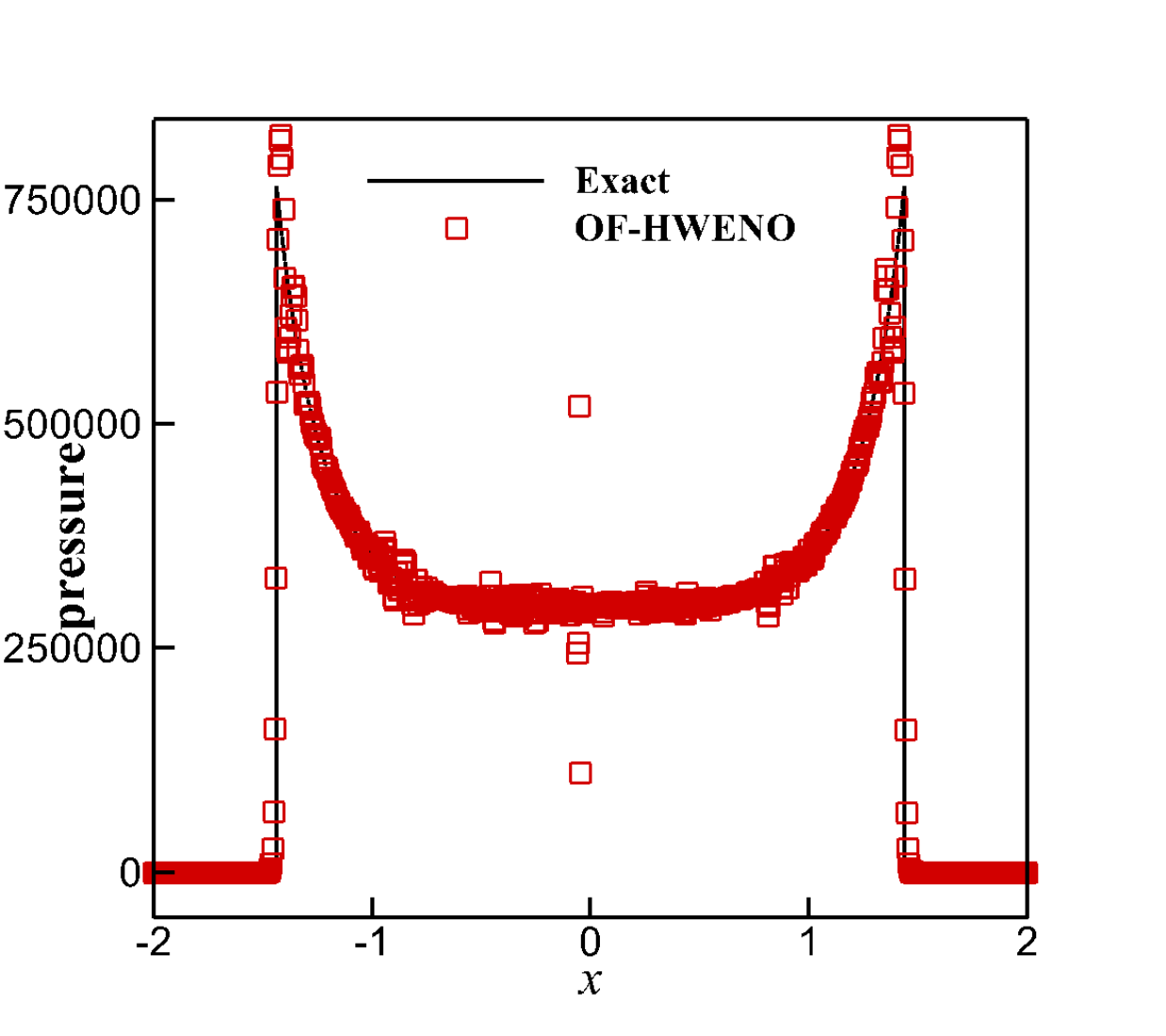}} 
		\end{subfigure} 
		\begin{subfigure}{0.32\textwidth}
			{\includegraphics[width=5.25cm,angle=0]{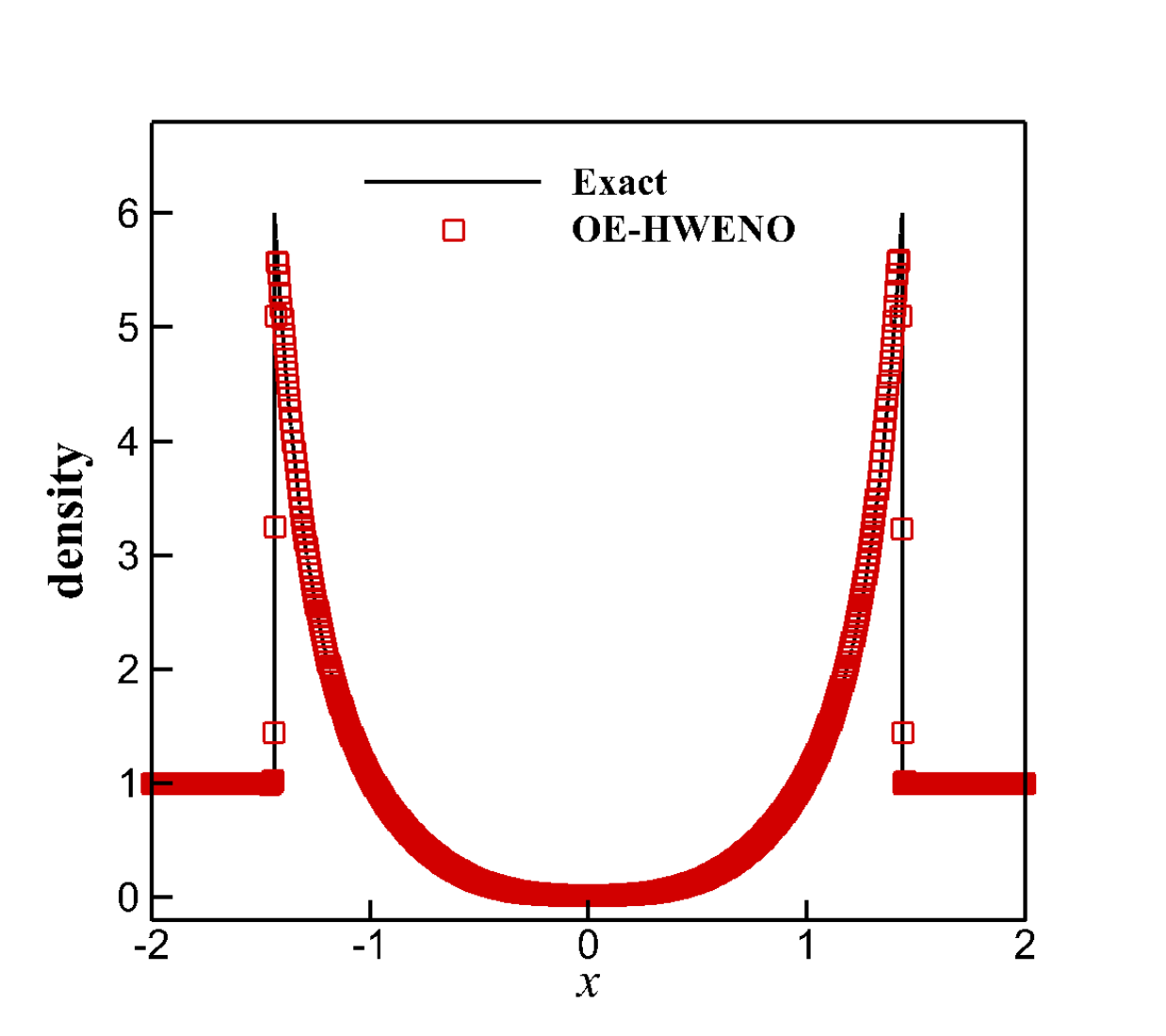}} 
		\end{subfigure} 
		\begin{subfigure}{0.32\textwidth}
			{\includegraphics[width=5.25cm,angle=0]{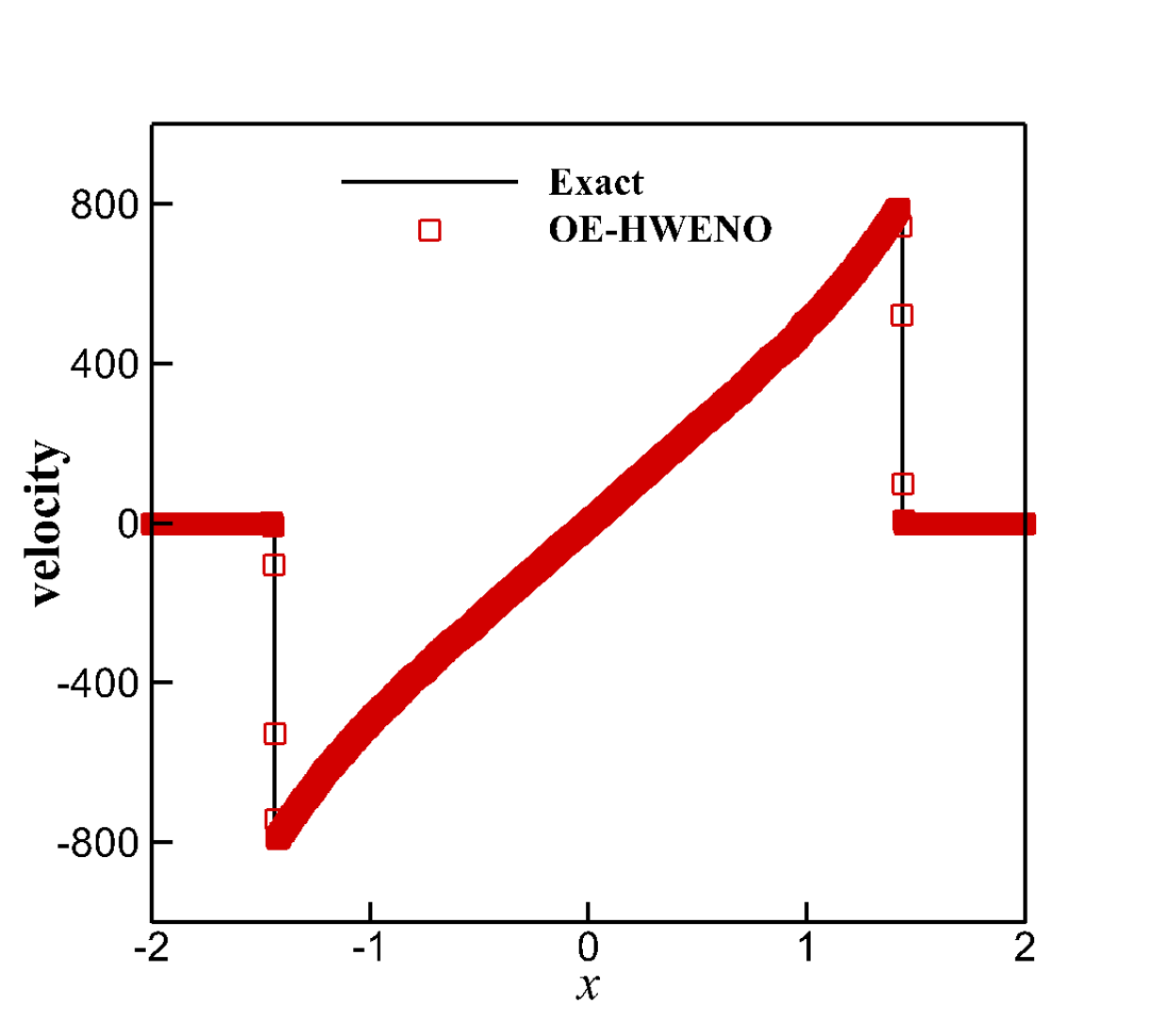}}
		\end{subfigure} 
		\begin{subfigure}{0.32\textwidth}
			{\includegraphics[width=5.25cm,angle=0]{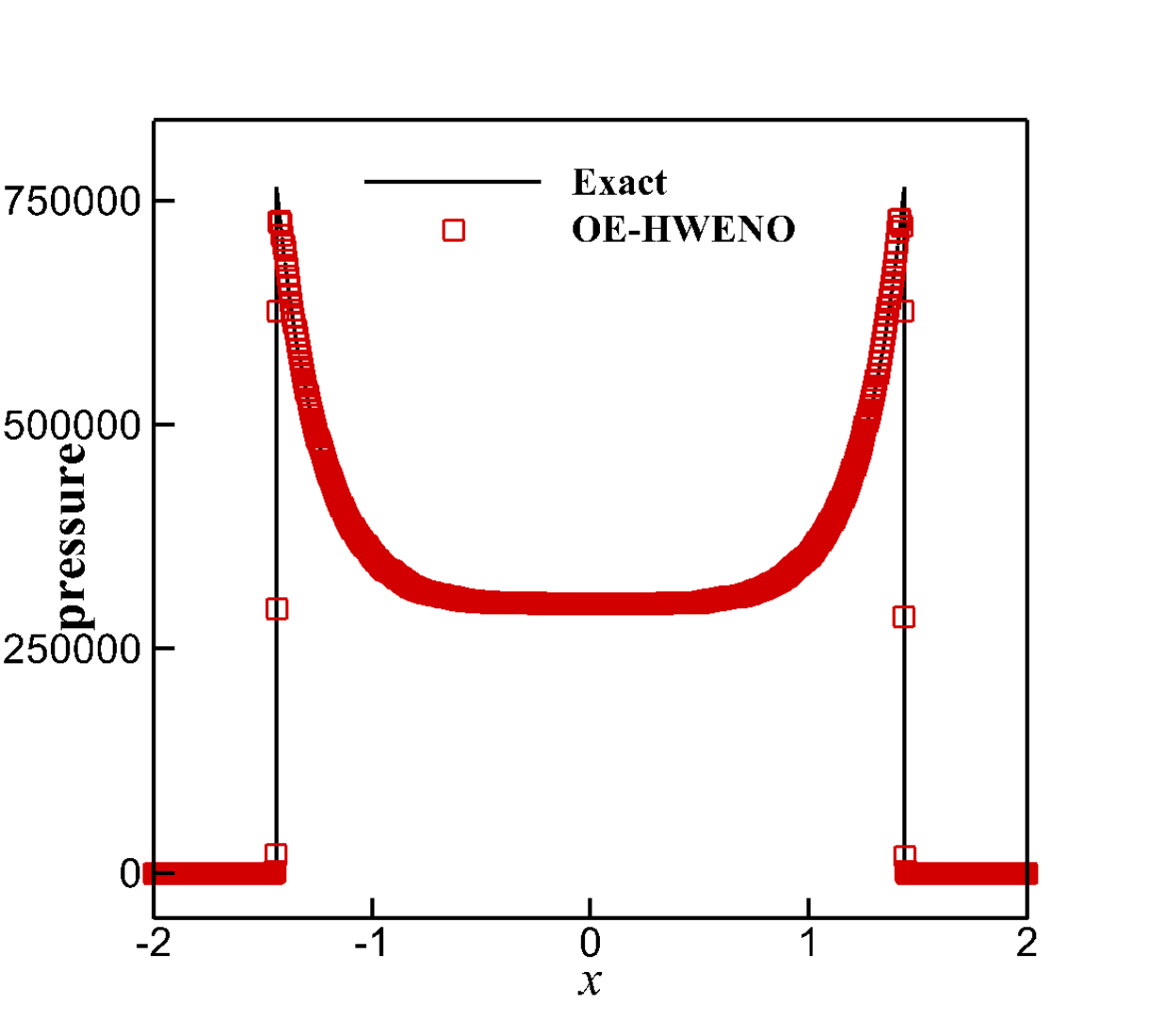}} 
		\end{subfigure}  
		\caption{Numerical results of the 1D Sedov blast wave problem computed by the St-HWENO (top), OF-HWENO (middle) and OE-HWENO (bottom) methods  with $800$ uniform cells. 
		} \label{sec3:Fig_Sedov}
	\end{figure}
	\begin{figure}[!htb]
		\centering
		{\includegraphics[width=7.0cm,angle=0]{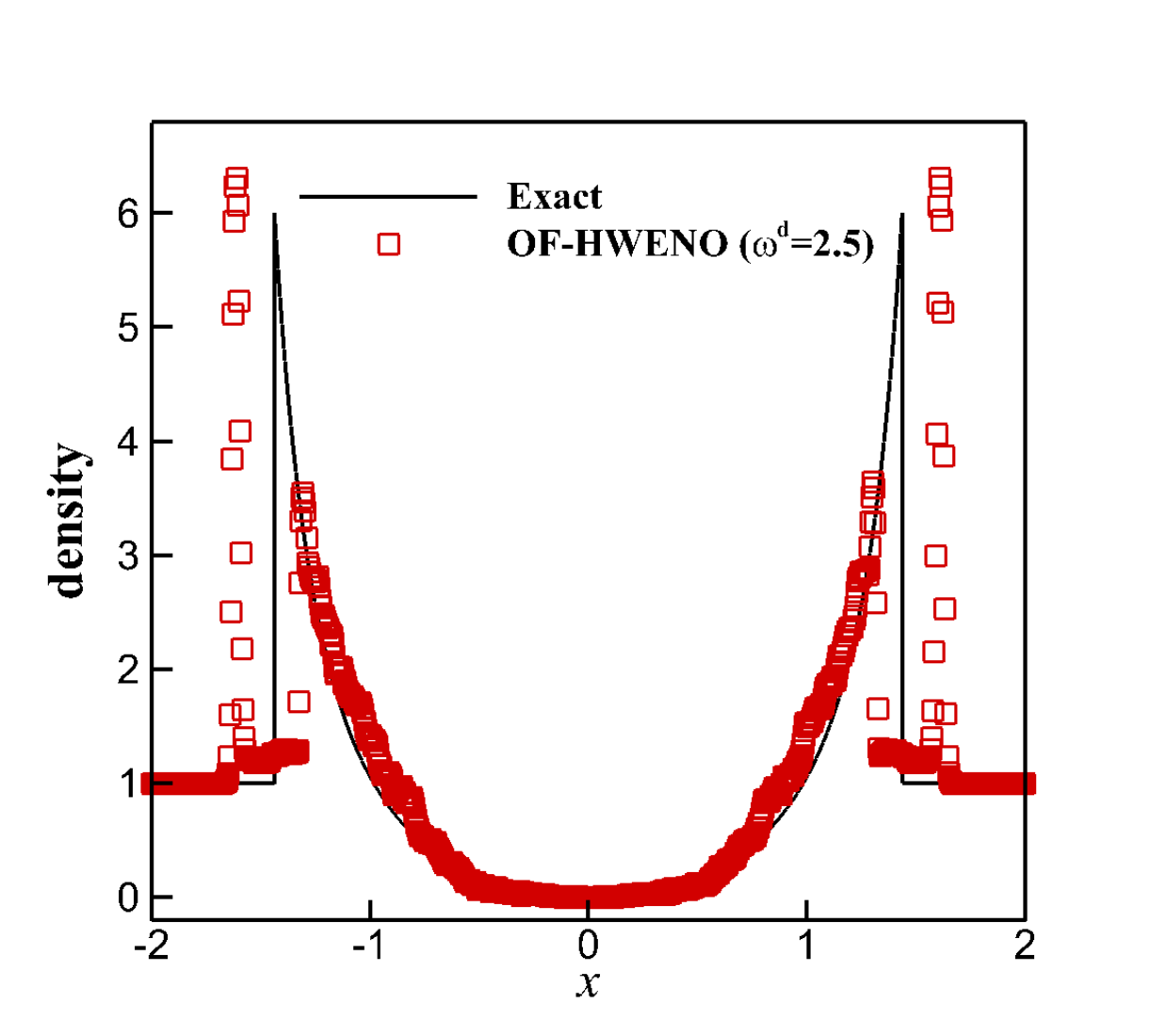}}
		{\includegraphics[width=7.0cm,angle=0]{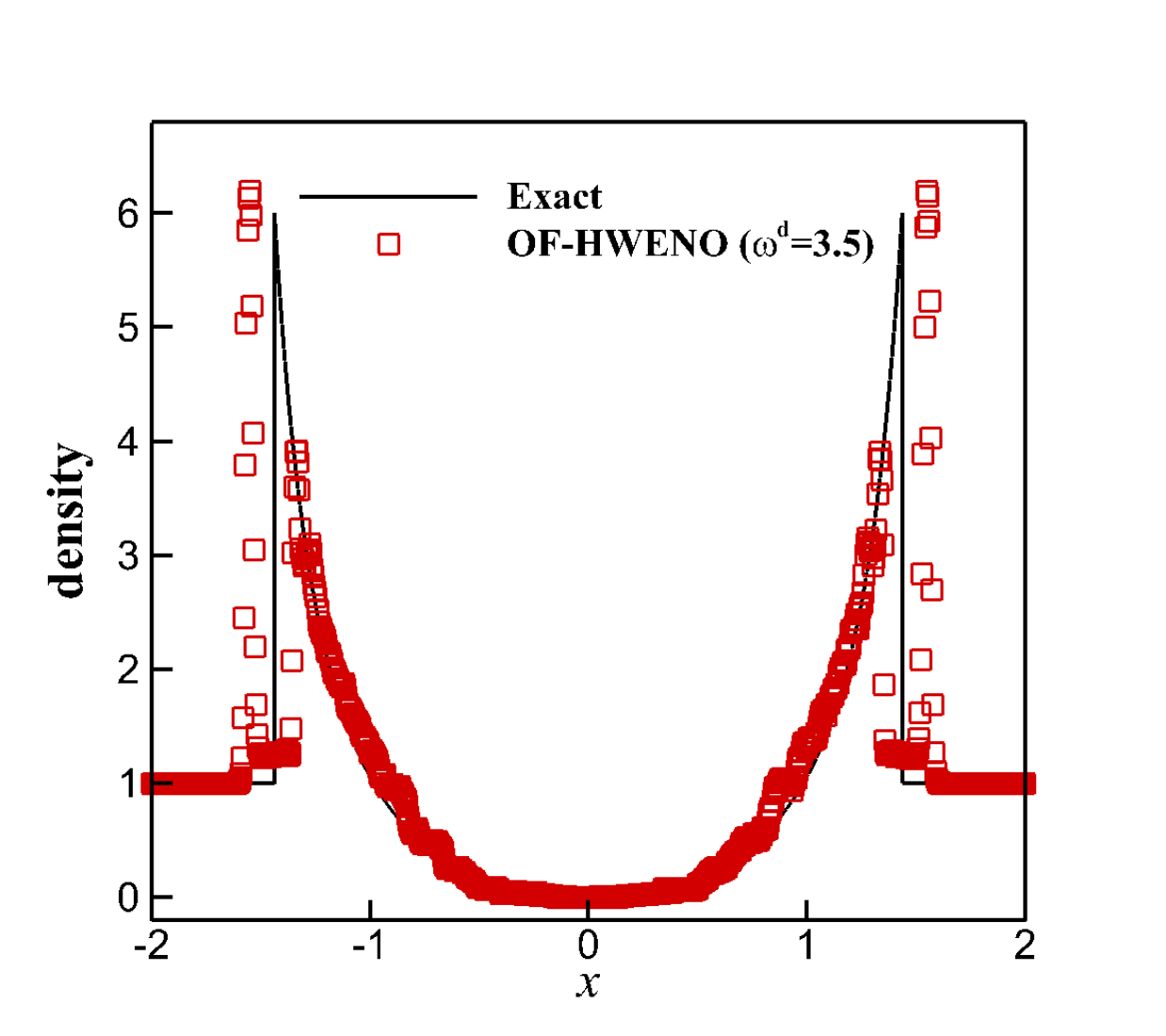}} 
		{\includegraphics[width=7.0cm,angle=0]{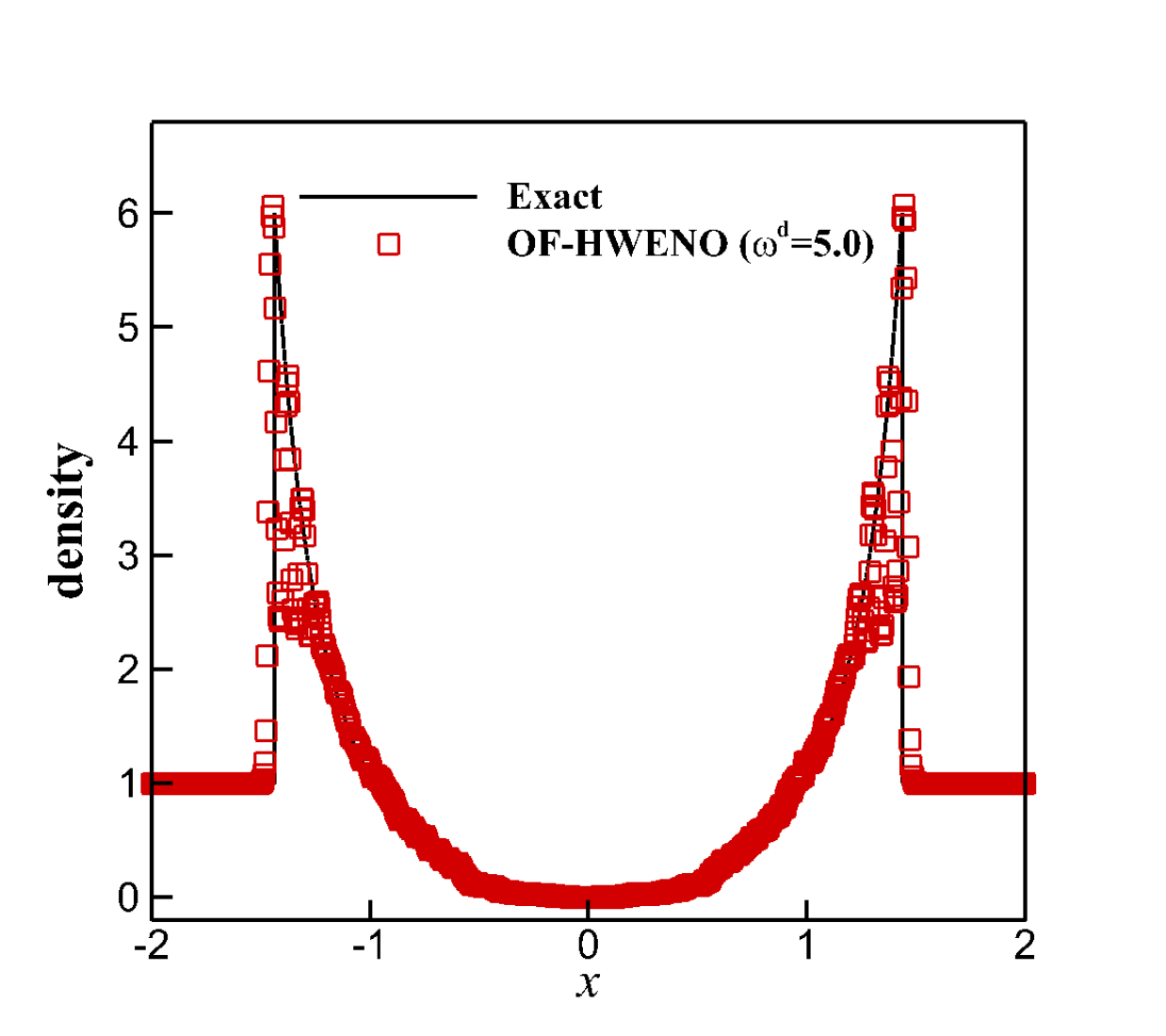}} 
		{\includegraphics[width=7.0cm,angle=0]{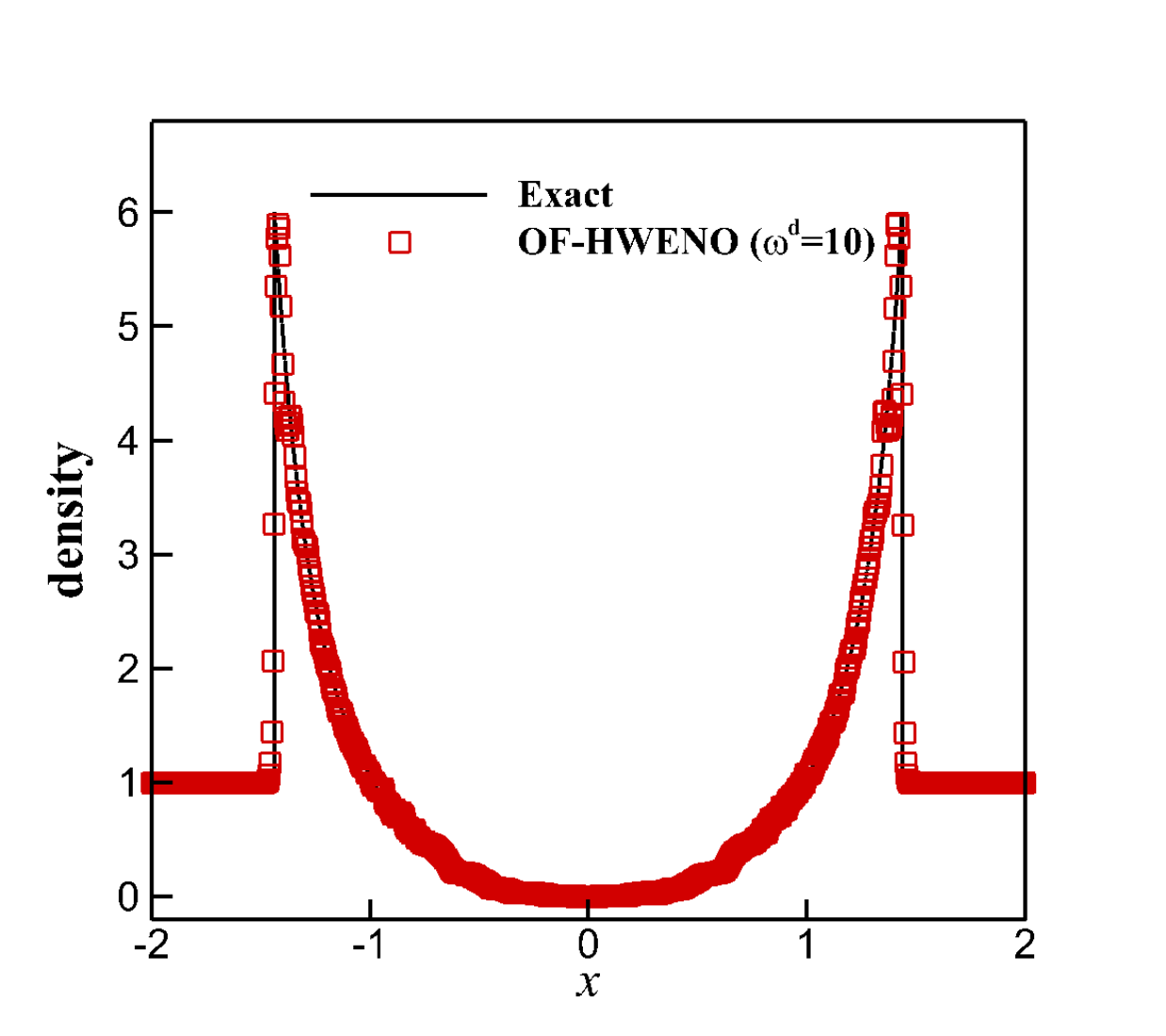}} 
		\caption{Density of the Sedov problem computed by the OF-HWENO method with $800$ uniform cells by varying the artificial parameter $\omega_d\in \{2.5, 3.5, 5, 10\}$ in the OF-HWENO method. 
		}\label{sec3:Fig_Sedov_1}
	\end{figure}
	This is an extreme problem involving very low internal energy and strong shock, which presents significant challenges in simulation. The initial condition is specified as follows
	\begin{equation*}
		(\rho_0,\mu_0,E_0)=\begin{cases}
			(1,0,10^{-12}),&x\in [-2,2] \setminus  \mbox{the center cell},
			\\(1,0,\frac{3200000}{h_x}),&x\in \mbox{the center cell}.
		\end{cases}
	\end{equation*}	
	Outflow boundary conditions are implemented on all boundaries. The exact solution is referenced in \cite{Slt,Kvp}, and the simulation runs until a final time of $T = 0.001$.  It is noteworthy that the PP limiter is necessary for maintaining the positivity of density and pressure in this example. We use ``St-HWENO" to denote the standard HWENO method without the OE procedure. For the St-HWENO and OF-HWENO methods with the PP limiter, the numerical results are presented in Fig.~\ref{sec3:Fig_Sedov}, exhibiting noticeable nonphysical oscillations. 
The damping terms of the OF-HWENO method \cite{ZQ2} contain an empirical artificial parameter $\omega^d$;  the authors of \cite{ZQ2} set $\omega^d = 3.5$ and suggested a range of $\omega^d \in [2, 5]$ for one-dimensional test cases based on their numerical experiments. However, the empirical parameter $\omega^d$ can be problem-dependent, and the challenging 1D Sedov problem, which contains large-scale variations, was not tested in \cite{ZQ2}. 
	The result of the OF-HWENO method in Fig.~\ref{sec3:Fig_Sedov} is obtained by setting $\omega^d = 10$. We make  this choice because we have conducted the 1D Sedov test by varying the artificial parameter $\omega^d \in \{2.5, 3.5, 5, 10\}$ in the OF-HWENO method. The comparison shown in Fig.~\ref{sec3:Fig_Sedov_1} indicates that the parameter $\omega^d = 10$ offers the best performance for the OF-HWENO method in the 1D Sedov test. 
	Fig.~\ref{sec3:Fig_Sedov_1} shows that the OF-HWENO method is sensitive to this artificial empirical parameter $\omega^d$, and different values of $\omega^d$ lead to various results with non-physical oscillations.  Conversely, our proposed OE-HWENO method with the PP limiter effectively handles this problem, producing essentially non-oscillatory results with high resolution, as depicted at the bottom of Fig.~\ref{sec3:Fig_Sedov}.
\end{example}

\begin{example}[Double rarefaction wave problem]\label{Sec3:Example_DoubleRarefaction} 
	This is an extreme problem containing very low density and pressure. The initial condition is given by
	\begin{equation*}
		(\rho_0,\mu_0,p_0)=\begin{cases}
			(7,-1,0.2),&-1<x<0,
			\\(7,1,0.2),&0<x<1.
		\end{cases}
	\end{equation*}	
	Outflow boundary conditions are applied to all boundaries, and the simulation is carried out up to the final time $T = 0.6$. It is noteworthy that the OE-HWENO method without the PP limiter also works well for this problem, and its numerical results are depicted in Fig.~\ref{sec3:Fig_DR}.
	\begin{figure}[!htb]
		\centering 
		\begin{subfigure}{0.32\textwidth}
			{\includegraphics[width=5.25cm,angle=0]{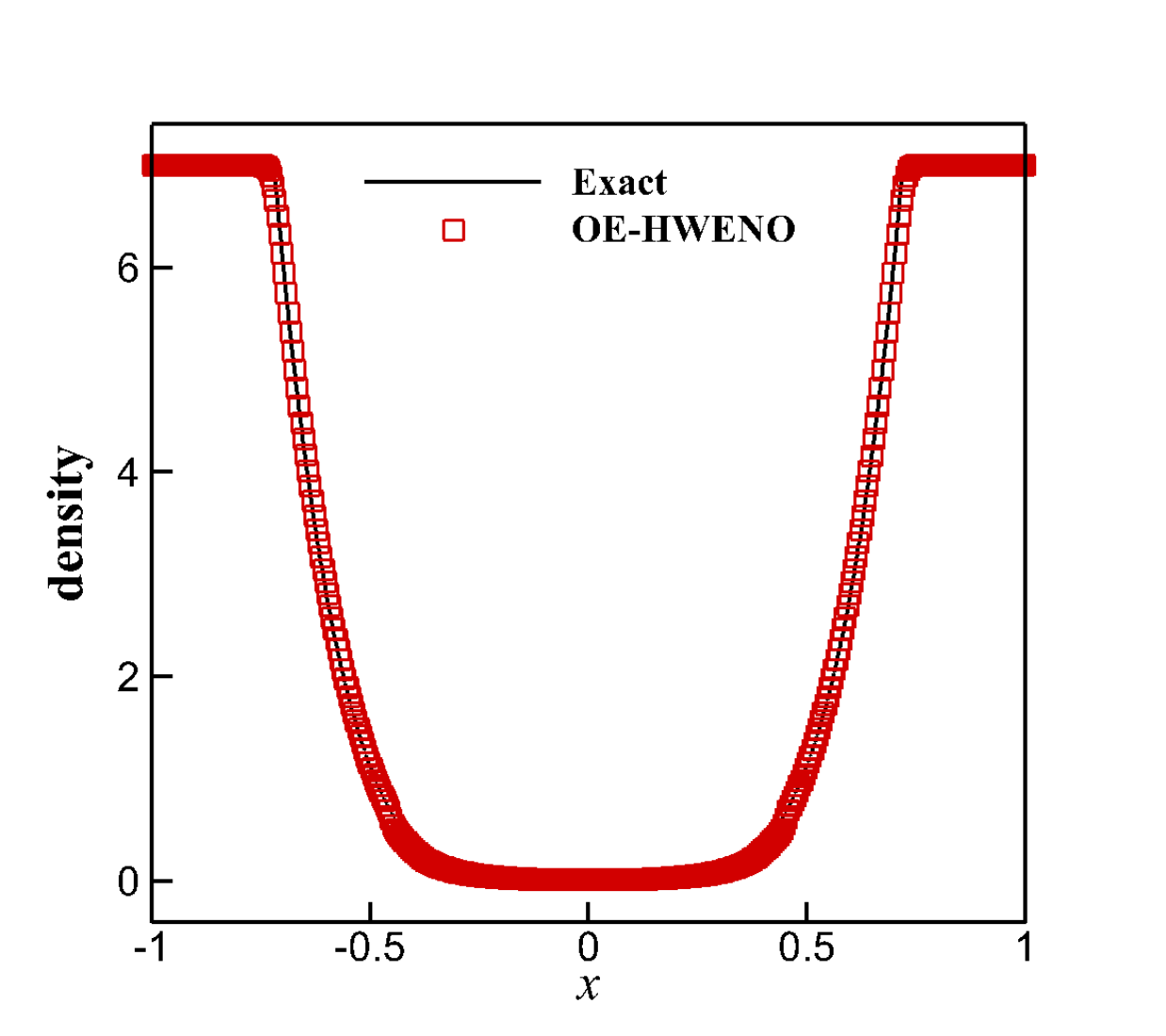}}
		\end{subfigure} 
		\begin{subfigure}{0.32\textwidth}
			{\includegraphics[width=5.25cm,angle=0]{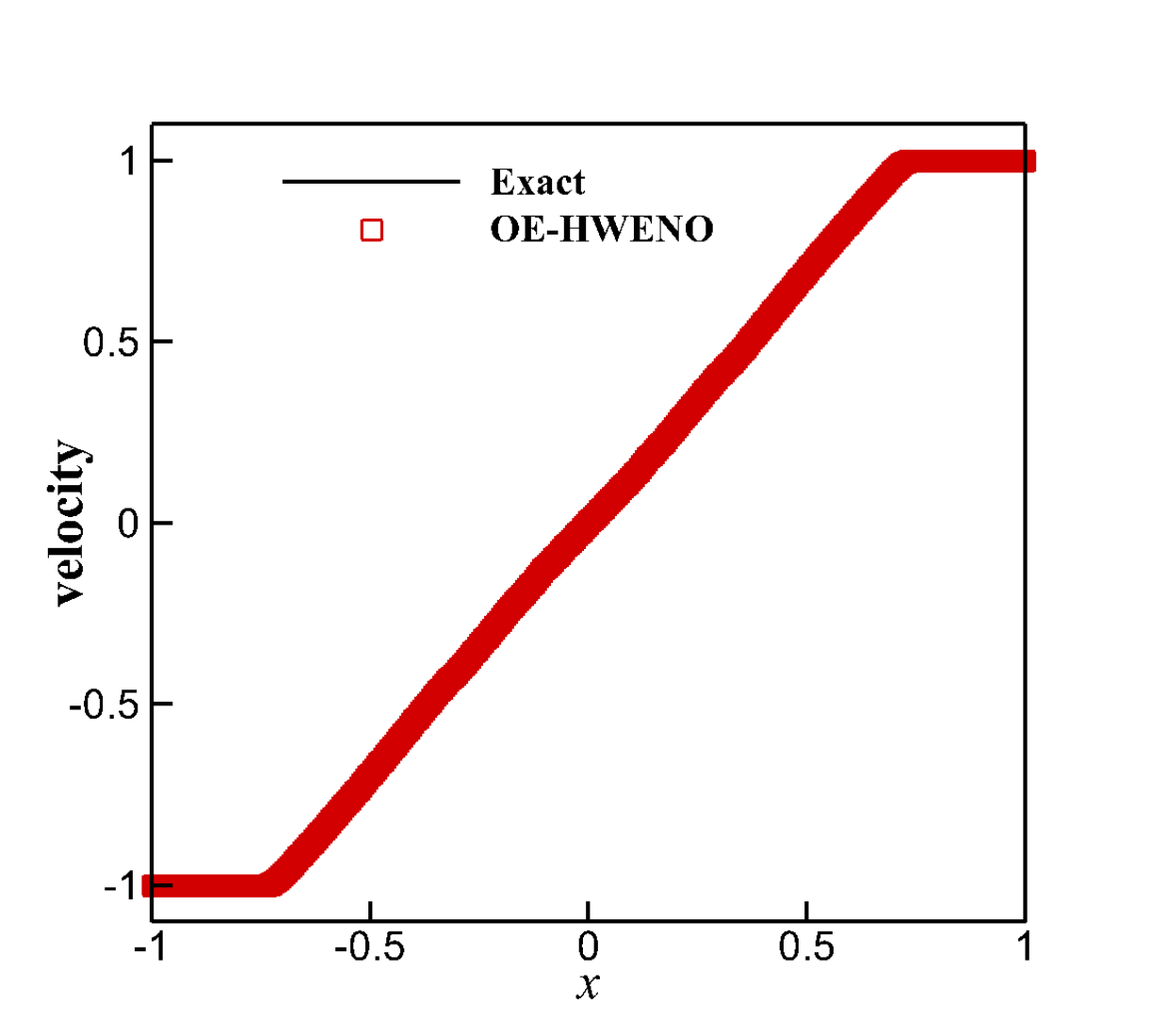}}
		\end{subfigure} 
		\begin{subfigure}{0.32\textwidth}
			{\includegraphics[width=5.25cm,angle=0]{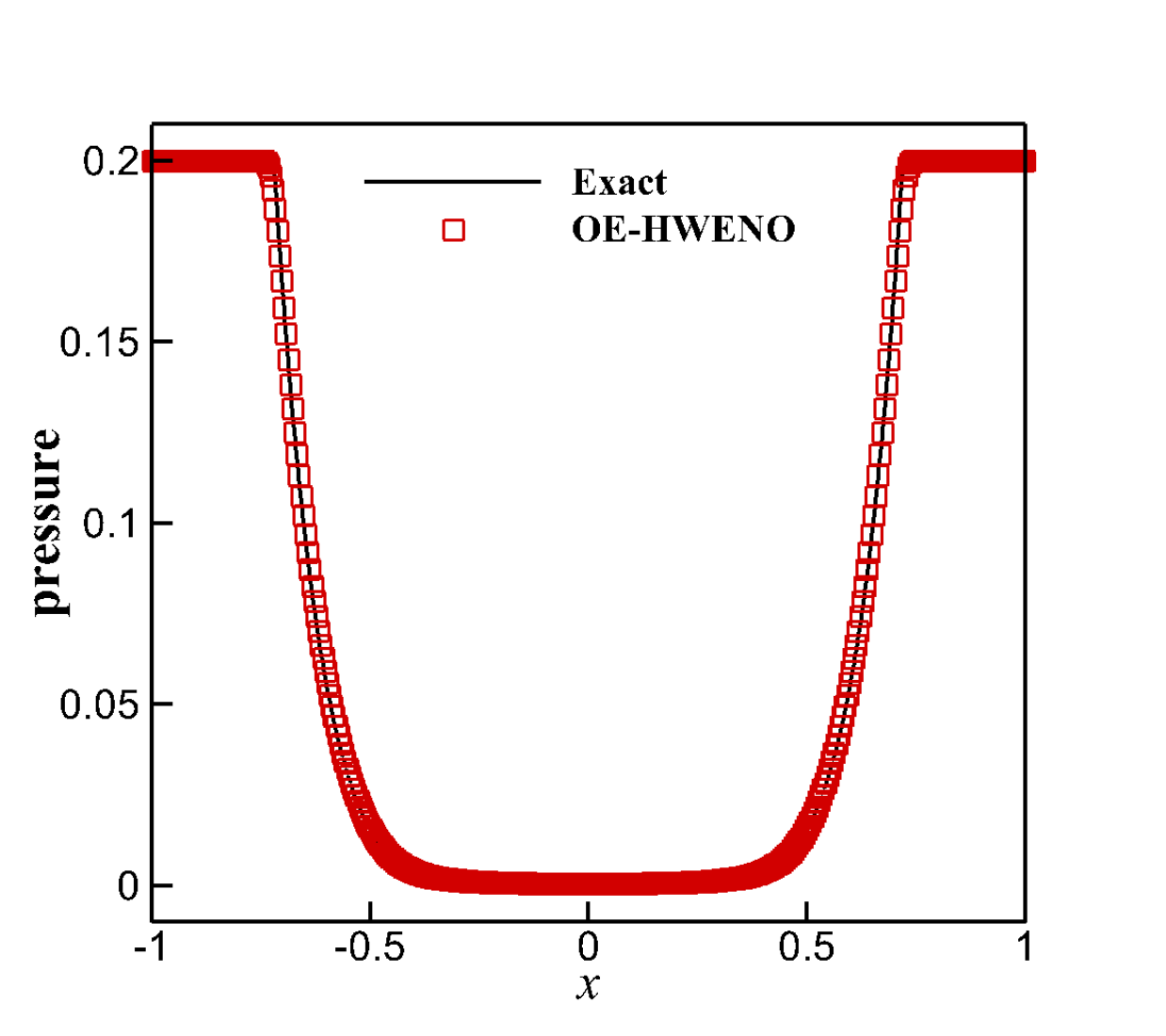}}
		\end{subfigure}  
		\caption{Numerical results of the double rarefaction wave problem computed by the OE-HWENO method with $400$ uniform cells.  
		} \label{sec3:Fig_DR}
	\end{figure}
\end{example}

\begin{example}[Leblanc problem]\label{sec3:Example_Leblanc}  
	This problem also presents extreme conditions characterized by very low internal energy and a strong shock, which poses significant challenges for robust simulation. The initial condition is specified as
	\begin{equation*}
		(\rho_0,\mu_0,p_0)=\begin{cases}
			(2,0,10^{9}),&-10<x<0,
			\\(10^{-3},0,1),&0<x<10.
		\end{cases}
	\end{equation*}	
	Outflow boundary conditions are applied to all boundaries, and the simulation is conducted until a final time of $T = 0.0001$. The computational results for the OE-HWENO method with the PP limiter are illustrated in Fig.~\ref{sec3:Fig_Leblanc}. Our method demonstrates robust performance with essentially non-oscillatory, high-resolution output for this demanding problem.
	\begin{figure}[!htb]
		\centering 
		\begin{subfigure}{0.32\textwidth}
			{\includegraphics[width=5.25cm,angle=0]{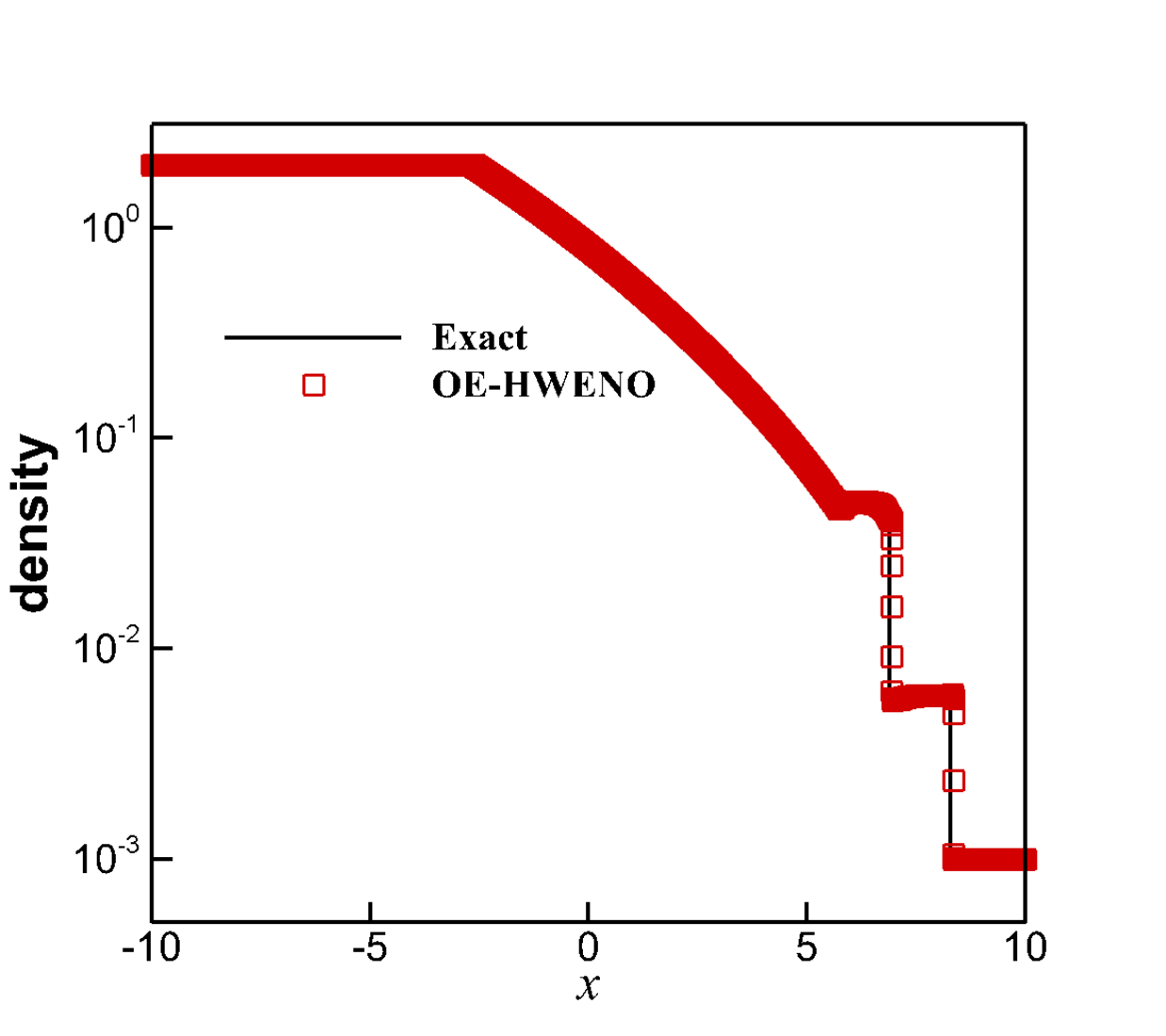}}
		\end{subfigure} 
		\begin{subfigure}{0.32\textwidth}
			{\includegraphics[width=5.25cm,angle=0]{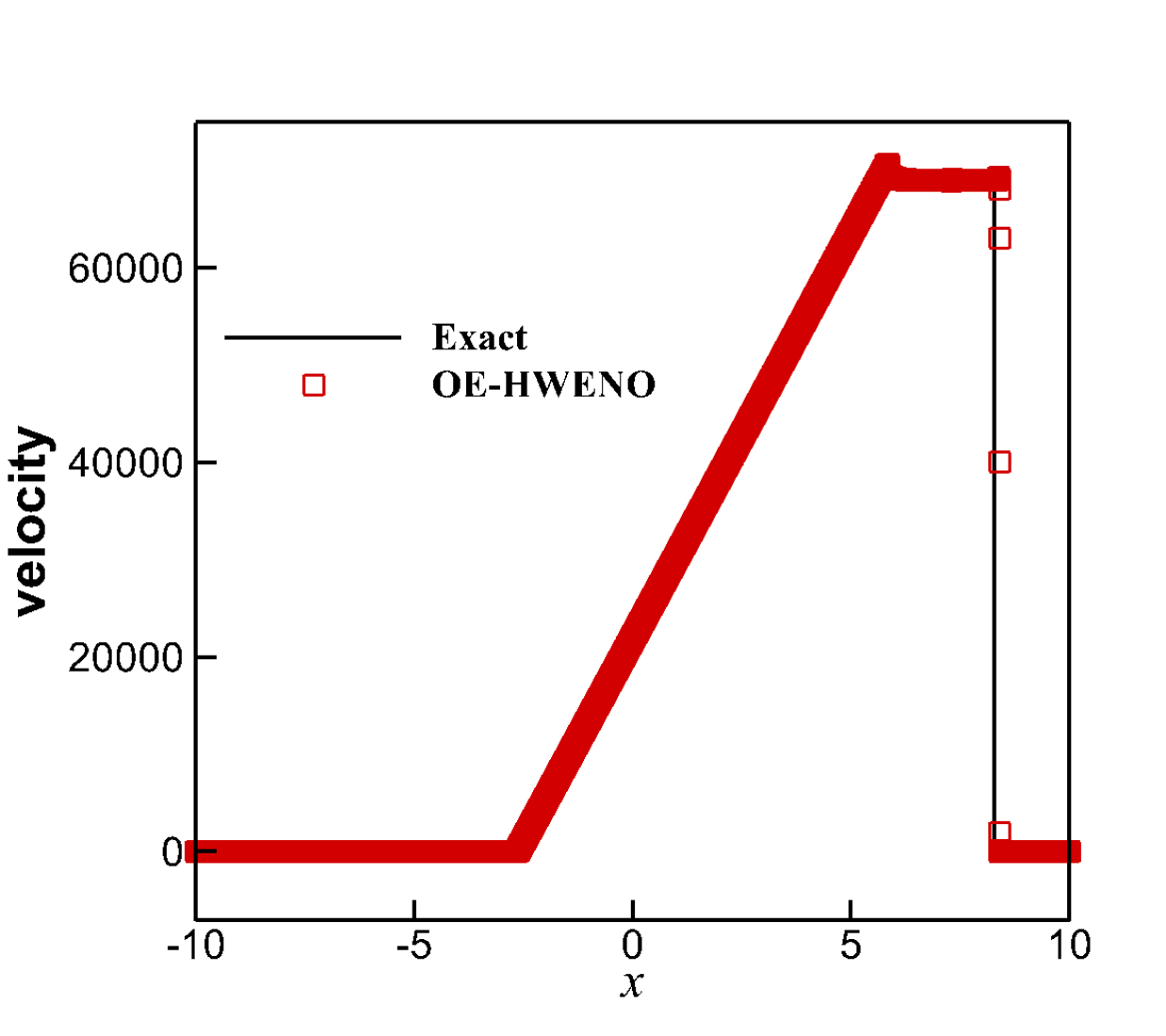}}
		\end{subfigure} 
		\begin{subfigure}{0.32\textwidth}
			{\includegraphics[width=5.25cm,angle=0]{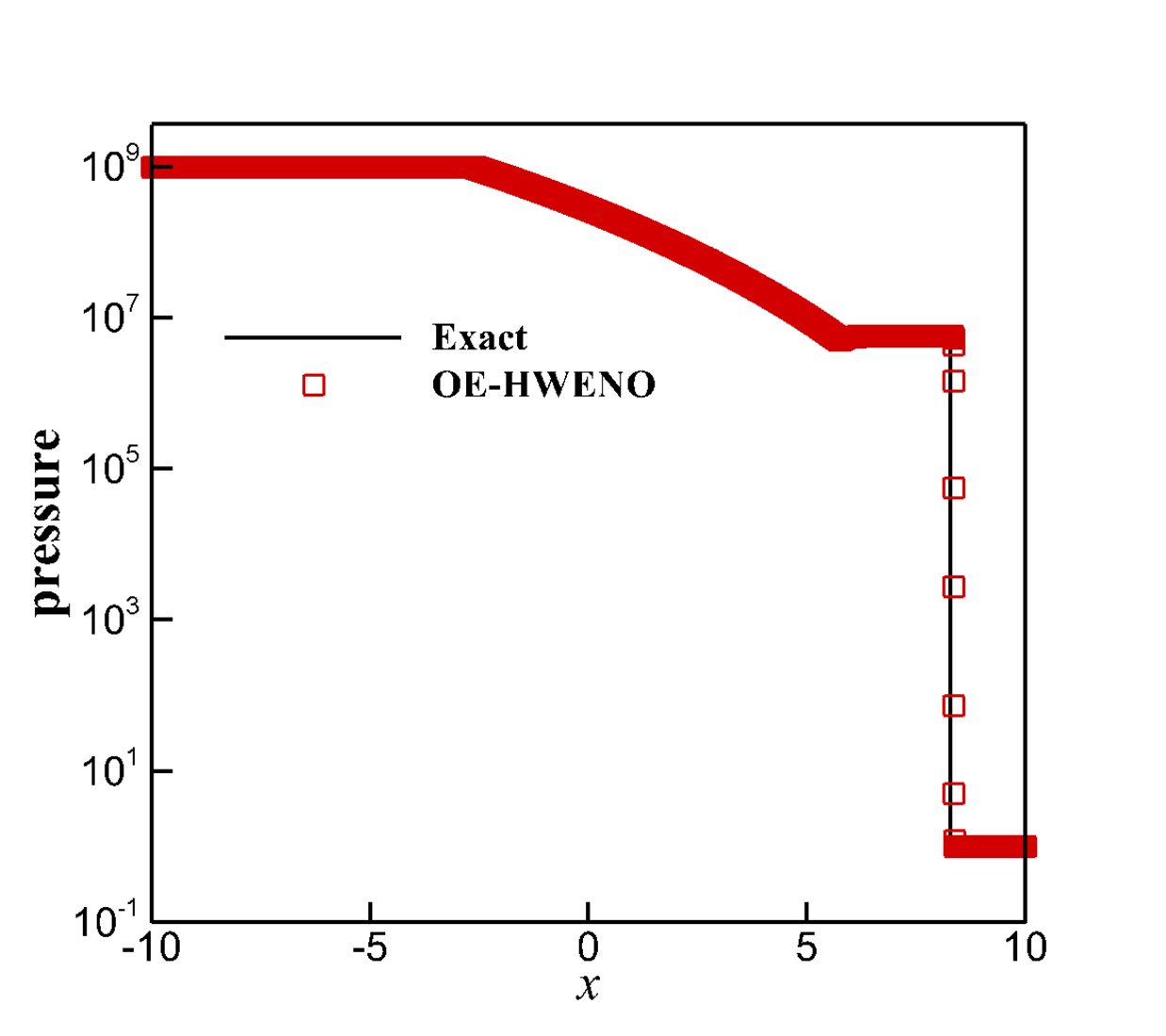}}
		\end{subfigure} 
		\caption{Numerical results of the Leblanc problem obtained  by the OE-HWENO method with $6400$ uniform cells.  
		} \label{sec3:Fig_Leblanc}
	\end{figure}
\end{example}

\begin{example}[2D Riemann problem]\label{Sec3:Example_2dRiemann}  
	\begin{figure}[!htb]
		\centering
		\begin{subfigure}{0.49\textwidth}
			{\includegraphics[width=7.5cm,angle=0]{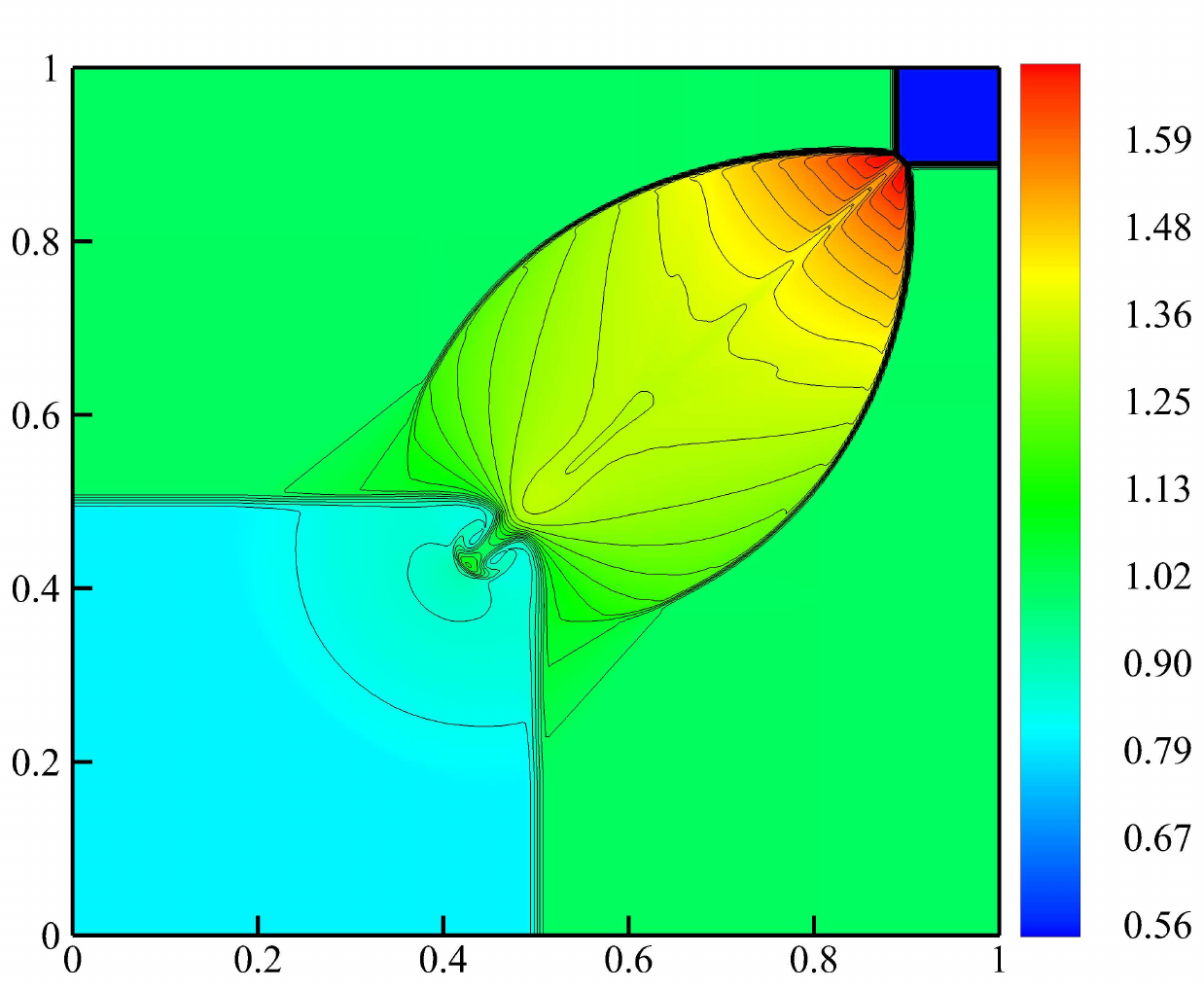}}
			\caption{OE-HWENO. $\lambda=1$}
		\end{subfigure} 
		\begin{subfigure}{0.49\textwidth}
			{\includegraphics[width=7.5cm,angle=0]{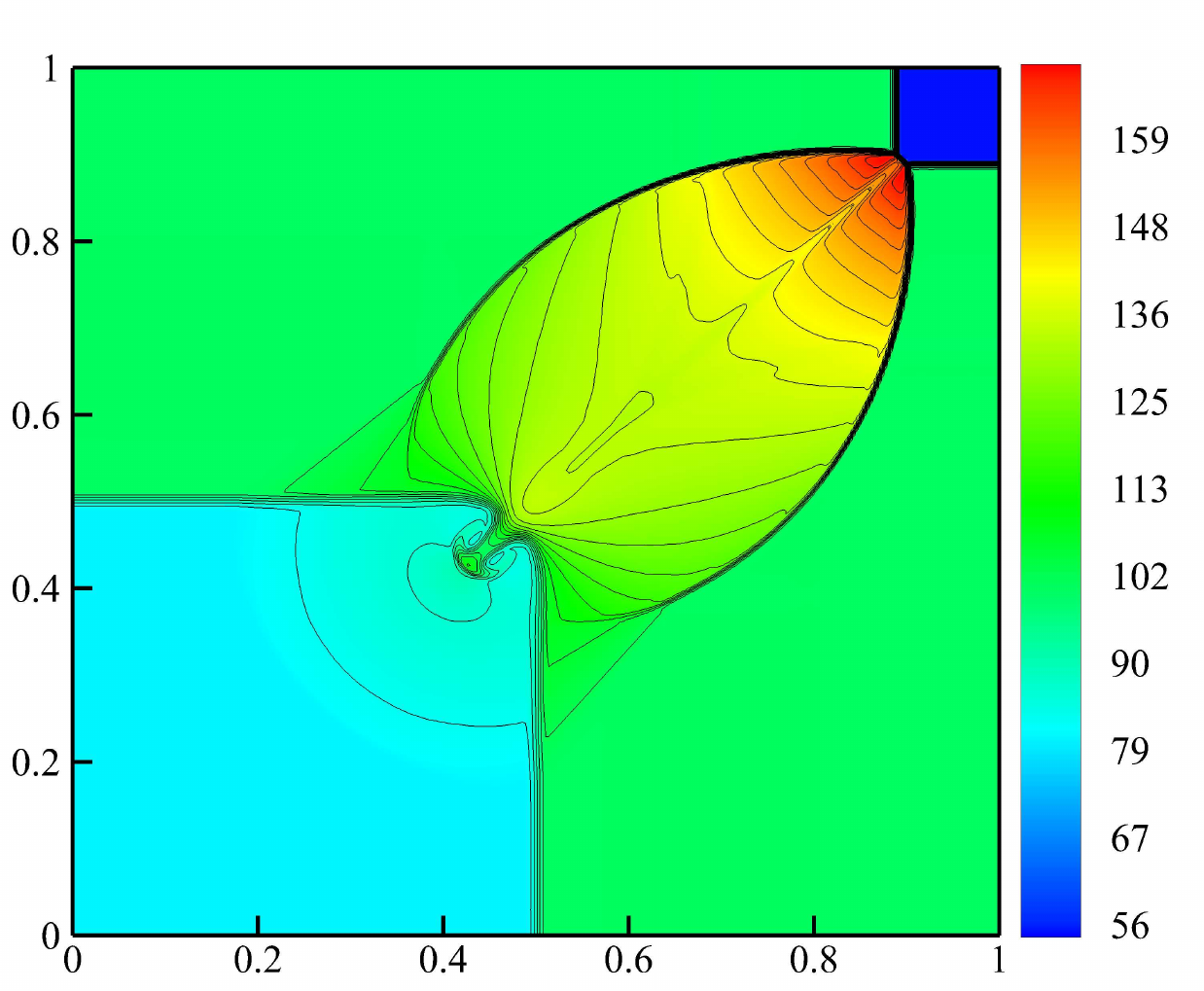}}
			\caption{OE-HWENO. $\lambda=100$}
		\end{subfigure}\vspace{-1pt}
		\begin{subfigure}{0.49\textwidth}
			{\includegraphics[width=7.5cm,angle=0]{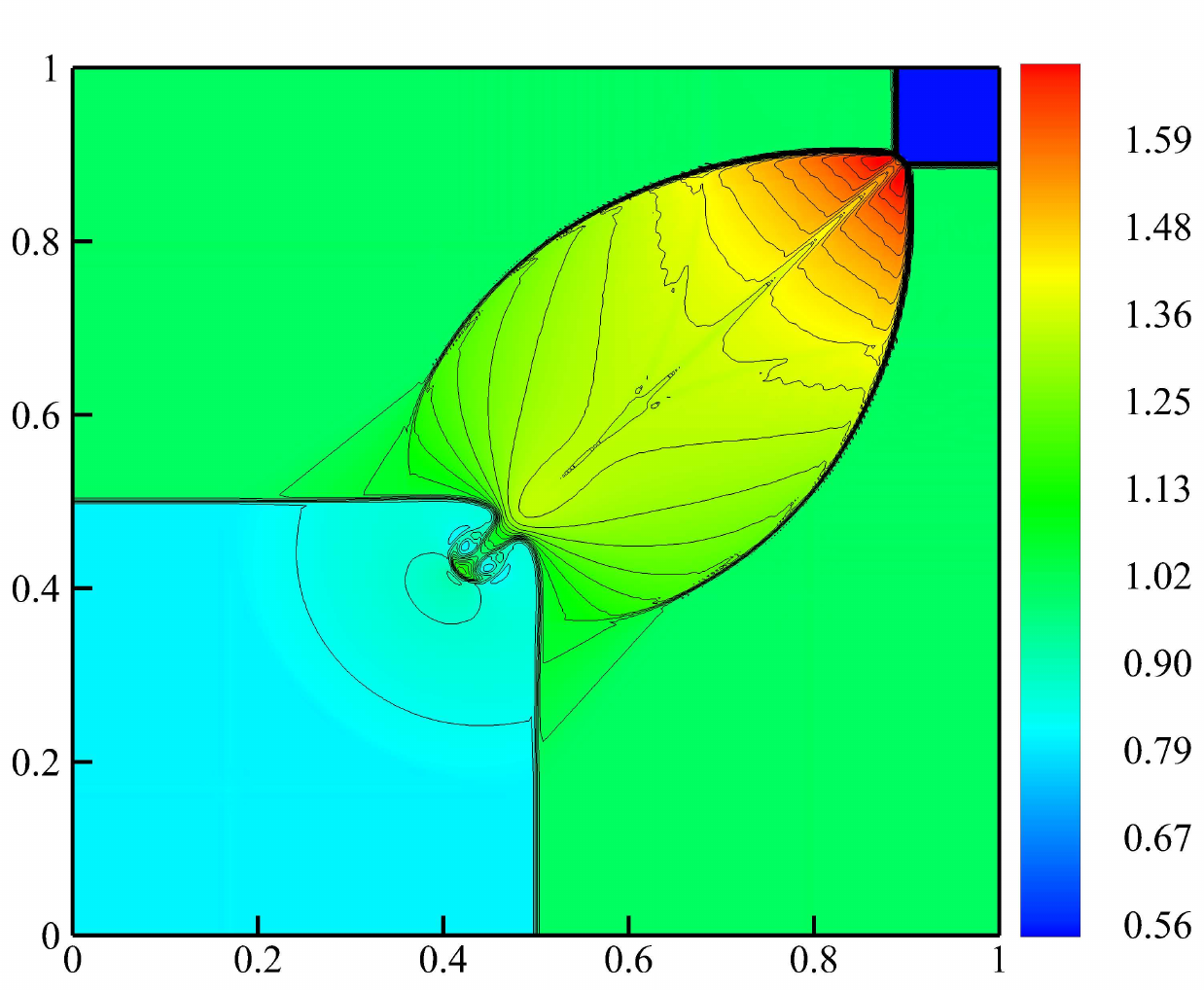}}
			\caption{OF-HWENO. $\lambda=1$}
		\end{subfigure} 
		\begin{subfigure}{0.49\textwidth}
			{\includegraphics[width=7.5cm,angle=0]{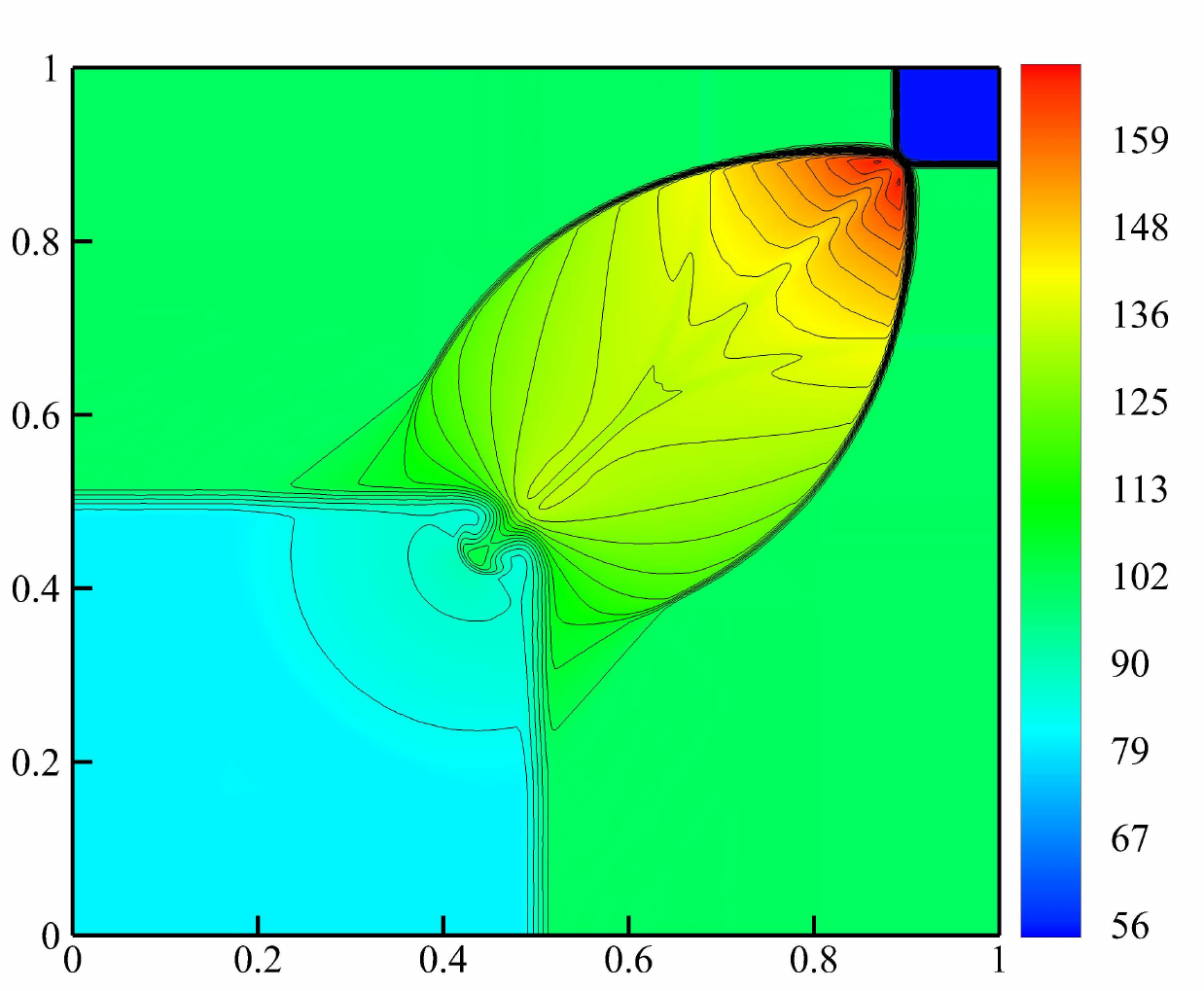}}
			\caption{OF-HWENO. $\lambda=100$}
		\end{subfigure}
		\caption{Contour plots of density computed by the OE-HWENO and OF-HWENO methods with $320\times320$ uniform cells: 30 equally spaced lines from 0.56 to 1.67 (left) and from 56 to 167 (right).  
		}\label{sec3:Fig_Riemann2d}
	\end{figure}
	This test involves two stationary contact discontinuities and two shocks for the 2D compressible Euler equations in the domain $[0,1]^2$. To verify the scale-invariant property of the 2D OE-HWENO method, we take the scaled initial data $\bm{u}^\lambda(x,y,0)=\lambda\bm{u}_0(x,y,0)$, where
	$\bm{u}_0(x,y)=(\rho_0,\rho_0\mu_0,\rho_0\nu_0,E_0)$ is defined by
	\begin{equation*}
		(\rho_0,\mu_0,\nu_0,p_0)=\begin{cases}
			(0.8,0,0,1),&x<0.5,y<0.5,
			\\
			(1,0.7276,0,1),&x<0.5,y>0.5,
			\\
			(1,0,0.7276,1),&x>0.5,y<0.5,
			\\
			(0.5313,0,0,4),&x>0.5,y>0.5,
		\end{cases}
	\end{equation*}	
	Outflow boundary conditions are imposed on all boundaries. The numerical results of density computed by the OE-HWENO and OF-HWENO methods at time $T = 0.25$ are presented in Fig.~\ref{sec3:Fig_Riemann2d} for two distinct scales ($\lambda=1$ and $\lambda=100$). For the case of the normal scale $\lambda=1$, both the OE-HWENO and OF-HWENO methods yield good results, although the OF-HWENO solution produces a few non-physical oscillations near $(0.75,0.65)$. For $\lambda=100$, the OF-HWENO method excessively smears and smooths out the detailed features near discontinuities. In contrast, the OE-HWENO method consistently produces satisfactory results, as it is scale-invariant and agrees with those obtained under the normal scale. This further demonstrates the superiority of the OE-HWENO method.
\end{example}

\begin{example}[Double Mach reflection problem.]\label{Sec3:Example_2dDM}  
	This is a benchmark test \cite{WC} involving strong shocks and their interactions, modeled by the 2D Euler equations. The computational domain is $[0,4]\times[0,1]$, and the initial condition is
	\begin{equation*}
		(\rho_0,\mu_0,\nu_0,p_0)=
		\begin{cases}
			(8,\frac{33}{4}\sin(\frac{\pi}{3}),-\frac{33}{4}\cos(\frac{\pi}{3}),116.5),&x<\frac{1}{6}+\frac{y}{\sqrt{3}},
			\\(1.4,0,0,1),&\mbox{otherwise}.
		\end{cases}
	\end{equation*}
	Inflow and outflow boundary conditions are imposed on the left and right boundaries, respectively. For the bottom boundary, the exact post-shock condition is imposed from $x=0$ to $x=\frac{1}{6}$, and a reflection wall is used for the rest. Regarding the upper boundary, the exact motion of a Mach 10 shock is specified, with the postshock state from $x = 0$ to $x =\frac{1}{6}+\frac{1}{\sqrt{3}}(1+20t)$ and the remaining portion in the preshock state. The resulting density contour plot at the final time $T = 0.2$ is presented in Fig.~\ref{sec3:Fig_DoubleMach} for the OE-HWENO method. The intricate flow characteristics, including the double Mach region, are clearly resolved with high resolution, demonstrating the good performance of the OE-HWENO method.
	\begin{figure}[!htb]
		\centering
		{\includegraphics[width=15.0cm,height=5.0cm,angle=0]{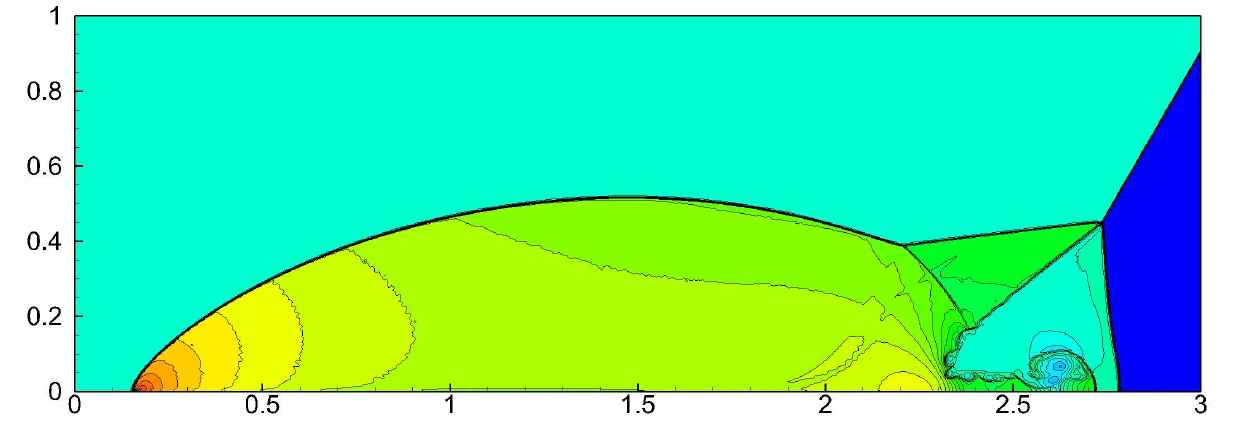}}%
		\caption{Density of double Mach reflection simulated by the OE-HWENO method with $1600\times400$ uniform cells.  
		} 
	\end{figure}\label{sec3:Fig_DoubleMach}
\end{example}

\begin{example}[2D Sedov problem]\label{Sec3:Example_2dSedov}  
	This is a challenging benchmark problem \cite{Kvp,Slt} of the 2D Euler equations. The computational domain is $[0,1.1]^2$, and the initial condition is
	\begin{equation*}
		(\rho_0,\mu_0,\nu_0,E_0)=\begin{cases}
			(1,0,0,\frac{0.244816}{h_xh_y}),&(x,y)\in{[0,h_x]\times[0,h_y]},
			\\(1,0,0,10^{-12}),&\mbox{otherwise}.
		\end{cases}
	\end{equation*}	
	The left and bottom boundaries are both reflective, while outflow conditions are applied on the right and upper boundaries. The final time is $T=1$. Similar to Example \ref{Sec3:Example_1dSedov}, this extreme problem involves extremely low pressure and strong shock, so that the PP limiter is necessary for successful simulation. The computational results are presented in Figure \ref{sec3:Fig_Sedov2D}, which are comparable to the results in \cite{ZS2}. The OE-HWENO method demonstrates good robustness in this demanding test, and the computed solutions are essentially non-oscillatory.
	\begin{figure}[!htb]
		\centering
		\begin{subfigure}{0.32\textwidth}
			{\includegraphics[width=5.0cm,height=5.0cm,angle=0]{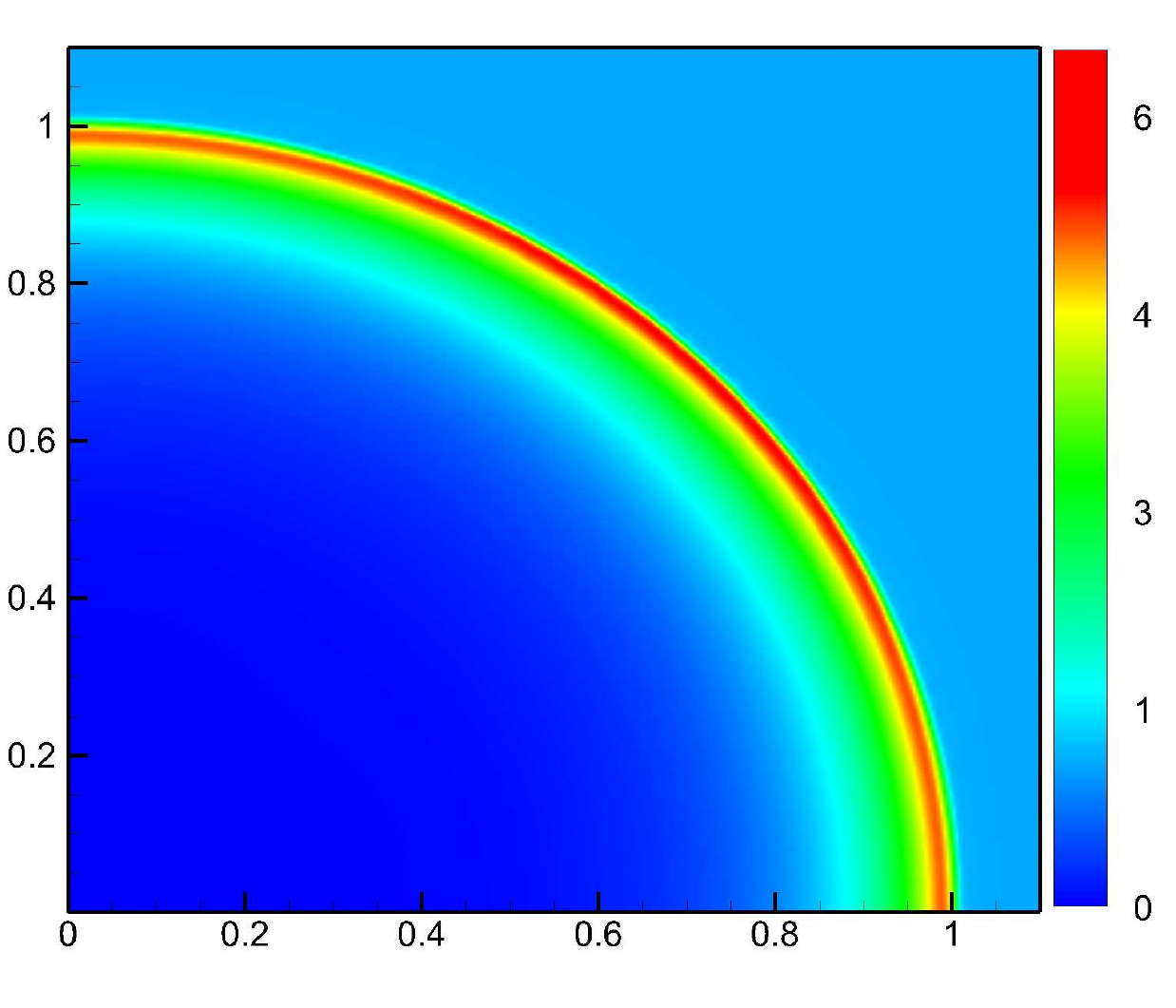}} 
			\caption{ Contour plots of density}
		\end{subfigure} 
		\begin{subfigure}{0.32\textwidth}
			{\includegraphics[width=5.0cm,height=5.0cm,angle=0]{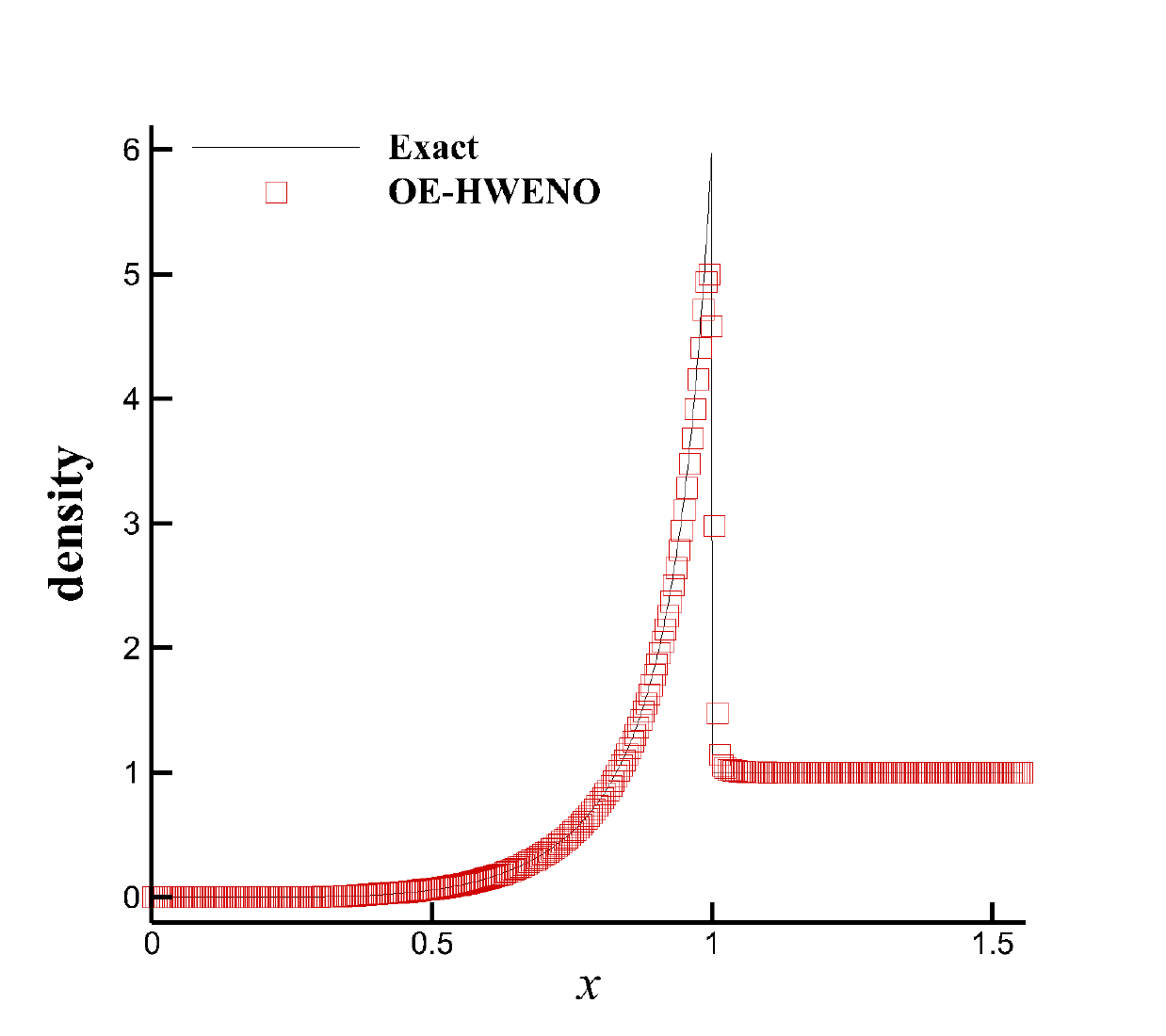}} 
			\caption{Density at $x=y$}
		\end{subfigure} 
		\begin{subfigure}{0.32\textwidth}
			{\includegraphics[width=5.25cm,height=5.25cm,angle=0]{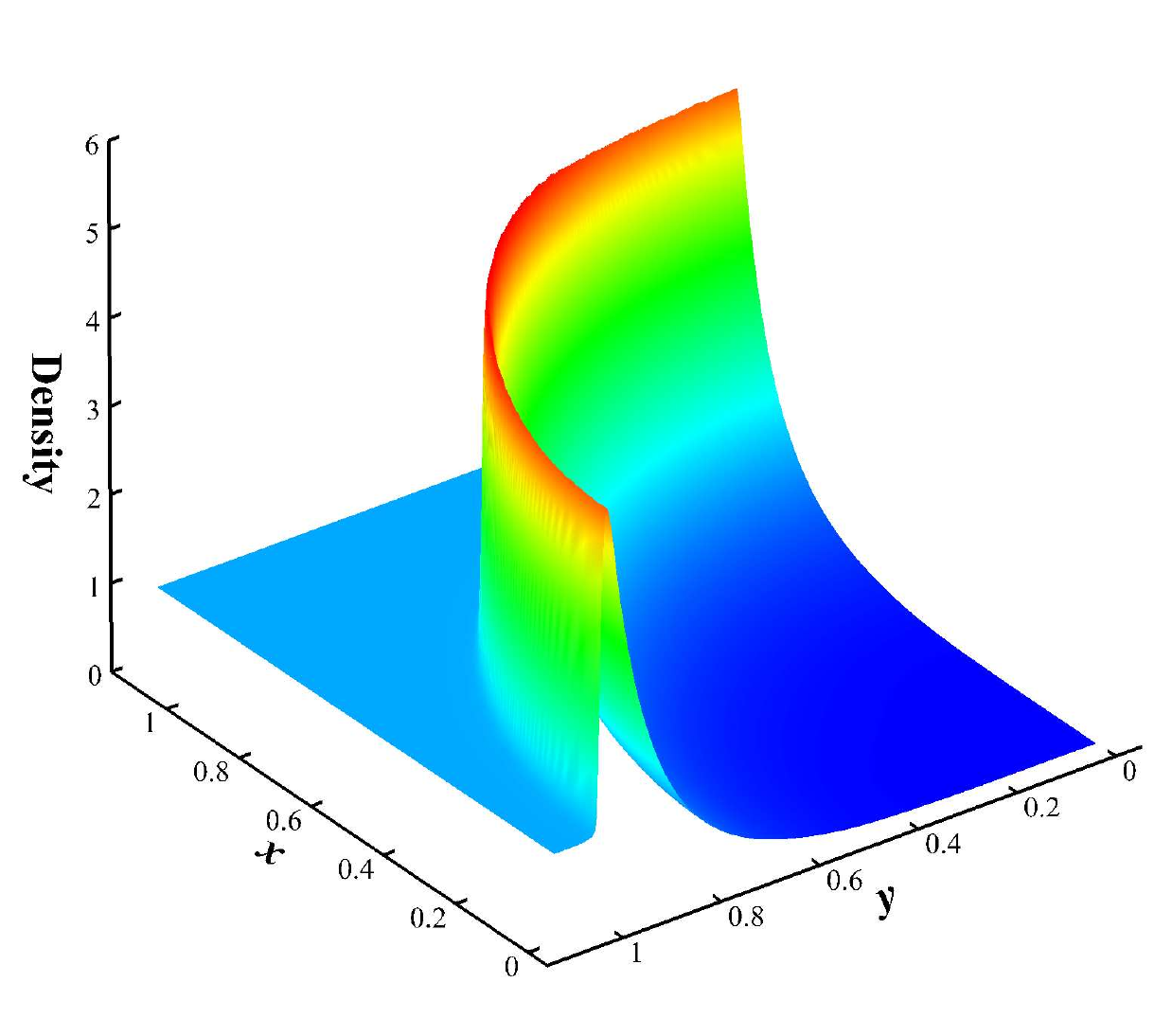}} 
			\caption{Density surface}
		\end{subfigure} 
		\caption{Numerical results of the 2D Sedov problem computed by OE-HWENO method with $320\times320$ uniform cells.  
		}
		\label{sec3:Fig_Sedov2D}
	\end{figure}
\end{example}

\begin{example}[Mach 2000 jet problem]\label{Sec3:Example_2dMach2000}  
	Finally, we further test the robustness of the OE-HWENO method by simulating a challenging Mach 2000 astrophysical jet problem studied in \cite{GD,HG,HGGS}. The computational domain is $[0,1]\times[-0.25,0.25]$. Initially, it is filled with an ambient gas with $(\rho,\mu,\nu,p)=(0.5,0,0,0.4127)$ and $\gamma=\frac{5}{3}$. The left boundary within the range $|y|<-0.05$ is inflow with a high-speed jet state $(\rho,\mu,\nu,p)=(5,800,0,0.4127)$, and all remaining parts are outflow. In Fig.~\ref{sec3:Fig_HM2000}, we present the results computed  by the OE-HWENO and St-HWENO methods with the PP limiter at  $T=0.001$. Notably, without the OE procedure, the St-HWENO method with the PP limiter exhibits obvious non-physical oscillations. In contrast, the OE-HWENO method clearly captures the intricate structures of the jet flow, including the bow shock and shear layer, which are resolved with high resolution and agree with those reported in \cite{LLS2,PSW,ZS2}.
	\begin{figure}[!thb]
		\centering 
		\begin{subfigure}{0.48\textwidth}
			{\includegraphics[width=7.5cm,height=3.75cm,angle=0]{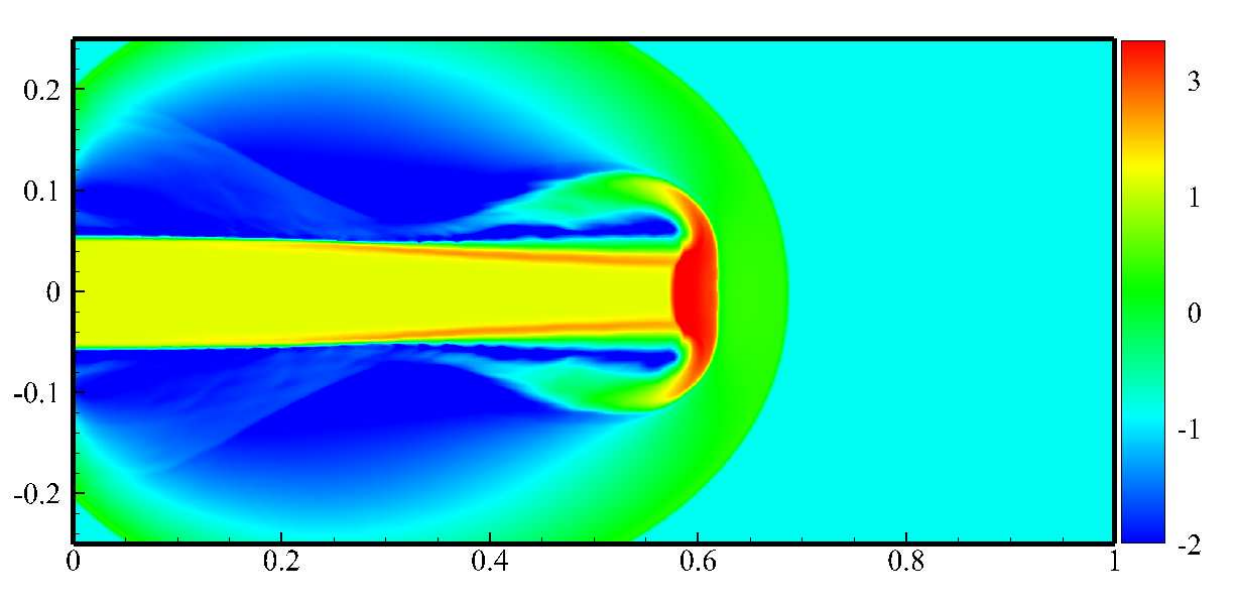}}
		\end{subfigure} 
		\begin{subfigure}{0.48\textwidth}
			{\includegraphics[width=7.5cm,height=3.75cm,angle=0]{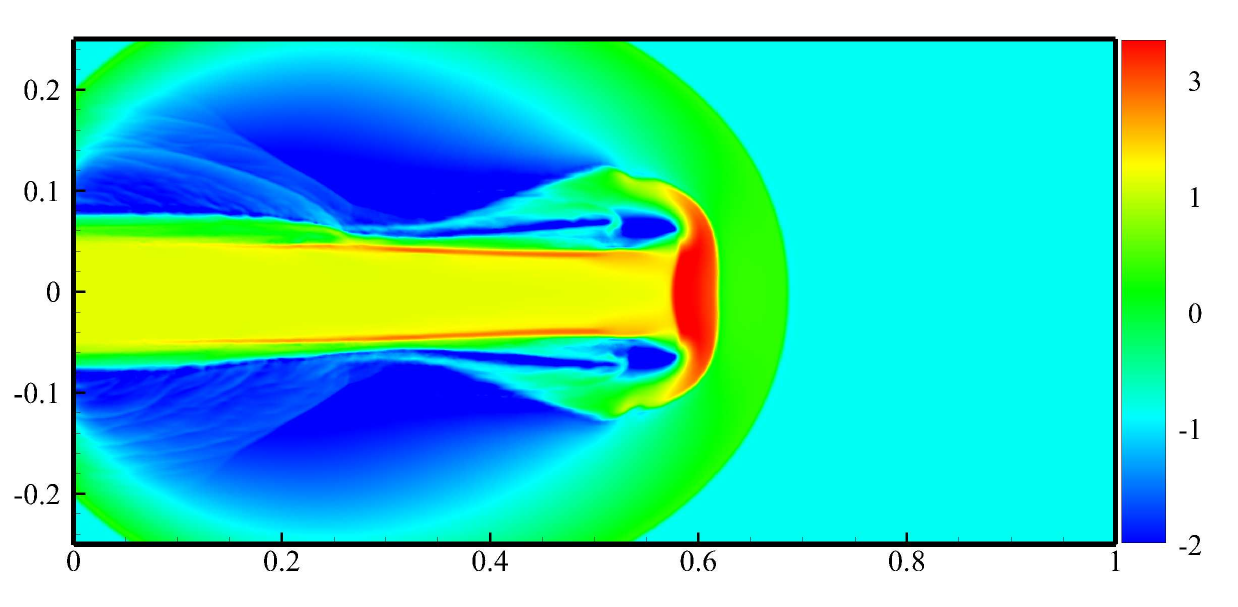}}
		\end{subfigure}
		\begin{subfigure}{0.48\textwidth}
			{\includegraphics[width=7.5cm,height=3.75cm,angle=0]{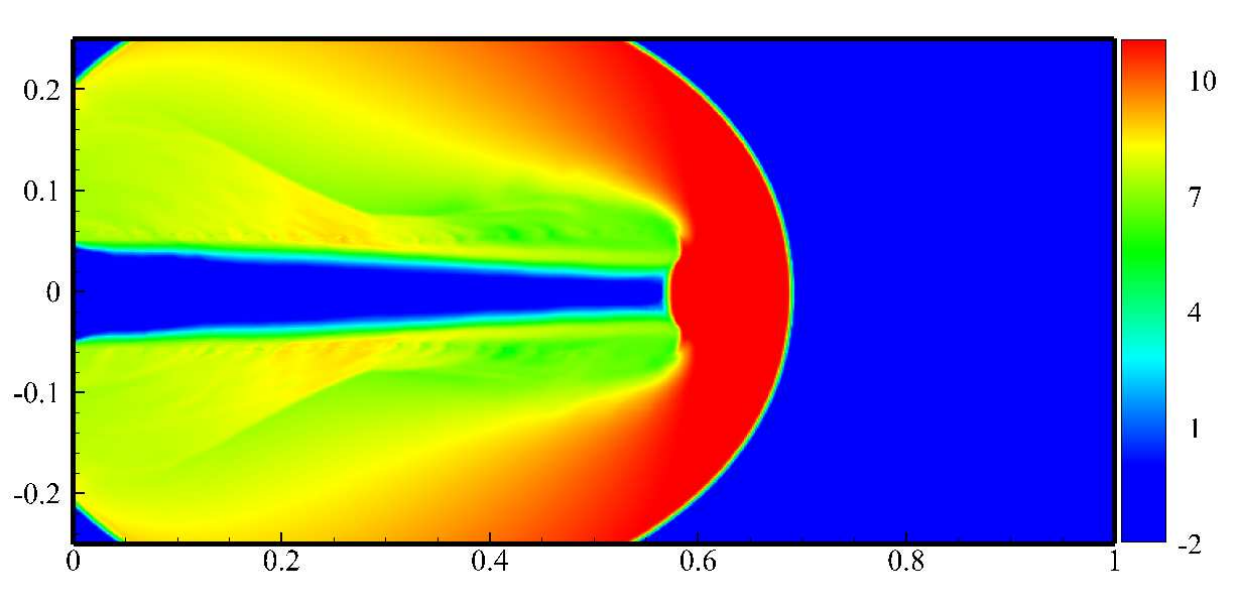}}
		\end{subfigure}
		\begin{subfigure}{0.48\textwidth}
			{\includegraphics[width=7.5cm,height=3.75cm,angle=0]{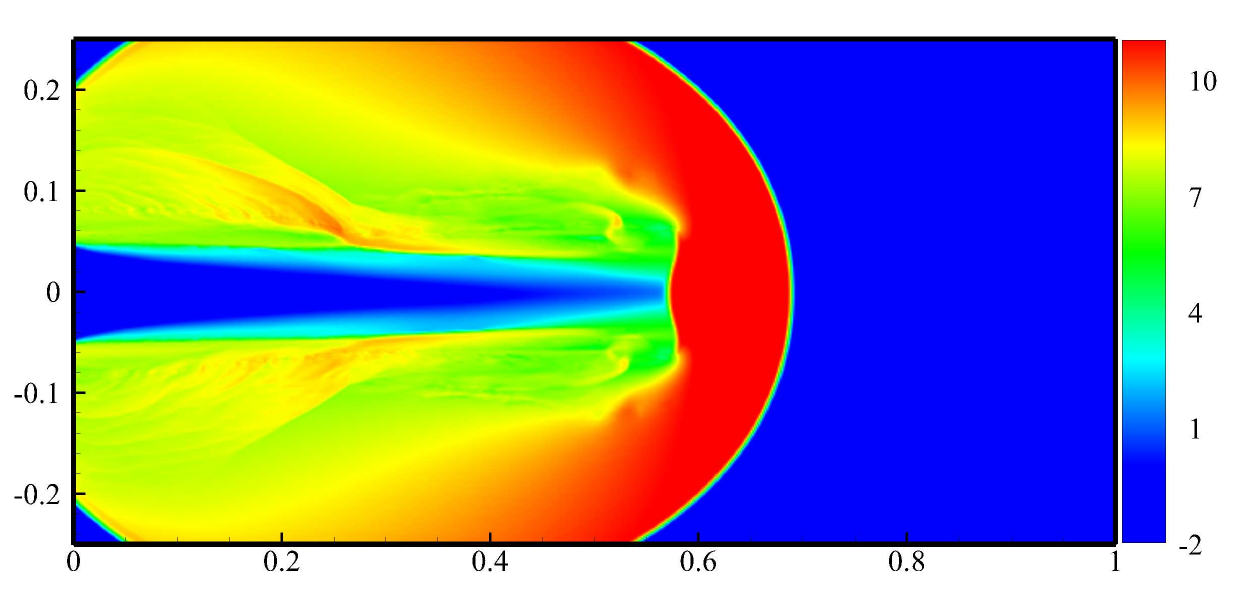}}
		\end{subfigure}
		\begin{subfigure}{0.48\textwidth}
			{\includegraphics[width=7.5cm,height=3.75cm,angle=0]{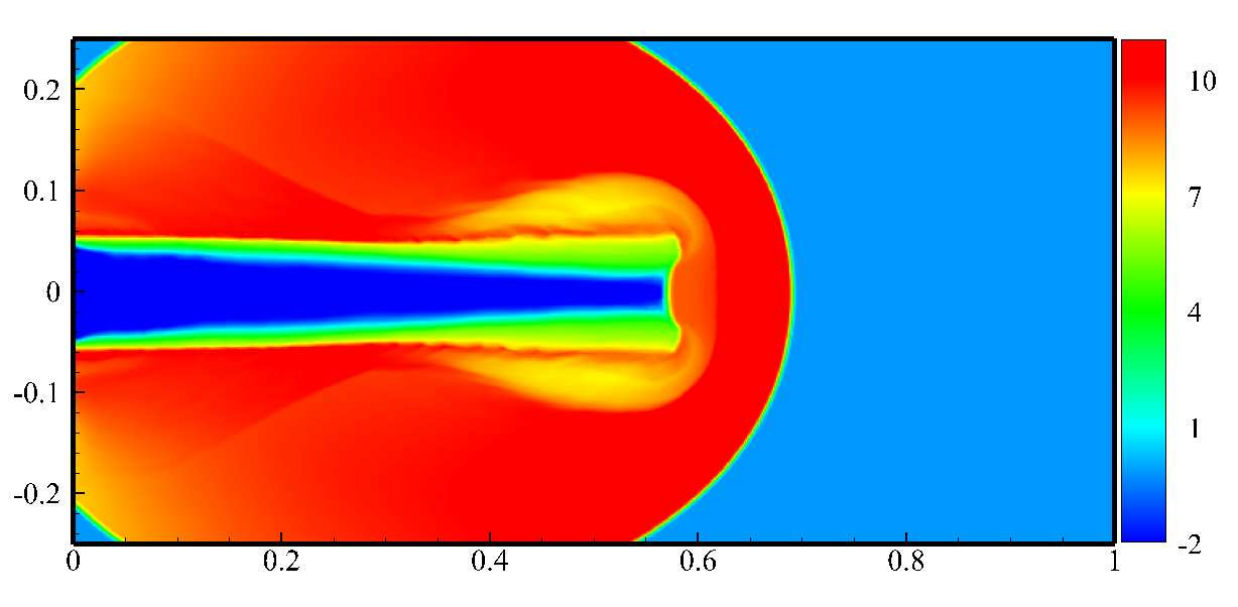}}
			\caption{OE-HWENO}
		\end{subfigure}
		\begin{subfigure}{0.48\textwidth}
			{\includegraphics[width=7.5cm,height=3.75cm,angle=0]{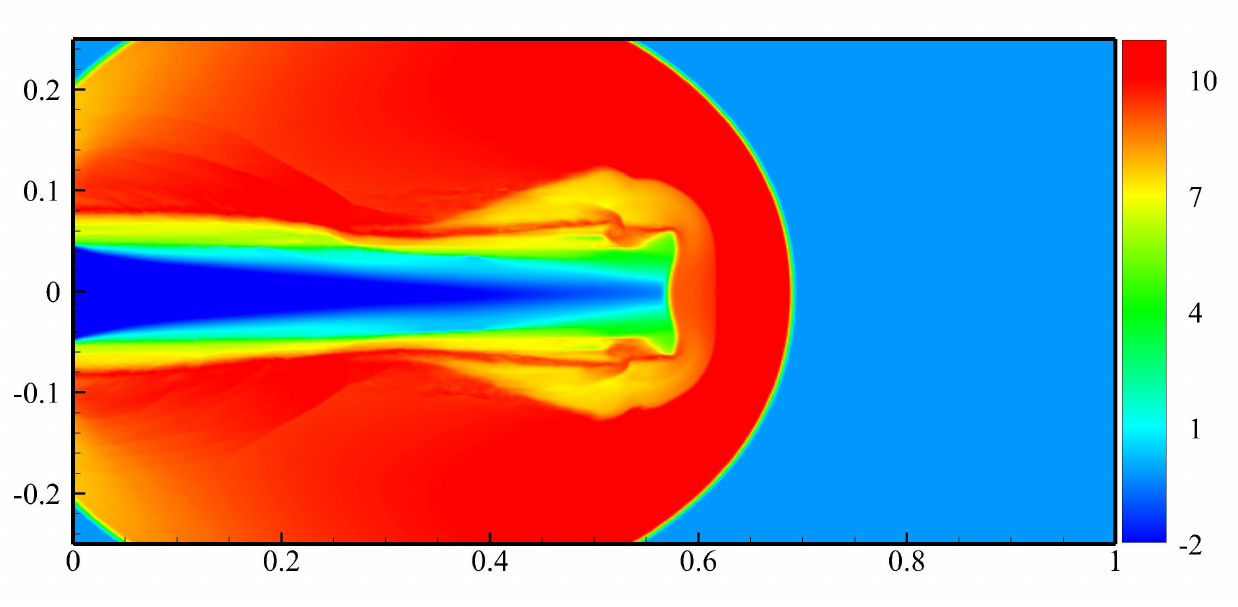}}
			\caption{St-HWENO}
		\end{subfigure}
		\caption{Numerical results of the Mach 2000 jet problem computed by the OE-HWENO and St-HWENO methods with $640\times320$ uniform cells. From top to bottom: the logarithms of density, pressure, and temperature. 
		}
		\label{sec3:Fig_HM2000}
	\end{figure}
\end{example}

\section{Conclusions}\label{sec4}

This paper has designed a high-order, oscillation-eliminating Hermite weighted essentially non-oscillatory (OE-HWENO) finite volume method for solving hyperbolic conservation laws. The OE-HWENO method stands out for its unique and efficient oscillation-eliminating (OE) procedure, which suppresses spurious oscillations in the numerical solution by damping first-order moments through a novel damping equation. This non-intrusive OE procedure is proven to maintain high-order accuracy and possesses scale- and evolution-invariant properties. 
The damping equation is exactly solvable without requiring discretization, ensuring the stability of the OE-HWENO scheme under standard CFL conditions, even in the presence of strong shocks. The efficiency of the OE procedure is significantly facilitated by this exact solver for the damping equation, requiring only a simple multiplication of the first-order moments by a damping factor. Additionally, a generic dimensionless procedure guarantees scale invariance in spatial reconstruction.  
The scale- and evolution-invariant properties of the OE and HWENO operators ensure that the numerical solutions obtained by the OE-HWENO method are oscillation-free across various problems with varying scales and wave speeds, as demonstrated in Examples \ref{sec3:Example_LWR}, \ref{sec3:Example_Lax}, and \ref{Sec3:Example_2dRiemann}, without relying on any problem-dependent parameters. Through the optimal cell average decomposition approach, the bound-preserving property of the OE-HWENO method has been rigorously proven under a milder time step constraint, provided that the HWENO reconstructed values satisfy \eqref{sec2:BP_2d_1} and \eqref{sec2:BP_2d_2}.  
	 This further enhances computational efficiency through relaxed bound-preserving time step constraints and fewer decomposition points.
The OE-HWENO method inherits the desirable features of traditional HWENO schemes, including compact stencils, high-order accuracy, high resolution, and the use of arbitrary linear weights. It also aligns well with the spectral properties of the original HWENO schemes. Extensive benchmark tests validate the robustness, accuracy, and efficiency of the OE-HWENO method across a wide range of problems.

In summary, the OE-HWENO schemes provide a highly effective and efficient numerical approach for solving hyperbolic conservation laws, overcoming challenges related to spurious oscillations and bound preservation while maintaining high-order accuracy. The non-intrusive nature of OE procedure allows seamless integration into existing HWENO codes, making it a practical tool for addressing a variety of challenging problems in computational fluid dynamics.

\section*{Acknowledgment}
This work was partially supported by Shenzhen Science and Technology Program (Grant No. RCJC20221008092757098) and  National Natural Science Foundation of China (Grant No.~12171227).

\begin{appendix}
 
\section{1D HWENO reconstruction}\label{sec:A_1dHWENO}	
 
The 1D HWENO reconstruction for a target cell $I_i$ consists of the following three steps: 

\textbf{Step I.} Construct a quintic polynomial $p_0(x)$, a cubic polynomial $p_1(x)$ and two linear polynomials $\{p_m(x)\}^3_{m=2}$ satisfying
\begin{align} \label{eq:p0}
		&\frac{1}{h_x}\int_{I_\ell}p_0(x)\mathrm{d}x=\bar{{u}}_{\ell},\quad \frac{1}{h_x}\int_{I_\ell}p_0(x)\frac{x-x_{\ell}}{h_x}\mathrm{d}x=\bar{{v}}_{\ell},~\ell=i-1,i,i+1,
		\\ \nonumber
		&\frac{1}{h_x}\int_{I_\ell}p_1(x)\mathrm{d}x=\bar{{u}}_{\ell},~\ell=i-1,i,i+1,\quad \frac{1}{h_x}\int_{I_\ell}p_1(x)\frac{x-x_\ell}{h_x}\mathrm{d}x=\bar{{v}}_{\ell},\ell=i,
		\\ \nonumber
		&\frac{1}{h_x}\int_{I_\ell}p_2(x)\mathrm{d}x=\bar{{u}}_{\ell},~\ell=i-1,i,\quad \frac{1}{h_x}\int_{I_\ell}p_3(x)\mathrm{d}x=\bar{{u}}_{\ell},~\ell=i,i+1.
	\end{align}
Let $p_m(x)=\sum\limits_{n=0}^{r_m} c_{m,n} \phi_n(x)$, where $r_m$ is the degree of $p_m(x)$. The coefficients of the resulting polynomials $\{p_m(x)\}^3_{m=0}$ are listed in Table  \ref{secA:coe1D}.
We can reformulate $p_0(x)$ and $p_1(x)$ as follows: 
\begin{equation*} 
	\left\{ 
	\begin{aligned}
		&{p_0(x)=\gamma^{\text{H}}_0\left(\frac{1}{\gamma^\text{H}_0}p_0(x) - \frac{\gamma^\text{H}_1}{\gamma^\text{H}_0}p_1(x) \right) + \gamma^\text{H}_1p_1(x),
		}
		\\&
		{p_1(x)=\gamma^\text{L}_1\left(\frac{1}{\gamma^\text{L}_1}p_1(x) - \frac{\gamma^\text{L}_2}{\gamma^\text{L}_1}p_2(x) - \frac{\gamma^\text{L}_3}{\gamma^\text{L}_1}p_3(x)\right) + \gamma^\text{L}_2p_2(x)  + \gamma^\text{L}_3p_3(x), 
		}
	\end{aligned}
	\right. 
\end{equation*}
where $\{\gamma^\text{H}_m\}^1_{m=0}$ and $\{\gamma^\text{L}_m\}^3_{m=1}$ are positive with $\sum^{1}_{m=0}\gamma^\text{H}_m=1$ and $\sum^{3}_{m=1}\gamma^\text{L}_m=1$.

\textbf{Step II.} Compute the smoothness indicators $\{\beta_m\}^3_{m=0}$ of $\{p_m(x)\}^3_{m=0}$ in $I_i$ by
\begin{equation*} 
	\beta_m=\sum\limits_{l=1}^{r_m} \int_{I_i}h_x^{2l-1} \left(\frac{{\rm d}^l p_m(x)}{{\rm d} x^l} \right)^2 \mathrm{d}x,\quad m=0,\dots,3,
\end{equation*}
where $r_m$ is the degree of $p_m(x)$. The explicit expressions of the smoothness indicators can be calculated as
\begin{equation*} 
	\left \{
	\begin{aligned}
		\beta_0=&\frac12(c_{0,1} + \frac15c_{0,3})^2 + \frac12(c_{0,1} + \frac{1}{63}c_{0,5})^2 + \frac{13}{3}(c_{0,2} + \frac{123}{455}c_{0,4})^2 + \frac{976}{25}(c_{0,3} + \frac{7235}{13664}c_{0,5})^2 \\&+ \frac{1421461}{2275}c_{0,4}^2 + \frac{242038614799}{15494976}c_{0,5}^2,
		\\
		\beta_1=&(c_{1,1} + \frac{1}{10}c_{1,3})^2 + \frac{781}{20}c_{1,3}^2 + \frac{13}{3}c_{1,2}^2, \\
		\beta_m=&c_{m,1}^2,m=2,3,\\
	\end{aligned}
	\right.
\end{equation*} 
where the coefficients $\{c_{m,n}\}$
are listed in Table \ref{secA:coe1D}.

\textbf{Step III.} Compute the nonlinear weights by
\begin{equation}\label{sec:1d_non_weights}
	\left\{ 
	\begin{aligned}
		&\omega^\text{H}_\ell= \frac{\bar{\omega}^\text{H}_\ell}{\bar{\omega}^\text{H}_0+\bar{\omega}^\text{H}_1}, \quad \mbox{with}~~ \bar{\omega}^\text{H}_\ell={\gamma}^\text{H}_\ell(1+\frac{\tau_0}{\beta_\ell+\varepsilon}),~ \ell=0,1,
		\\&                                                                   
		\omega^\text{L}_\ell= \frac{\bar{\omega}^\text{L}_\ell}{\bar{\omega}^\text{L}_1+\bar{\omega}^\text{L}_2+\bar{\omega}^\text{L}_3}, \quad \mbox{with}~~ \bar{\omega}^\text{L}_\ell={\gamma}^\text{L}_\ell(1+\frac{\tau_1}{\beta_\ell+\varepsilon}),~ \ell=1,2,3.
	\end{aligned}
	\right. 
\end{equation}
where $\tau_0=(\beta_0-\beta_1)^2$ measures the difference between $\beta_0$ and $\beta_1$; $\tau_1=(\frac{|\beta_1-\beta_2|+|\beta_1-\beta_3|}{2})^2$ measures the difference between  $\beta_1$, $\beta_2$ and $\beta_3$. Finally, we obtain the piecewise polynomial solution $u_h(x,t)$ in $I_i$ by a nonlinear HWENO reconstruction:
\begin{equation*} 
	\begin{aligned}
		&{u_h(x,t)=\omega^\text{H}_0\bigg(\frac{1}{\gamma^\text{H}_0}p_0(x) - \frac{\gamma^\text{H}_1}{\gamma^\text{H}_0}\tilde{q}_1(x) \bigg) + \omega^\text{H}_1\tilde{q}_1(x),} 
	\end{aligned} 
\end{equation*}
where $\tilde{q}_1(x)=\omega^\text{L}_1\big(\frac{1}{\gamma^\text{L}_1}p_1(x) - \frac{\gamma^\text{L}_2}{\gamma^\text{L}_1}p_2(x) - \frac{\gamma^\text{L}_3}{\gamma^\text{L}_1}p_3(x)\big) + \omega^\text{L}_2p_2(x)  + \omega^\text{L}_3p_3(x)$. 

\begin{table}[!htb]
	\centering
	\fontsize{10.}{13}\selectfont
	\begin{threeparttable}
		\caption{The coefficients in \eqref{sec:1d_non_weights}. } 
		\begin{tabular}{ccccccccccccc}
			\toprule
			{Coefficients} & $m=0$&$m=1$&$m=2$&$m=3$ \cr
			\midrule 
			$c_{m,0}$	&$\bar{u}_{{i}}$ &$\bar{u}_{{i}}$  &$\bar{u}_{{i}}$  &$\bar{u}_{{i}}$   \\
			$c_{m,1}$	&$12\bar{v}_{{i}}$  &$12\bar{v}_{{i}}$  &$\bar{u}_{i}-\bar{u}_{i-1}$  &$\bar{u}_{i+1}-\bar{u}_{i}$   \\
			$c_{m,2}$	&${\frac {73}{56}}( \bar{u}_{{i-1}}-2\bar{u}_{{i}}+\bar{u}_{{i+1}}){+\frac{135}{28}(\bar{{v}}_{i-1}-\bar{{v}}_{i+1})}$ &$\frac12( \bar{u}_{{i-1}}-2\bar{u}_{{i}}+\bar{u}_{{i+1}})$  &  &   \\
			$c_{m,3}$	&${\frac {595}{324}}(\bar{u}_{{i+1}}-\bar{u}_{{i-1}})-{\frac {5}{81}}(517\bar{v}_{{i}}+197(\bar{v}_{i-1}+{\bar{v}_{i+1}}))$  &${\frac{5}{11}} (\bar{u}_{{i+1}}-\bar{u}_{{i-1}})-{\frac {120}{11}}\bar{v}_{{i}}$  &  &   \\
			$c_{m,4}$	&$-\frac{5}{8}(\bar{u}_{i-1}-2\bar{u}_{i}+\bar{u}_{i+1}){-\frac{15}{4}(\bar{{v}}_{i-1}-\bar{{v}}_{i+1})}$  &  &  &   \\
			$c_{m,5}$	&${\frac {35}{36}}(\bar{u}_{{i-1}}-\bar{u}_{{i+1}})+{\frac {1}{9}}(133v_{{i}}+77(v_{i-1}{+v_{i+1}}))$  &  &  &   \\			 
			\bottomrule
		\end{tabular}\label{secA:coe1D}	
	\end{threeparttable}
\end{table}

\begin{figure}[!htb]
	\centering
	\tikzset{global scale/.style={scale=#1,every node/.append style={scale=#1}}}
	\centering
	\begin{tikzpicture}[global scale = 1]
		\draw(0,0)rectangle+(1.8,1.8);\draw(1.8,0)rectangle+(1.8,1.8);\draw(3.6,0)rectangle+(1.8,1.8);
		\draw(0,1.8)rectangle+(1.8,1.8);\draw(1.8,1.8)rectangle+(1.8,1.8);\draw(3.6,1.8)rectangle+(1.8,1.8);
		\draw(0,3.6)rectangle+(1.8,1.8);\draw(1.8,3.6)rectangle+(1.8,1.8);\draw(3.6,3.6)rectangle+(1.8,1.8);
		\draw(0.9,0.9)node{1};
		\draw(2.7,0.9)node{2};
		\draw(4.5,0.9)node{3};
		\draw(0.9,2.7)node{4};
		\draw(2.7,2.7)node{5};
		\draw(4.5,2.7)node{6};
		\draw(0.9,4.5)node{7};
		\draw(2.7,4.5)node{8};
		\draw(4.5,4.5)node{9};
		\draw(0.9,-0.25)node{$~~ i-1$};
		\draw(2.7,-0.25)node{$~i$};
		\draw(4.5,-0.25)node{$~~i+1$};
		\draw(5.8,0.9)node{$~~~j-1$};
		\draw(5.8,2.7)node{$~~~j$};
		\draw(5.8,4.5)node{$~~~j+1$};
	\end{tikzpicture}
	\caption{Stencil for the HWENO reconstruction with the respective label.}
	\label{sec2:stenciL2d}
\end{figure}
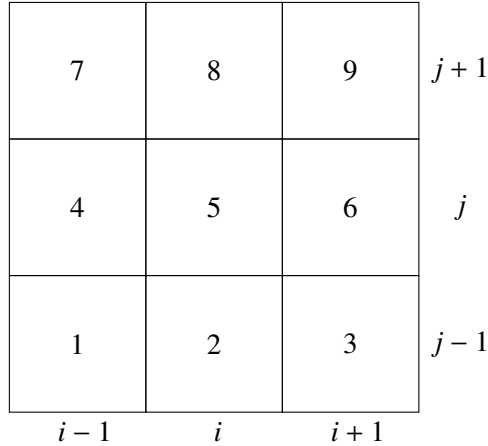

\section{2D HWENO reconstruction}\label{sec:A_2dHWENO}

For convenience, under the serial numbers in Fig.~\ref{sec2:stenciL2d}, we relabel the cell $I_{i,j}$ and its adjacent cells as $I_1,...,I_9$, e.g., $I_{i,j}\triangleq I_5$. Let $\{\bar{{u}}_k,\bar{{v}}_k,\bar{{w}}_k\}$ denote the zeroth- and first-order moments of the cell $I_k$, e.g., $\{\bar{{u}}_{i,j}\triangleq \bar{{u}}_5, \bar{{v}}_{i,j}\triangleq \bar{{v}}_5,  \bar{{w}}_{i,j}\triangleq \bar{{w}}_5 \}$.   
The 2D HWENO reconstruction for the target cell $I_{i,j}$ consists of the following steps:

\textbf{Step I.} Construct a quintic polynomial $p_0(x,y)$, a cubic polynomial $p_1(x,y)$ and four linear polynomials $\{p_m(x,y)\}^5_{m=2}$ satisfying 
\begin{equation}\label{sec:p0(x,y)}
		\left\{
	\begin{aligned}
		&\frac{1}{h_xh_y}\int_{I_k}p_0(x,y) \mathrm{d}x\mathrm{d}y=\bar{{u}}_k,~k=1,...,9,\\
		&\frac{1}{h_xh_y}\int_{I_{k}}p_0(x,y)\frac{x-x_{k}}{h_x}\mathrm{d}x\mathrm{d}y=\bar{{v}}_{k},~k={1,3,4,5,6,7,9},\\
		&\frac{1}{h_xh_y}\int_{I_{k}}p_0(x,y)\frac{y-y_{k}}{h_y}\mathrm{d}x\mathrm{d}y=\bar{{w}}_{k}, ~k={1,2,3,5,7,8,9},
	\end{aligned}
		\right.
\end{equation}
\begin{equation*} 
		\left\{
	\begin{aligned}
		&\frac{1}{h_xh_y}\int_{I_k}p_1(x,y) \mathrm{d}x\mathrm{d}y=\bar{{u}}_k,~k=1,...,9,\\
		&\frac{1}{h_xh_y}\int_{I_{k}}p_1(x,y)\frac{x-x_{k}}{h_x}\mathrm{d}x\mathrm{d}y={\bar{{v}}_{k},k=5,}\\
		&\frac{1}{h_xh_y}\int_{I_{k}}p_1(x,y)\frac{y-y_{k}}{h_y}\mathrm{d}x\mathrm{d}y={\bar{{w}}_{k},k=5,}
	\end{aligned}
		\right.
\end{equation*}
\begin{equation*} 
	\left\{  
	\begin{aligned}
		&\frac{1}{h_xh_y}\int_{I_k}p_2(x,y) \mathrm{d}x\mathrm{d}y=\bar{{u}}_k,~k=2,4,5, 
		\frac{1}{h_xh_y}\int_{I_k}p_3(x,y) \mathrm{d}x\mathrm{d}y=\bar{{u}}_k,~k=2,5,6,
		\\&
		\frac{1}{h_xh_y}\int_{I_k}p_4(x,y) \mathrm{d}x\mathrm{d}y=\bar{{u}}_k,~k=4,5,8, 
		\frac{1}{h_xh_y}\int_{I_k}p_5(x,y) \mathrm{d}x\mathrm{d}y=\bar{{u}}_k,~k=5,6,8. 
	\end{aligned} 
\right. 
\end{equation*}
The quintic polynomial $p_0(x,y)$ and cubic polynomial $p_1(x,y)$ can be uniquely determined by requiring them to exactly match $\bar{{u}}_5$ using the least squares method \cite{HS,ZQ1}, while the four polynomials $\{p_m(x,y)\}^5_{m=2}$ can be directly obtained by solving $3\times3$ linear systems.
We can reformulate $p_0(x,y)$ and $p_1(x,y)$ as follows:
\begin{equation*}
	\left\{ 
	\begin{aligned}
		&{p_0(x,y)=\gamma^\text{H}_0\left(\frac{1}{\gamma^\text{H}_0}p_0(x,y) - \frac{\gamma^\text{H}_1}{\gamma^\text{H}_0}p_1(x,y) \right) + \gamma^\text{H}_1p_1(x,y), 
		}
		\\&
		{p_1(x,y)=\gamma^\text{L}_1\left(\frac{1}{\gamma^\text{L}_1}p_1(x,y) - \sum\limits_{m=2}^5 \frac{\gamma^\text{L}_m}{\gamma^\text{L}_1}p_m(x,y)\right) + \sum\limits_{m=2}^5\gamma^\text{L}_m p_m(x,y),
		}
	\end{aligned}
	\right. 
\end{equation*}
where $\{\gamma^\text{H}_m\}^1_{m=0}$ and $\{\gamma^\text{L}_m\}^5_{m=2}$ are arbitrary positive linear weights satisfying $\sum_{m=0}^{1}\gamma^\text{H}_m=1$ and $\sum_{m=1}^{5}\gamma^\text{L}_m=1$. 

\textbf{Step II.} Compute the smoothness indicators $\{\beta_m\}^5_{m=0}$ of polynomials $\{p_m(x,y)\}^5_{m=0}$ in $I_{i,j}$ by
\begin{equation*} 
	\beta_m= \sum_{| {\bm \ell} |=1}^{r_{m}} |I_{i,j}|^{| {\bm \ell} |-1} \int_{I_{i,j}}\left( \frac {\partial^{|{\bm \ell}|}}{\partial x^{\ell_1}\partial y^{\ell_2}}p_m(x,y)\right)^2 \mathrm{d}x\mathrm{d}y, \quad m=0,...,4,
\end{equation*}
where ${\bm \ell}=(\ell_1,\ell_2)$, $|{\bm \ell}|=\ell_1+\ell_2$, and $r_m$ is the degree of polynomial $p_m(x, y)$.
The explicit expressions of the smoothness indicators are
\begin{align*} 
	\begin{aligned} 
		\beta_0&=
		\frac12(c_{0,1}+\frac15c_{0,6} ) ^{2}+
		\frac12(c_{0,1}+{\frac{1}{63}c_{0,15}} )^{2}+ 
		\frac12(c_{0,2}+\frac15c_{0,9} ) ^{2}+ 
		\frac12(c_{0,2}+\frac{1}{63}c_{0,20})^{2}\\&+ 
		\frac{13}{3}(c_{0,3}+\frac{123}{455}c_{0,10})^{2}+
		\frac{7}{12}(c_{0,4}+{\frac{13}{70}c_{0,13}})^{2}+
		\frac{7}{12}(c_{0,4}+{\frac{13}{70}c_{0,11}})^{2}+
		\frac{13}{3}(c_{0,5}+{\frac{123}{455}c_{0,14}})^{2}\\&+ 
		\frac{976}{25}(c_{0,6}+{\frac{7235}{13664}c_{0,15}})^{2}+
		\frac{47}{20}(c_{0,7}+{\frac{781}{4230}c_{0,18}})^{2}+
		\frac{47}{20}(c_{0,7}+{\frac{533}{987}c_{0,16}})^{2}+
		\frac{47}{20}(c_{0,8}+{\frac{781}{4230}c_{0,17}})^{2}\\&+
		\frac{47}{20}(c_{0,8}+{\frac{533}{987}c_{0,19}})^{2}+
		\frac{976}{25}(c_{0,9}+{\frac{7235}{13664}c_{0,20}})^{2}+	
		\frac{1421461}{2275}c_{0,10}^{2}+	
		\frac{4441}{105}(c_{0,11}+{\frac {21}{88820}c_{0,13}}) ^{2}\\&+
		\frac{116856056}{172725}(c_{0,16}+{\frac{40467}{233712112}c_{0,18}} ) ^{2}+
		\frac{564287369}{3331125}(c_{0,17}+{\frac{780435}{1128574738}c_{0,19}})^{2}+
		\frac{5083}{270}c_{0,12}^{2}\\&+
		\frac{7888991959}{186522000}c_{0,13}^{2}+
		\frac{1421461}{2275}c_{0,14}^{2}+
		\frac{242038614799}{15494976}c_{0,15}^{2}+ 
		\frac{263761553985963511}{1557048518172000}c_{0,18}^{2} \\&	+
		\frac{263761553985963511}{389866143242100}c_{0,19}^{2}+ 
		\frac{242038614799}{15494976}c_{0,20}^{2},
	\end{aligned}	
		\\
\begin{aligned} 
		\beta_1&=(c_{1,1} + \frac{1}{10}c_{1,6})^2 +(c_{0,2} + \frac{1}{10}c_{1,9})^2 +\frac{13}{3}(c_{1,3}^2+c_{1,5}^2 )+\frac76c_{1,4}^2 +\frac{781}{20}(c_{1,6}^2+c_{1,9}^2) +\frac{47}{10}(c_{1,7}^2+c_{1,8}^2),\\
		\beta_m&= c^2_{m,1} + c^2_{m,2},m=2,3,4,5,\\ 
	\end{aligned}
\end{align*}
where the coefficients $\{c_{m,n}\}$ 
are listed in Table \ref{secA:coe2D}.

\textbf{Step III.} Compute the nonlinear weights by
\begin{equation}\label{sec:2d_non_weights}
	\left\{ 
	\begin{aligned}
		&\omega^\text{H}_\ell= \frac{\bar{\omega}^\text{H}_\ell}{\bar{\omega}^\text{H}_0+\bar{\omega}^\text{H}_1}, \quad \mbox{with}~~ \bar{\omega}^\text{H}_\ell={\gamma}^\text{H}_\ell(1+\frac{\tau_0}{\beta_\ell+\varepsilon}),~ \ell=0,1,
		\\&                                                                   
		\omega^\text{L}_\ell= \frac{\bar{\omega}^\text{L}_\ell}{\bar{\omega}^\text{L}_1+\cdots+\bar{\omega}^\text{L}_5}, \quad  \mbox{with}~~ \bar{\omega}^\text{L}_\ell={\gamma}^\text{L}_\ell(1+\frac{\tau_0}{\beta_\ell+\varepsilon}),~ \ell=1,\cdots,5.
	\end{aligned}
	\right. 
\end{equation}
where  $\tau_0=(\beta_0-\beta_1)^2$ measures the difference between $\beta_0$ and $\beta_1$; $\tau_1=(\frac{\sum_{\ell=2}^{5}|\beta_1-\beta_\ell|}{4})^2$ measures the difference $\beta_1$ to $\beta_\ell, \ell=2,\ldots,5$. 
Finally, we obtain the piecewise polynomial solution $u_h(x,y,t)$ in the target cell $I_{i,j}$ by the nonlinear HWENO reconstruction 
\begin{equation*} 
	\begin{aligned}
		&{u_h(x,y,t)=\omega^\text{H}_0\bigg(\frac{1}{\gamma^\text{H}_0}p_0(x,y) - \frac{\gamma^\text{H}_1}{\gamma^\text{H}_0}\tilde{q}_1(x,y) \bigg) + \omega^\text{H}_1\tilde{q}_1(x,y),} 
	\end{aligned} 
\end{equation*}
where $\tilde{q}_1(x,y)=\omega^\text{L}_1\big(\frac{1}{\gamma^\text{L}_1}p_1(x,y) - \sum\limits_{m=2}^5\frac{\gamma^\text{L}_m}{\gamma^\text{L}_1}p_m(x,y) \big) + \sum\limits_{m=2}^5\omega^\text{L}_mp_m(x,y)$.   

\begin{table}[H]
	\centering
	\fontsize{10.5}{13}\selectfont
	\begin{threeparttable}
		\caption{The coefficients in \eqref{sec:2d_non_weights}.}  
		\begin{tabular}{cccccccccccccccccc}
			\toprule
			{Coefficients} & $m=0$ \cr
			\midrule	 
			$c_{m,0}$	&$\bar{u}_{{5}}$     \\	
			$c_{m,1}$	&$12\bar{v}_{{5}}$    \\	
			$c_{m,2}$	&$12\bar{w}_{{5}}$  \\	
			$c_{m,3}$	&${\frac {883(\bar{u}_{{1}}-2\bar{u}_{{2}}+\bar{u}_{{3}}+\bar{u}_{{7}}-8\bar{u}_{{8}}+\bar{u}_{{9}})}{3304}}+{\frac {363(\bar{u}_{{4}}-2\bar{u}_{{5}}+\bar{u}_{{6}})}{472}}+{\frac {663(\bar{v}_{{1}}-\bar{v}_{{3}}+\bar{v}_{{7}}-\bar{v}_{{9}})}{413}}+{\frac {2661(\bar{v}_{{4}}-\bar{v}_{{6}})}{1652}}-{\frac {3(\bar{w}_{{1}}-\bar{w}_{{2}}+\bar{w}_{{3}}+\bar{w}_{{7}}+2\bar{w}_{{8}}-\bar{w}_{{9}})}{1652}}
			$    \\	
			$c_{m,4}$	&$\frac{41(\bar{u}_{{1}}-\bar{u}_{{3}}-\bar{u}_{{7}}+\bar{u}_{{9}})}{76} + \frac{33(\bar{v}_{{1}}+\bar{v}_{{3}}-\bar{v}_{{7}}-\bar{v}_{{9}})}{19}+ \frac{33(\bar{w}_{{1}}-\bar{w}_{{3}}+\bar{w}_{{7}}-\bar{w}_{{9}})}{19}$   \\	
			$c_{m,5}$	&$\frac{883(\bar{u}_1+\bar{u}_3-2\bar{u}_4-2\bar{u}_6+\bar{u}_7+\bar{u}_9)}{3304}+\frac{363(\bar{u}_2-2\bar{u}_5+\bar{u}_8)}{472} -\frac{3(\bar{v}_1-\bar{v}_3-2\bar{v}_4+2\bar{v}_6+\bar{v}_7-\bar{v}_9)}{1652} +\frac{663(\bar{w}_1+\bar{w}_3-\bar{w}_7-\bar{w}_9)}{413}+\frac{2661(\bar{w}_2-\bar{w}_8)}{1652}$   \\	
			$c_{m,6}$	&$-\frac{595(\bar{u}_4-\bar{u}_6)}{324}-\frac{5(197(\bar{v}_4+\bar{v}_6)+1034\bar{v}_5)}{162}$    \\	
			$c_{m,7}$	&$-\frac{1695(\bar{u}_1-2\bar{u}_2+\bar{u}_3-\bar{u}_7+2\bar{u}_8-\bar{u}_9)}{2128} -\frac{135(\bar{v}_1-\bar{v}_3-\bar{v}_7+\bar{v}_9)}{56} -\frac{33(\bar{w}_1-2\bar{w}_2+\bar{w}_3+\bar{w}_7-2\bar{w}_8+\bar{w}_9)}{19}$    \\	
			$c_{m,8}$	&$-\frac{1695(\bar{u}_1-\bar{u}_3-2\bar{u}_4+2\bar{u}_6+\bar{u}_7-\bar{u}_9)}{2128} -\frac{33(\bar{v}_1+\bar{v}_3-2\bar{v}_4-2\bar{v}_6+\bar{v}_7+\bar{v}_9)}{19} -\frac{135(\bar{w}_1-\bar{w}_3-\bar{w}_7+\bar{w}_9)}{56}$    \\	
			$c_{m,9}$	&$-\frac{595(\bar{u}_2-\bar{u}_8)}{324}-\frac{5(197(\bar{w}_2+\bar{w}_8)+1034\bar{w}_5)}{162}$    \\	
			$c_{m,10}$	&$-\frac{95(\bar{u}_1-2\bar{u}_2+\bar{u}_3+\bar{u}_7-2\bar{u}_8+\bar{u}_9) -105(\bar{u}_4-2\bar{u}_5+\bar{u}_6)}{472} -\frac{290(\bar{v}_1-\bar{v}_3-\bar{v}_7+\bar{v}_9)}{236} +\frac{5(\bar{w}_1-2\bar{w}_2+\bar{w}_3-\bar{w}_7+2\bar{w}_8-\bar{w}_9)}{236}$   \\	
			$c_{m,11}$	&$-\frac{5(\bar{u}_1-\bar{u}_3-\bar{u}_7+\bar{u}_9)}{38} -\frac{30(\bar{v}_1+\bar{v}_3-\bar{v}_7+\bar{v}_9)}{19}$   \\	
			$c_{m,12}$	&$\frac{27(\bar{u}_1+\bar{u}_3+\bar{u}_7+\bar{u}_9)-54(\bar{u}_2+\bar{u}_4-2\bar{u}_5+\bar{u}_6+\bar{u}_8)}{118} -\frac{15(\bar{v}_1-\bar{v}_3-2\bar{v}_4+2\bar{v}_6+\bar{v}_7-\bar{v}_9)}{236} -\frac{15(\bar{w}_1-2\bar{w}_2+\bar{w}_3-\bar{w}_7+2\bar{w}_8-\bar{w}_9)}{236}$   \\	
			$c_{m,13}$	&$-\frac{5(\bar{u}_1-\bar{u}_3-\bar{u}_7+\bar{u}_9)}{38} -\frac{30(\bar{w}_1-\bar{w}_3+\bar{w}_7-\bar{w}_9)}{19}$   \\	
			$c_{m,14}$	&$-\frac{95(\bar{u}_1+\bar{u}_3-2\bar{u}_4-2\bar{u}_6+\bar{u}_7+\bar{u}_9)+105(\bar{u}_2-2\bar{u}_5+\bar{u}_8)}{472} +\frac{5(\bar{v}_1-\bar{v}_3-2\bar{v}_4+2\bar{v}_6+\bar{v}_7-\bar{v}_9)}{236} -\frac{290(\bar{w}_1+\bar{w}_3-\bar{w}_7-\bar{w}_9)+305(\bar{w}_2-\bar{w}_8)}{236}$   \\	
			$c_{m,15}$	&$\frac{35(\bar{u}_4-\bar{u}_6)}{36} +\frac{77(\bar{v}_4+\bar{v}_6)+266\bar{v}_5}{18}$   \\	
			$c_{m,16}$	&$\frac{5(\bar{u}_1-2\bar{u}_2+\bar{u}_3-\bar{u}_7+2\bar{u}_8-\bar{u}_9)}{16} +\frac{15(\bar{v}_1-\bar{v}_3-\bar{v}_7+\bar{v}_9)}{8}$   \\		
			$c_{m,17}$	&$\frac{5(\bar{u}_1-\bar{u}_3-2\bar{u}_4+2\bar{u}_6-\bar{u}_7-\bar{u}_9)}{38} +\frac{15(\bar{v}_1+\bar{v}_3-2\bar{v}_4-2\bar{v}_6+\bar{v}_7+\bar{v}_9)}{8}$   \\	
			$c_{m,18}$	&$\frac{5(\bar{u}_1-2\bar{u}_2+\bar{u}_3-\bar{u}_7+2\bar{u}_8-\bar{u}_9)}{38} +\frac{15(\bar{w}_1-2\bar{w}_2+\bar{w}_3+\bar{w}_7-2\bar{w}_8+\bar{w}_9)}{8}$   \\	
			$c_{m,19}$	&$\frac{5(\bar{u}_1-\bar{u}_3-2\bar{u}_4+2\bar{u}_6-\bar{u}_7-\bar{u}_9)}{16} +\frac{15(\bar{w}_1-\bar{w}_3-\bar{w}_7+\bar{w}_9)}{8}$   \\
			$c_{m,20}$	&$\frac{35(\bar{u}_2-\bar{u}_8)}{36} +\frac{77(\bar{w}_2+\bar{w}_8)+266\bar{w}_5}{18}$   \\		 
		\end{tabular}\label{secA:coe2D}	 
		\begin{tabular}{ccccccccccccccccc}
			\toprule
			{Coefficients} &$m=1$&& $m=2$&&$m=3$&&$m=4$&&$m=5$&& \cr
			\midrule	 
			$c_{m,0}$	&$\bar{u}_{5}$  &&$\bar{u}_{5}$  &&$\bar{u}_{5}$  &&$\bar{u}_{5}$ &&$\bar{u}_{5}$  \\	
			$c_{m,1}$	&$12\bar{v}_{{5}}$  &&${\bar{u}_5}-{\bar{u}_4}$  &&${\bar{u}_6}-{\bar{u}_5}$  &&${\bar{u}_5}-{\bar{u}_4}$ &&${\bar{u}_6}-{\bar{u}_5}$  \\	
			$c_{m,2}$	&$12\bar{w}_{{5}}$  &&${\bar{u}_5}-{\bar{u}_2}$  &&${\bar{u}_5}-{\bar{u}_2}$  &&${\bar{u}_8}-{\bar{u}_5}$ &&${\bar{u}_8}-{\bar{u}_5}$   \\	
			$c_{m,3}$	&$\frac{(\bar{u}_1+\bar{u}_3+\bar{u}_7+\bar{u}_9)-2(\bar{u}_2+\bar{u}_8)+3(\bar{u}_4-2\bar{u}_5+\bar{u}_6 )}{10}$  &  &  &   \\	
			$c_{m,4}$	&$\frac{\bar{u}_1+\bar{u}_3+\bar{u}_7+\bar{u}_9}{4}$  &  &  &   \\	
			$c_{m,5}$	&$\frac{(\bar{u}_1+\bar{u}_3+\bar{u}_7+\bar{u}_9)-2(\bar{u}_4+\bar{u}_6)+3(\bar{u}_2-2\bar{u}_5+\bar{u}_8 )}{10}$  &  &  &   \\	
			$c_{m,6}$	&$-\frac{5(\bar{u}_4-\bar{u}_6)}{11}-\frac{120\bar{v}_5}{11}$  &  &  &   \\	
			$c_{m,7}$	&$-\frac{\bar{u}_1+\bar{u}_3-\bar{u}_7-\bar{u}_9-2(\bar{u}_2-\bar{u}_8)}{4}$  &  &  &   \\	
			$c_{m,8}$	&$-\frac{\bar{u}_1-\bar{u}_3+\bar{u}_7-\bar{u}_9-2(\bar{u}_4-\bar{u}_6)}{4}$  &  &  &   \\	
			$c_{m,9}$	&$-\frac{5(\bar{u}_2-\bar{u}_8)}{11}-\frac{120\bar{w}_5}{11}$  &  &  &   \\	  	 
			\bottomrule
		\end{tabular}
	\end{threeparttable}	
\end{table} 	

\section{Constant matrices in \eqref{eq:2Djump} of Theorem \ref{thm:2dOE}}\label{sec:A3_Matrices}

The constant matrices $A^{(m)}$, $B^{(m)}$, and $C^{(m)}$ for $m=0,1$ are defined as 
 \begin{align*}
 	&A^{(0)}=\left[ \begin {array}{cccc} -{\frac{83711}{1721856}}&{\frac{213373}{
 			1721856}}&-{\frac{213373}{1721856}}&{\frac{83711}{1721856}}
 	\\ \noalign{\medskip}-{\frac{179273}{7748352}}&-{\frac{4144421}{
 			7748352}}&{\frac{4144421}{7748352}}&{\frac{179273}{7748352}}
 	\\ \noalign{\medskip}-{\frac{83711}{1721856}}&{\frac{213373}{1721856}}
 	&-{\frac{213373}{1721856}}&{\frac{83711}{1721856}}\end {array}
 	\right]
 	,\quad
 	A^{(1)}= \left[ \begin {array}{cccc} -{\frac{72361}{286976}}&{\frac{72361}{
 			286976}}&{\frac{72361}{286976}}&-{\frac{72361}{286976}}
 	\\ \noalign{\medskip}{\frac{471889}{1291392}}&-{\frac{471889}{1291392}
 	}&-{\frac{471889}{1291392}}&{\frac{471889}{1291392}}
 	\\ \noalign{\medskip}-{\frac{72361}{286976}}&{\frac{72361}{286976}}&{
 		\frac{72361}{286976}}&-{\frac{72361}{286976}}\end {array} \right],
 	\\
 	&B^{(0)}= \left[ \begin {array}{cccc} -{\frac{68215}{215232}}&{\frac{39895}{
 			215232}}&{\frac{39895}{215232}}&-{\frac{68215}{215232}}
 	\\ \noalign{\medskip}{\frac{165535}{968544}}&-{\frac{3677215}{968544}}
 	&-{\frac{3677215}{968544}}&{\frac{165535}{968544}}
 	\\ \noalign{\medskip}-{\frac{68215}{215232}}&{\frac{39895}{215232}}&{
 		\frac{39895}{215232}}&-{\frac{68215}{215232}}\end {array} \right]
 	,\quad
 	B^{(1)}=  \left[ \begin {array}{cccc} -{\frac{23771}{17936}}&-{\frac{21411}{
 			17936}}&{\frac{21411}{17936}}&{\frac{23771}{17936}}
 	\\ \noalign{\medskip}{\frac{164615}{80712}}&{\frac{7959}{8968}}&-{
 		\frac{7959}{8968}}&-{\frac{164615}{80712}}\\ \noalign{\medskip}-{\frac
 		{23771}{17936}}&-{\frac{21411}{17936}}&{\frac{21411}{17936}}&{\frac{
 			23771}{17936}}\end {array} \right],
 	\\
 	&C^{(0)}= \left[ \begin {array}{cccc} -{\frac{1541}{45312}}&{\frac{1541}{15104}
 	}&-{\frac{1541}{15104}}&{\frac{1541}{45312}}\\ \noalign{\medskip}0&0&0
 	&0\\ \noalign{\medskip}{\frac{1541}{45312}}&-{\frac{1541}{15104}}&{
 		\frac{1541}{15104}}&-{\frac{1541}{45312}}\end {array} \right] 
 	,\quad\quad\quad~
 	C^{(1)}=	 \left[ \begin {array}{cccc} {\frac{1665}{7552}}&-{\frac{1665}{7552}}&
 	-{\frac{1665}{7552}}&{\frac{1665}{7552}}\\ \noalign{\medskip}0&0&0&0
 	\\ \noalign{\medskip}-{\frac{1665}{7552}}&{\frac{1665}{7552}}&{\frac{
 			1665}{7552}}&-{\frac{1665}{7552}}\end {array} \right]
 	.
 \end{align*}

\end{appendix}

\bibliographystyle{abbrv}
\bibliography{Ref_OEHWENO}
 
\end{document}